\documentclass[11pt, reqno]{amsart}
\usepackage{enumitem,mathrsfs}
\makeatletter
\newcommand{\mylabel}[2]{#2\def\@currentlabel{#2}\label{#1}}
\makeatother
\usepackage{array}
\usepackage[english]{babel}
\usepackage[top=1in, bottom=1in, left=1in, right=1in]{geometry}
\usepackage{mathtools}

\expandafter\let\csname ver@amsthm.sty\endcsname\relax

\usepackage{tikz}

\usepackage{amsmath}
\usepackage{amssymb}
\usepackage{mathdots}
\usepackage{mathtools}

\usepackage{dsfont}
\usepackage{chngpage}

\usepackage{hyperref}
\usepackage{amsthm}
\usepackage[capitalize,noabbrev]{cleveref}

\allowdisplaybreaks

\numberwithin{equation}{section}

\newtheorem{thm}{Theorem}[section]
\newtheorem{lemma}[thm]{Lemma}
\newtheorem{cor}[thm]{Corollary}

\newtheorem{Example}[thm]{Example}

\newtheorem{Remark}[thm]{Remark}

\crefname{thm}{Theorem}{Theorems}
\crefname{lemma}{Lemma}{Lemmas}
\crefname{cor}{Corollary}{Corollaries}
\crefname{prop}{Proposition}{Propositions}

\crefname{example}{Example}{Examples}
\crefname{remark}{Remark}{Remarks}

\newcommand{\emailhref}[1]{\email{\href{#1}{#1}}}

\newcommand{\Z}{\mathbb{Z}}

\newcommand{\M}{\operatorname{M}}
\newcommand{\w}{\operatorname{w}}

\newcommand{\ka}{\operatorname{k}}
\newcommand{\h}{\operatorname{h}}
\newcommand{\al}{\alpha}
\newcommand{\be}{\beta}

\title[Propp's benzels and Lai's nearly symmetric hexagons with holes]{Propp's benzels and Lai's nearly symmetric\\ hexagons with holes}

\author[Seok Hyun Byun]{Seok Hyun Byun}\emailhref{sbyun@amherst.edu}
\address{Department of Mathematics, Amherst College, Amherst, Massachusetts 01002, U.S.A.}

\author[Mihai Ciucu]{Mihai Ciucu}\emailhref{mciucu@indiana.edu}
\address{Department of Mathematics, Indiana University, Bloomington, Indiana 47405, U.S.A.}
\author[Yi-Lin Lee]{Yi-Lin Lee}\emailhref{yillee@ntnu.edu.tw}
\address{Department of Mathematics, National Taiwan Normal University, Taipei 116059, Taiwan}
\thanks{S.H.B. was supported in part by the AMS-Simons Travel Grant.}
\thanks{M.C. was supported in part by Simons Foundation Collaboration Grant 710477.}

\begin{document}
\begin{abstract}
  In this paper we present a new version of the second author's factorization theorem for perfect matchings of symmetric graphs.
We then use our result to solve four open problems of Propp on the enumeration of trimer tilings on the hexagonal lattice.

As another application, we obtain a semi-factorization result for the number of lozenge tilings of a large class of hexagonal regions with holes (obtained by starting with an arbitrary symmetric hexagon with holes, and translating all the holes one unit lattice segment in the same direction). This in turn leads to the solution of two open problems posed by Lai and to an extension of a result due to Fulmek and Krattenthaler, which results in exact enumeration formulas for some new families of hexagonal regions with holes.

Our result also allows us to find new, simpler proofs (and in one case, a new, simpler form) of some formulas due to Krattenthaler for the number of perfect matchings of Aztec rectangles with unit holes along a lattice diagonal.


\end{abstract}

\maketitle

\section{Introduction}

The second author's factorization theorem for perfect matchings of symmetric graphs
\cite[Theorem 1.2]{CiucuMatchingFactorization} opens up the possibility of simple proofs for results stating that the number of perfect matchings of a given family of symmetric planar graphs is expressed by an explicit product formula (see e.g. \cite{ByunIntrusion}\cite{ByunGeneralIntrusion}\cite{CiucuPP1}\cite{CiucuAxialShamrock}\cite{CiucuPP2}\cite{FulmekKrattenthaler1}\cite{LaiAxialShamrock}).

Two particular classes of results of this kind concern honeycomb graphs with holes along a symmetry axis (for some illustrative examples, see \cite{CiucuPP1}), and Aztec rectangles with holes along the axis of symmetry (see e.g.\ \cite{CiucuMatchingFactorization}).
As it turns out, in both these cases interesting results hold also when the collection of holes is translated some distance away from the symmetry axis. Indeed, Krattenthaler proved such formulas for Aztec rectangles with unit holes (see \cite{KrattenthalerAztecRectangle}). Also, Lai conjectured (see \cite{LaiOpenProblems}) that simple product formulas exist for the number of lozenge tilings of two families of regions, both obtained from symmetric hexagons with certain types of holes along the symmetry axis, by translating all the holes one unit lattice segment in the same direction\footnote{By identifying lozenge tilings of a region on the triangular lattice with perfect matchings of its planar dual, this is equivalent to enumerating perfect matchings of a honeycomb graph with holes.}: one family corresponds to the holes being an arbitrary collection of like-oriented triangles of side-length two along the symmetry axis (the untranslated version of this is one of the families of regions for which a simple product formula is proved in \cite{CiucuPP1}), while in the other there is a single hole in the shape of a shamrock, symmetric about the symmetry axis (the original, untranslated case is treated in \cite{CiucuAxialShamrock} and \cite{LaiAxialShamrock}, where a simple product formula is given for it).
However, in both these situations the factorization theorem cannot be applied directly,
because the dual graphs of the involved regions are not quite symmetric.

In this paper we present a new version (see Theorem \ref{tba}) of \cite[Theorem 1.2]{CiucuMatchingFactorization}, which allows us to handle situations like the ones described above.
In fact, it turns out that this leads to a quite general semi-factorization result\footnote{As opposed to a factorization result, which expresses the quantity of interest as a product of two quantities of the same kind, pertaining to two ``halves'' of the original graph, a semi-factorization result expresses the quantity of interest as a {\it sum} of two such products.} for the number of lozenge tilings of a large class of hexagonal regions with holes, obtained by starting with an arbitrary symmetric hexagon with holes, and then translating all the holes one unit lattice segment in the same direction (see Theorem \ref{tbb}). 

In Section 3 we use Theorem \ref{tba} to solve four open problems of Propp (Problems 8 through 11 in \cite{propp2022trimer}) on the enumeration of trimer tilings\footnote{ More precisely, these are tilings by triples of contiguous unit hexagons, which in the dual picture become trimer coverings. For ease or reference we call them simply trimer tilings.} on the hexagonal lattice.

Section 4 shows how Theorem \ref{tbb} leads to the solution of the first open problems of Lai mentioned above. In Section 5 we discuss the second open problem of Lai (as well as some related problems), and then describe how to obtain explicit product formulas for the number of lozenge tilings of the regions appearing in those problems. In Section 6 we make further use of Theorem \ref{tbb} to obtain an extension of a result due to Fulmek and Krattenthaler on the number of lozenge tilings of a symmetric hexagon containing a fixed lozenge just off the symmetry axis.

Theorem \ref{tba} also allows us to find new, simpler proofs (and in one case, a simpler form) of some formulas due to Krattenthaler for the number of perfect matchings of Aztec rectangles with unit holes along a lattice diagonal. We present this in Section 7.

\section{Two general results}

All the graphs in this section will be assumed to be bipartite. The two classes of vertices will be called white and black. 

A {\it perfect matching} of a graph is a collection of vertex-disjoint edges
that are collectively incident to all vertices. Given a graph $G$, we denote by $\M(G)$ the number of its perfect matchings. If $v$ is a vertex of $G$, a {\it $v$-near-perfect matching} (or $v$-matching, for short) of $G$ is a perfect matching of $G\setminus v$. We denote the number of $v$-matchings of $G$ by $\M_v(G)$. If $v$ is not specified, we call a collection of disjoint edges of $G$ that covers all but a single vertex of $G$ a {\it near-perfect matching}.

Let $G$ be a plane graph. We say that $G$ is {\it symmetric} if it is invariant
under the reflection across some straight line. The picture on the left in Figure \ref{fba}  shows an example of
a symmetric graph. Clearly, a symmetric graph has 
no perfect (resp., near-perfect) matching unless the axis of symmetry contains an even (resp., odd) number of vertices (otherwise, the total number of vertices of $G$ has the wrong parity). 

A {\it weighted} symmetric graph is a symmetric graph equipped with a weight function on the edges that is constant on the orbits of the reflection. The {\it width} of a symmetric graph~$G$, denoted $\w(G)$, is defined to be the integer part of half the number of vertices of $G$ lying on the symmetry axis. 

\begin{figure}[t]
\vskip0.2in
\centerline{
\includegraphics[width=1\textwidth]{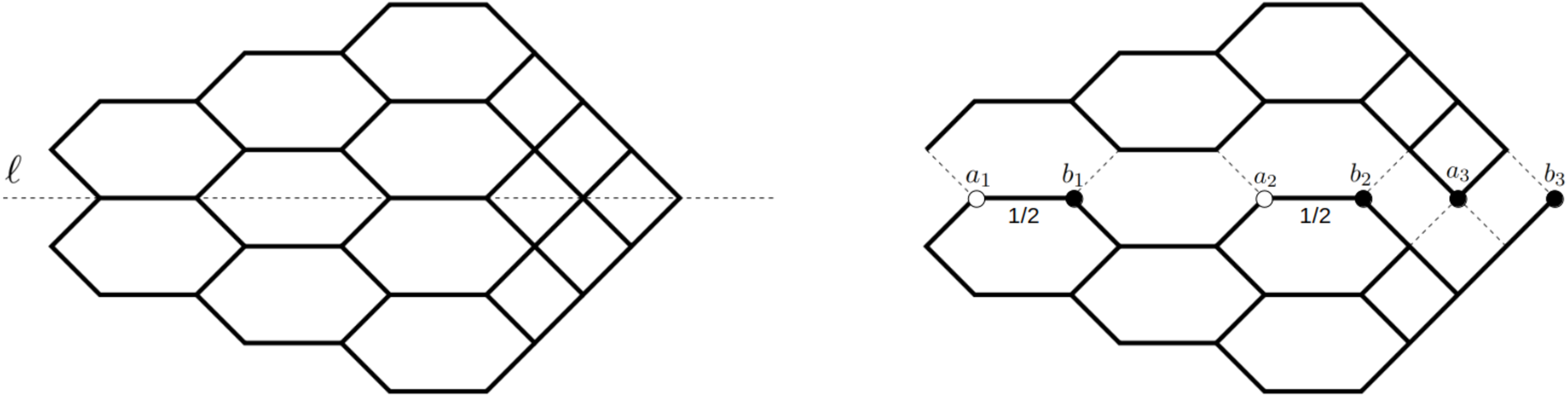}
}
\caption{\label{fba}  Illustrating Theorem \ref{tba}(a): A plane bipartite symmetric graph $G$ (left), and the resulting graphs $G^+$ and $G^-$ (right).}
\end{figure}

\begin{figure}[t]
\centerline{
\hfill
{\includegraphics[width=0.20\textwidth]{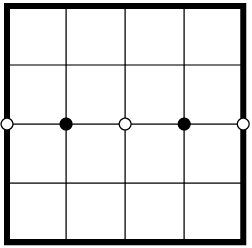}}
\hfill
{\includegraphics[width=0.20\textwidth]{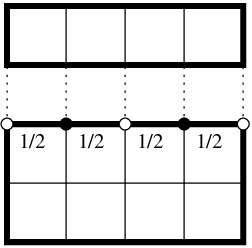}}
\hfill
}
\vskip0.4in
\centerline{
\hfill
{\includegraphics[width=0.20\textwidth]{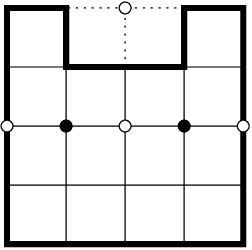}}
\hfill
{\includegraphics[width=0.20\textwidth]{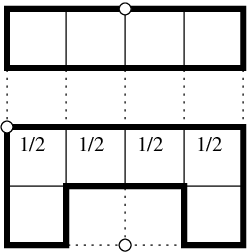}}
\hfill
}
\caption{\label{fbxa} Illustrating the case of Theorem \ref{tba}(b) when the removed boundary vertex $v$ is white. A plane, bipartite symmetric graph $G$ with an odd number of vertices on the symmetry axis, with the number of white vertices being one more than the number of black vertices (top left); the corresponding graphs $G^+$ and $G^-$ (top right). When a white vertex $v$ is removed from the boundary (bottom left), the factorization involves the graphs $G^+$ and $G^-\setminus v'$, where $v'$ is the mirror image of $v$ (bottom right).
}
\vskip-0.1in
\end{figure}

\begin{figure}[t]
\centerline{
\hfill
{\includegraphics[width=0.20\textwidth]{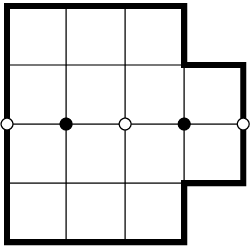}}
\hfill
{\includegraphics[width=0.20\textwidth]{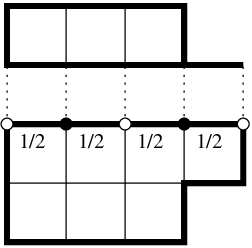}}
\hfill
}
\vskip0.4in
\centerline{
\hfill
{\includegraphics[width=0.20\textwidth]{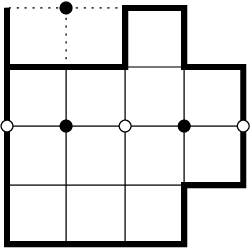}}
\hfill
{\includegraphics[width=0.20\textwidth]{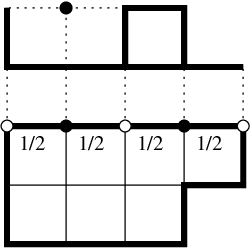}}
\hfill
}
\caption{\label{fbxb} Illustrating the case of Theorem \ref{tba}(b) when the removed boundary vertex $v$ is black. A plane, bipartite symmetric graph $G$ with an odd number of vertices on the symmetry axis, with one more black than white vertices (top left); the corresponding graphs $G^+$ and $G^-$ (top right). When a black vertex $v$ is removed from the boundary (bottom left), the factorization involves the graphs $G^+\setminus v$ and $G^-$ (bottom right).
}
\vskip-0.1in
\end{figure}

Let $G$ be a weighted symmetric graph with symmetry axis $\ell$, which we consider to be horizontal. Let $a_{1},b_{1},a_{2},b_{2},\dotsc$ be the vertices lying on $\ell$, as they occur from left to right (if $G$ has an even number of vertices, this sequence ends with $b_{\w(G)}$, while if $G$ has an odd number of vertices, it ends with $a_{\w(G)+1}$).

The weight of a matching $\mu$ is defined to be the product of the weights of the edges contained in $\mu$. The matching generating function of a weighted graph $G$, also denoted\footnote{ When all weights are 1, this becomes just the number of perfect matchings of $G$.} by $\M(G)$, is the sum of the weights of all perfect matchings\footnote{ For weighted graphs, $\M_v(G)$ denotes the sum of the weights of all $v$-near-perfect matchings of $G$.} of $G$. The matching generating function is clearly multiplicative with respect to disjoint unions of graphs. We will henceforth assume that all graphs under consideration are connected. 

Let $G$ be a weighted symmetric graph that is also bipartite. For definiteness, choose the leftmost vertex on the symmetry axis $\ell$ to be white. We also assume that $G$ has at least one vertex on the symmetry axis $\ell$ (note that this implies that the symmetry is color-preserving; see the fifth paragraph of the proof of Theorem \ref{tba}).
We define two subgraphs $G^{+}$ and $G^{-}$ as follows.

Given a vertex $u$ of $G$ on the symmetry axis, we call the operation of deleting all edges above~$\ell$ which are incident to $u$ {\it cutting above $u$}; similarly, we call deleting all edges below $\ell$ incident to $u$ {\it cutting below~$u$}.
Perform cutting operations above all white $a_{i}$'s and black $b_{i}$'s and below all black $a_{i}$'s and white~$b_{i}$'s. Note that this procedure yields cuts of the same kind at the endpoints of each edge lying on $\ell$ (see the fourth paragraph in the proof of Theorem \ref{tba} for a justification). Reduce the weight of each such edge by half; leave all other weights unchanged.

As shown in the proof of Theorem \ref{tba}(a), the graph $G_0$ produced by the above procedure is disconnected into one part lying above and one lying below $\ell$.
Denote the portion above $\ell$ by $G^{+}$, and the one below $\ell$ by $G^{-}$; see the picture on the right in Figure \ref{fba} for an illustration of this procedure (the edges whose weight has been reduced by half are marked by 1/2). 

Part (a) of the following result is a slight strengthening of the original factorization theorem of \cite[Theorem 1.2]{CiucuMatchingFactorization} (namely, we show that the assumption made there that the graph is separated by its symmetry axis can be dropped). The new version is contained in part (b); its two cases are illustrated in Figures \ref{fbxa} and \ref{fbxb}.

\begin{thm}
\label{tba}
Let $G$ be a plane bipartite weighted symmetric graph. 

\parindent0pt
$($a$)$. If $G$ has an even number of vertices, then
\begin{equation}
\M(G)=2^{\w(G)}\M(G^{+})\M(G^{-}).
\label{eba}
\end{equation}
$($b$)$. Suppose $G$ has an odd number of vertices and $v$ is a vertex on the unbounded face of $G$ lying off the symmetry axis. Assume $($without loss of generality, as $G$ is symmetric$)$ that $v$ is above the symmetry axis.
Then
\begin{equation}
\M_v(G)=
\begin{cases}
2^{\w(G)}\M(G^{+})\M_{v'}(G^{-}), & \text{\rm if $v$ is white,}\\
2^{\w(G)}\M_v(G^{+})\M(G^{-}), & \text{\rm if $v$ is black,}
\end{cases}
\label{ebb}
\end{equation}
where $v'$ is the mirror image of $v$.

\end{thm}

\parindent12pt

\begin{proof} (a) This part is almost exactly the same as the second author's factorization theorem \cite[Theorem 1.2]{CiucuMatchingFactorization}): the only difference is that one of the assumptions in the latter --- namely, that the graph is separated by its symmetry axis (the meaning of this will be recalled below) --- is not present in the statement of the current theorem. We show that this condition is in fact a consequence of the other assumptions (a fact overlooked at the time of writing of \cite{CiucuMatchingFactorization}). Then part (a) will follow directly from \cite[Theorem~1.2]{CiucuMatchingFactorization}.

The needed separation condition is that the graph $G_0$ produced by the cutting procedure described before the statement of the theorem is disconnected into one part lying above and one lying below $\ell$. We show below that this is a consequence of the current assumptions.


What we need to show is that
there exists a Jordan arc $J$ in the plane connecting some point $A$ on $\ell$ to some point $B$ on $\ell$ --- where $A$ is to the left of $a_1$, and $B$ is to the right of $b_{\w(G)}$ --- so that the portion $H$ of $G_0$ which is above or on $\ell$ is above $J$, while the portion $K$ of $G_0$ which is below or on $\ell$ is below~$J$.

  To see this, note first that, as mentioned above, the cuts at the endpoints of each edge lying on $\ell$ are of the same kind --- both above, or both below $\ell$; this is because we cut above white $a_i$'s and black $b_i$'s (and below black $a_i$'s and white $b_i$'s), and if an edge lies along $\ell$, its endpoints have on the one hand opposite color, and on the other opposite $a_i$- or $b_i$-type.

  Secondly, note that $G_0$ cannot have any edge $\{u,v\}$ with $u$ above $\ell$ and $v$ below $\ell$. Indeed, if $u$ and $v$ are not mirror images with respect to $\ell$, the symmetry of $G$ implies that the mirror images $u'$ and $v'$ of $u$ and $v$ form another edge $\{u',v'\}$ of $G_0$; but then the edges $\{u,v\}$ and $\{u',v'\}$ would cross, which is a contradiction, since $G$ is a plane graph (see the picture on the left in Figure \ref{fbb}). On the other hand, if $u$ and $v$ are mirror images of each other, let $P$ be a path in $G$ connecting $u$ to $a_1$ (as $G$ is assumed to be connected, such a path exists). Then the mirror image $P'$ of $P$ across $\ell$ is a path in $G$ connecting $v$ to $a_1$. Since $P$ and $P'$ have the same length and $G$ is bipartite, it follows that $u$ and $v$ have the same color, so the edge $\{u,v\}$ has endpoints of the same color; this contradicts the fact that $G$ is bipartite.

  The connected components of $G_0\cap\ell$ are either single vertices or paths with at least one edge. By the definition of $G_0$, all edges incident to a single-vertex component are confined to one side of $\ell$. Furthermore, by the second to last paragraph, all edges of $G_0$ incident to the vertices of a path-component are also contained in one of the two half-planes bounding $\ell$. Therefore we can connect $A$ to $B$ by a Jordan arc $J$ that stays close to $\ell$ and meanders between the connected components of $G_0\cap\ell$ so that it does not cross any edge of $G_0$ incident to a vertex on $\ell$.
  In fact, this arc $J$ does not cross any other edge of $G_0$ either. Indeed, as $J$ can stay arbitrarily close to $\ell$, such an edge of $G_0$ would need to have its endpoints on opposite sides of $\ell$, which we showed in the previous paragraph that is impossible. This proves our claim, and completes the proof of part (a).

\begin{figure}[t]
\vskip0.4in
\centerline{
\hfill
{\includegraphics[width=0.25\textwidth]{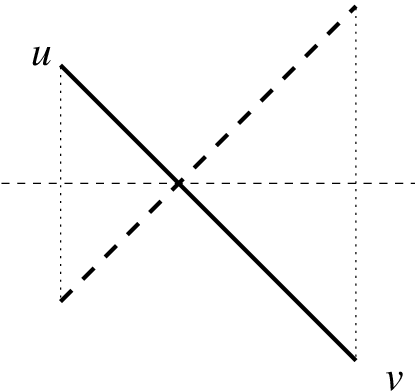}}
\hfill
{\includegraphics[width=0.55\textwidth]{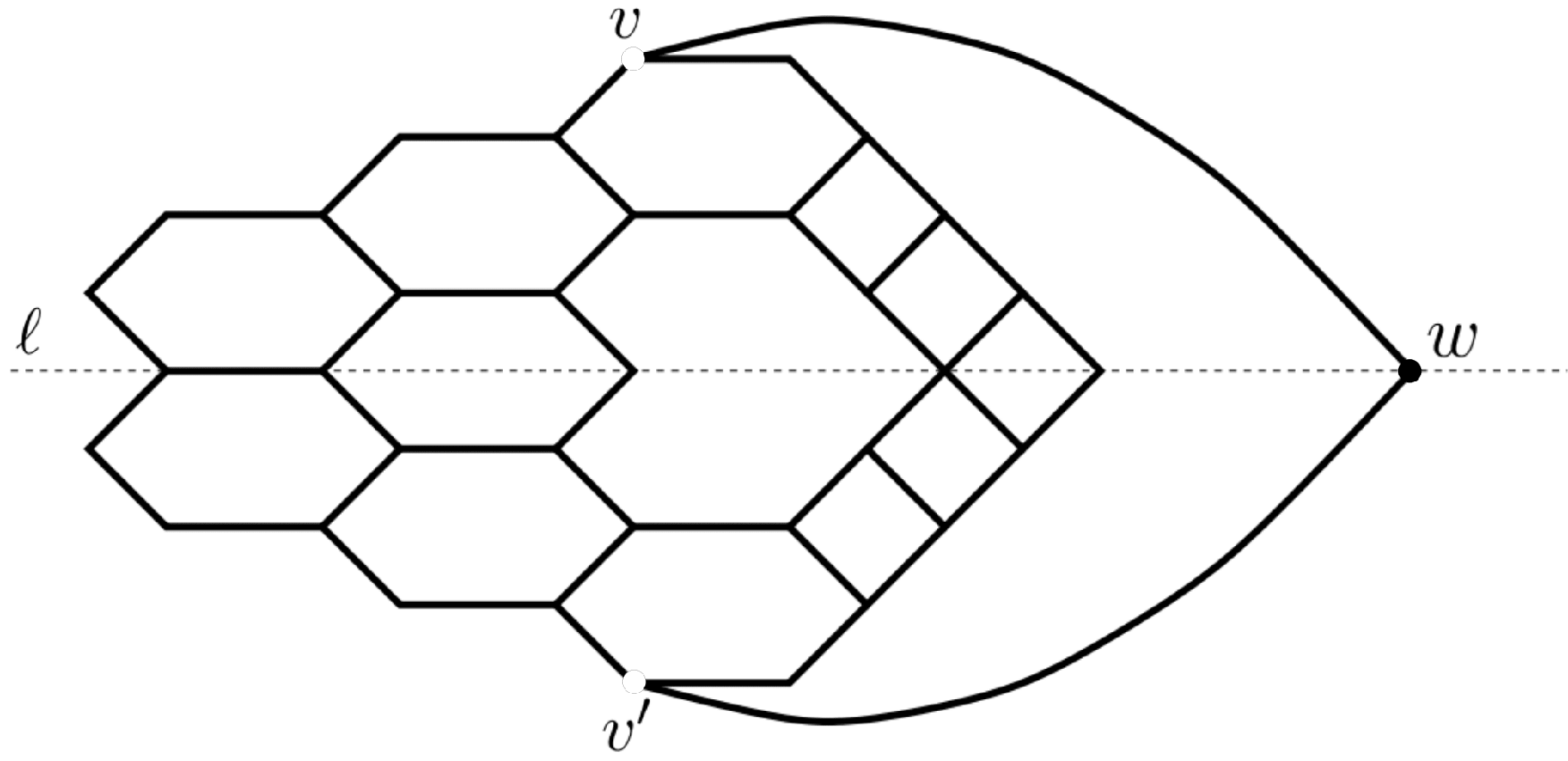}}
\hfill
}
\caption{\label{fbb} {\it Left.} Illustrating the argument in the fifth paragraph of the proof of Theorem \ref{tba}. {\it Right.} The graph $\overline{G}$ corresponding to the graph $G$ shown in Figure \ref{fba} with one vertex removed (since the graph in Figure \ref{fba} has the same number of white and black vertices, we have to delete from it a vertex on $\ell$ of the color opposite to the color of $v$ in order to obtain a graph meeting the hypotheses of part (b) of Theorem~\ref{tba}).
}
\vskip-0.1in
\end{figure}


(b)  To prove the second part, note first that unless the difference between the number of white and black vertices of $G$ is $\pm1$, and in addition $v$ has the color of the majority of vertices, equation \eqref{ebb} holds trivially, as both sides equal zero. Therefore, without loss of generality we may assume that these conditions hold.

  Next, define $\overline{G}$ to be the graph obtained from $G$ by including a new vertex $w$ on $\ell$ contained in the unbounded face of $G$, and the two new edges $\{w,v\}$ and $\{w,v'\}$ (both weighted by 1), where $v'$ is the mirror image of $v$ (see the picture on the right in Figure \ref{fbb}). Then $\overline{G}$ is a plane, bipartite, weighted symmetric graph, to which
part (a) of the theorem can be applied.

If $v$ is white, it follows that $w$ is black, and (since $w$ is also a $b_i$-type vertex) the cutting operation at $w$ prescribed by the factorization theorem occurs above $w$. Therefore, $\overline{G}^+$ is precisely $G^+$. Furthermore, in $\overline{G}^-$ vertex $w$ has degree one, and after removing the edge $\{u,v'\}$ that is forced to match it in every perfect matching, the resulting graph (which has the same matching generating function as $\overline{G}^-$, as the removed edge has weight 1) is precisely $G^-\setminus v'$. Thus part (a) of the theorem applied to $\overline{G}$ yields
\begin{equation}
  \M(\overline{G})=2^{\w(G)+1}\M(G^+)\M(G^-\setminus v')
  =2^{\w(G)+1}\M(G^+)\M_{v'}(G^-).
\label{ebc}
\end{equation}  
On the other hand, since in $\overline{G}$ the vertex $w$ can only be matched to $v$ or $v'$, and since $G\setminus v$ and $G\setminus v'$ are isomorphic, it follows that $\M(\overline{G})=2\M(G\setminus v)=2\M_v(G)$. Combined with \eqref{ebc}, this proves \eqref{ebb} when $v$ is white. The case when $v$ is black follows by the same argument.
\end{proof}

The second result in this section is more specific, but is general enough to have several interesting consequences, including solutions to two open problems posed by Lai (see Sections 4 and~5), closed formulas for the number of lozenge tilings of some new families of hexagonal regions with holes (see Section 5), and an extension of a result by Fulmek and Krattenthaler (see Section~6).

Draw the triangular lattice so that one family of lattice lines is horizontal, and let $H$ be a non-degenerate (i.e., with all sides of positive length) hexagonal region on this lattice, symmetric about a vertical symmetry axis\footnote{ We will be interested in symmetric (or nearly-symmetric) regions obtained from $H$ by making in it some holes, and since the symmetry of such regions seems to be more readily appreciated visually when the symmetry axis is vertical, we adopt this point of view for the second result in this section.} $\ell$. Let $R_0$ be the region obtained from $H$ by removing an arbitrary collection ${\mathcal C}$ of unit triangles, so that ${\mathcal C}$ is symmetric about $\ell$ and contains no unit triangles touching the southern or southeastern edge of $H$. Define $R$ to be the region obtained from $R_0$ by translating all unit triangles in ${\mathcal C}$ one unit in the southeastern lattice direction. We call the region $R$ a {\it nearly symmetric hexagon with holes} (an example is shown in the picture on the left in Figure \ref{fbc}). Therefore, while the outer, hexagonal boundary of the region $R$ is symmetric with respect to the original symmetry axis $\ell$, the holes in $R$ are symmetric about the line $\ell'$, the translation of $\ell$ half a unit to the right.

Let $G$ be the planar dual graph of $R$. Even though $G$ is not symmetric about $\ell'$, let us consider the graphs $G^+$ and $G^-$ obtained from it by rotating it 90 degrees counterclockwise and performing cutting operations around the vertices on $\ell'$ as prescribed by the factorization theorem (assuming, in particular, that after the rotation the leftmost vertex on $\ell'$ is white). In addition to these two graphs, consider also the graphs $\widehat{G}^+$ and $\widehat{G}^-$ obtained from $G$ by performing cutting operations around the vertices on $\ell'$ as prescribed by the factorization theorem, {\it but assuming this time that after the rotation the leftmost vertex on $\ell'$ is black}. Define $R^+$, $R^-$, $\widehat{R}^+$ and $\widehat{R}^-$ to be the lattice regions whose planar duals are the graphs $G^+$, $G^-$, $\widehat{G}^+$ and $\widehat{G}^-$, respectively\footnote{ If $R$ is the region on the triangular lattice whose planar dual is the graph $G$, and $G$ has weights on its edges, the lozenge position in $R$ corresponding to any given edge of $G$ comes weighted with the weight of that edge.}.

We can now state our second result in this section.

\begin{thm}
\label{tbb}
Let $R$ be a nearly symmetric hexagon with holes, and assume that the symmetric region $R_0$ from which $R$ was constructed\footnote{ I.e.\ the region obtained from $R$ by translating all the holes one unit lattice segment in the northwest direction.} admits at least one lozenge tiling. Then the number of lozenge tilings\footnote{ We denote the matching generating function of the lozenge tilings of a region $R$ (in which lozenge positions may carry weights) by $\M(R)$, because these lozenge tilings can be identified with perfect matchings of the (weighted) planar dual graph of the region $R$.} of $R$ is given by
\begin{equation}
\M(R)=2^{\w_{\ell'}(R)-1}\left(\M(R^+)\M(R^-)+\M(\widehat{R}^+)\M(\widehat{R}^-)\right),
\label{ebd}
\end{equation}
where $\w_{\ell'}(R)$ is half the number of unit triangles in $R$ crossed by $\ell'$.

\end{thm}

{\it Remark $1$.} Formula \eqref{ebd} has the following amusing interpretation. The nearly symmetric hexagon with holes $R$ is not quite symmetric about $\ell'$, but its holes are, so we might expect that if we form the expression on the right hand side of the factorization theorem (see equation~\eqref{eba}), the resulting number will be a reasonable estimate for $\M(R)$. But note: there are two ways to form this expression, depending on whether we consider the vertex $a_1$ in the factorization theorem to be white or black! (For symmetric graphs, these two ways are equivalent, because they result in isomorphic pairs of graphs, due to the symmetry.) Form then both these expressions, and average them, to get an even more reasonable estimate. Then Theorem \ref{tbb} states that this latter ``estimate'' is actually the exact value of $\M(R)$.

This phenomenon turns out to be very sensitive to changing the shape of the outer boundary of the region. For it to hold, it seems essential that the outer boundary is a symmetric hexagon, with symmetry axis half a unit away from the symmetry axis of the holes.

Our proof of the above theorem employs Kuo's graphical condensation \cite{kuo2004applications}. Let $G$ be a plane bipartite graph, $V_{1}$ and $V_{2}$ the vertex sets of the graph consisting of the two color classes, and $E$ the edge set of the graph. For any subset $W\subset V_{1} \cup V_{2}$ of vertices of $G$, let $G\setminus W$ be the subgraph obtained from $G$ by deleting all vertices in $W$ and all their incident edges.

\begin{thm}[Theorem 5.2 in \cite{kuo2004applications}]
\label{tbc}
Let $G=(V_1,V_2,E)$ be a weighted plane bipartite graph in which $|V_1|=|V_2|$. Let vertices $a,b,c,$ and $d$ appear in a cyclic order on the same face of $G$. If $a,b\in V_1$ and $c,d\in V_2$, then
\begin{equation*}
    \M(G\setminus\{a,d\})\M(G\setminus\{b,c\})=\M(G)\M(G\setminus\{a,b,c,d\})+\M(G\setminus\{a,c\})\M(G\setminus\{b,d\}).
\end{equation*}
\end{thm}

\begin{proof}[Proof of Theorem $\ref{tbb}$]
The details are slightly different depending on the parity of the length of the top and the bottom side of the region $R$ (there are $2\times2=4$ cases). We present here the detailed arguments for the case when the top side is odd and the bottom side even; the other cases follow similarly.

\begin{figure}
\vskip0.2in
\centerline{
\includegraphics[width=0.75\textwidth]{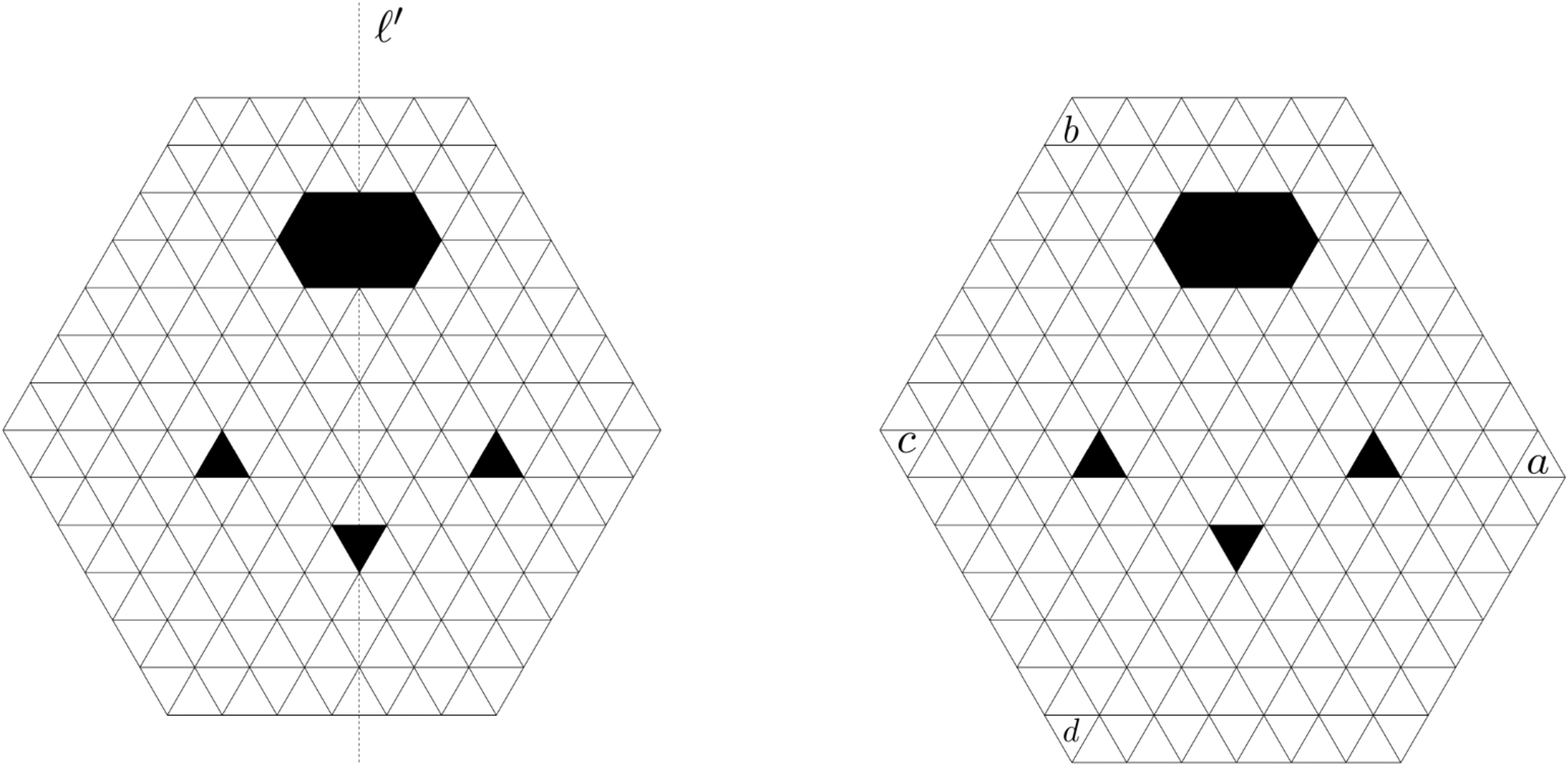}
}
\caption{\label{fbc} A nearly symmetric hexagon with holes $R$ (left), and the extended region $\widehat{R}$ (right). The positions of four unit triangles $a,b,c,$ and $d$ are marked on $\widehat{R}$.}
\end{figure}

\begin{figure}
\vskip0.2in
\centerline{
\includegraphics[width=1\textwidth]{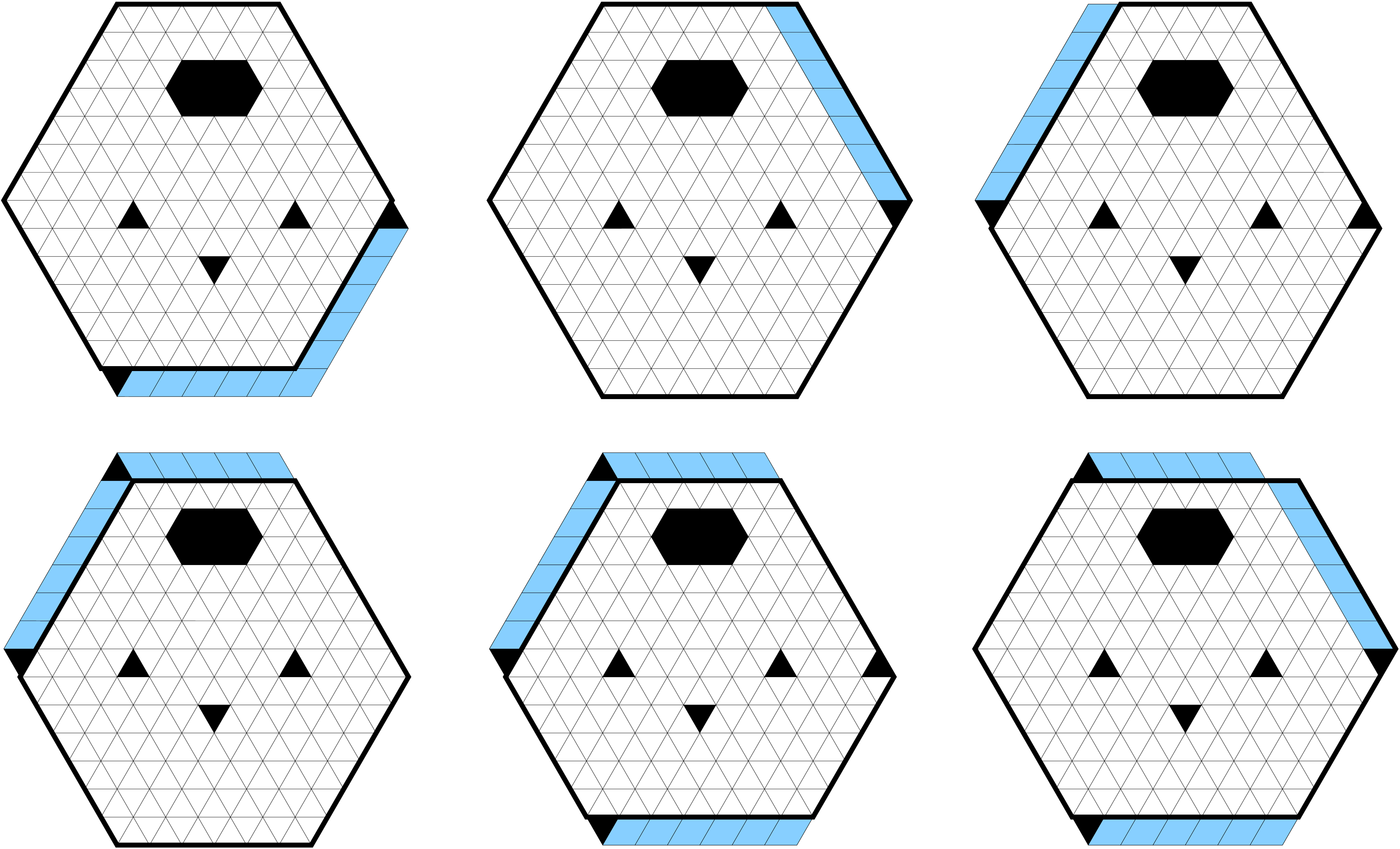}
}
\caption{\label{fbd} The six regions appearing in the recurrence obtained by applying Kuo condensation to the extended region $\widehat{R}$, with the indicated choice of the unit triangles $a,b,c,$ and $d$.}
\end{figure}

We first add one layer of unit triangles on the bottom and bottom right side of $R$, as illustrated in Figure \ref{fbc}; denote the resulting region by $\hat{R}$, and call it {\it the extended region}. Choose the unit triangles $a,b,c$ and $d$ along the boundary of $\hat{R}$ as indicated in the  picture on the right in Figure~\ref{fbc} (the assumption that the outer boundary of $R$ is a non-degenerate hexagon guarantees that $a$, $b$, $c$ and $d$ can be chosen as indicated). The dual graph of $\hat{R}$ and the four vertices of it that correspond to $a,b,c,$ and $d$ satisfy the conditions in Kuo's graphical condensation theorem.
Applying the latter
and taking the duals of the resulting graphs, one gets the following recurrence relation (the six regions appearing in the recurrence are illustrated in Figure \ref{fbd}):
\begin{equation}
    \M(\hat{R}\setminus\{a,d\})\M(\hat{R}\setminus\{b,c\})=\M(\hat{R})\M(\hat{R}\setminus\{a,b,c,d\})+\M(\hat{R}\setminus\{a,c\})\M(\hat{R}\setminus\{b,d\}).
\label{Kuorec}
\end{equation}
Note that after removing some forced lozenges, $\hat{R}\setminus\{a,d\}$ is precisely the region~$R$ whose tilings we want to enumerate (this is in fact what guided our construction of the extended region $\hat{R}$), and $\hat{R}\setminus\{b,c\}$ is a symmetric region (see the bottom left picture in Figure \ref{fbd}). Furthermore, each of the two regions $\hat{R}\setminus\{a,b,c,d\}$ and $\hat{R}\setminus\{a,c\}$ can be viewed as being obtained from a symmetric region with a unit dent on the boundary, by removing some forced lozenges (see the bottom center picture and top right picture, respectively, in Figure \ref{fbd}). The remaining two regions $\hat{R}$ and $\hat{R}\setminus\{b,d\}$ are also equivalent to symmetric regions with a unit defect on the boundary, but we first need to enlarge them to see this. In the case of $\hat{R}$, we add one layer of unit triangles on the top right side of $\hat{R}$, and remove the bottom-most of these (see the top center picture in Figure \ref{fbd}). The region $\hat{R}\setminus\{b,d\}$ can be handled similarly, but we first remove from it the forced lozenges along the top and bottom sides
(see the bottom right picture in Figure \ref{fbd}).


Therefore, viewing the regions in \eqref{Kuorec} as described above, we can apply Theorem \ref{tba}(a) to $\hat{R}\setminus\{b,c\}$, and Theorem \ref{tba}(b) to the four regions $\hat{R}$, $\hat{R}\setminus\{a,b,c,d\}$, $\hat{R}\setminus\{a,c\}$, and $\hat{R}\setminus\{b,d\}$. When applying\footnote{Strictly speaking, we apply Theorem \ref{tba} to their planar dual graphs, and then consider the regions whose planar duals are the resulting graphs.}
Theorem \ref{tba}, the powers of two in the factorizations of $\M(\hat{R}\setminus\{b,c\})$, $\M(\hat{R})$, and $\M(\hat{R}\setminus\{a,c\})$ are $2^{w_{l'}(R)}$, while those in the factorizations of $\M(\hat{R}\setminus\{a,b,c,d\})$ and $\M(\hat{R}\setminus\{b,d\})$ are $2^{w_{l'}(R)-1}$. Figure \ref{fbe} shows how the five regions are split via Theorem \ref{tba} (the shaded ellipses indicate lozenge positions weighted by $1/2$); note that the first branch of Theorem \ref{tba}(b) is applied when the unit boundary dent (which corresponds to $v$) has the same orientation as the top unit triangle along the symmetry axis (which corresponds to $a_1$), and the second branch when the orientations are opposite. One readily sees that the resulting ``half''-subregions --- we call them {\it parts} --- have the following properties:

\begin{figure}
\vskip0.2in
\centerline{
\includegraphics[width=1\textwidth]{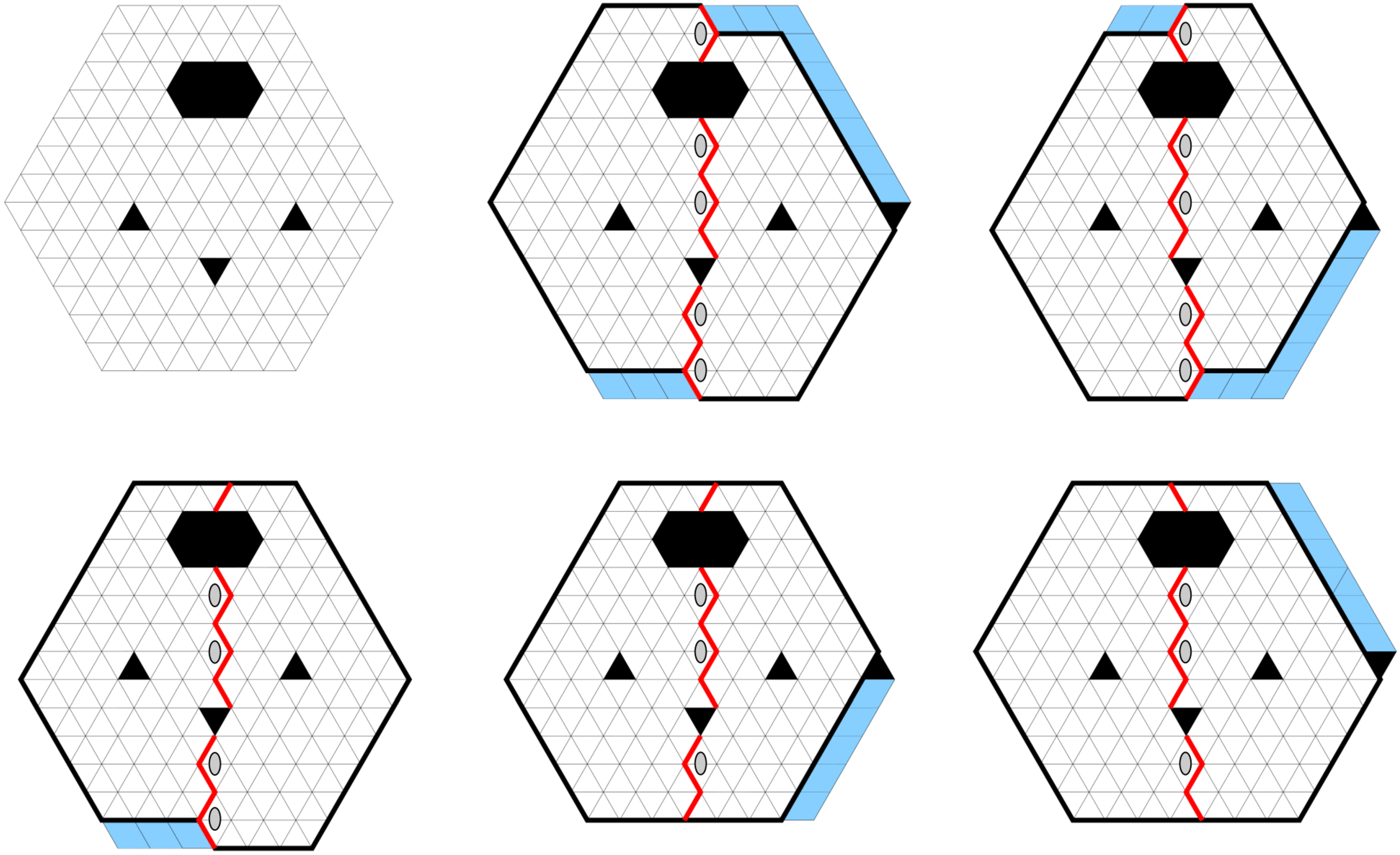}
}
\caption{\label{fbe} Applying Theorem 2.1 (a) and (b) to the five regions.}
\end{figure}

\begin{figure}
\vskip0.2in
\centerline{
\includegraphics[width=0.75\textwidth]{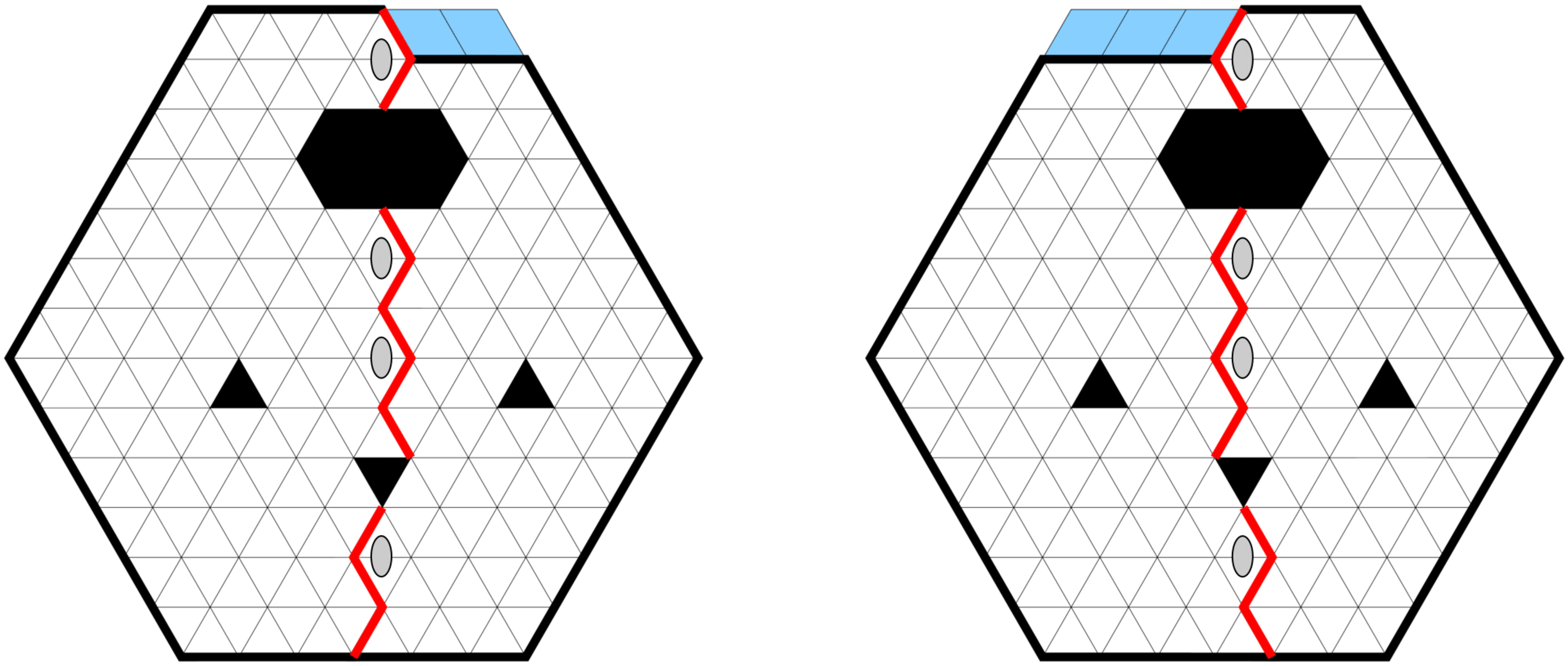}
}
\caption{\label{fbf} The two different ways of decomposing $R$ into parts,
 using zigzag lines along $\ell'$ as prescribed by the factorization theorem.}
\end{figure}

\medskip
$(i)$ The left part of $\hat{R}\setminus\{b,c\}$ is the same (up to forced lozenges) as the left part of $\hat{R}\setminus\{a,b,c,d\}$, and the right part of $\hat{R}\setminus\{b,c\}$ is the same (up to some forced lozenges) as the right part of~$\hat{R}$.

\vskip0.05in
$(ii)$ The left part of $\hat{R}\setminus\{b,c\}$ is the same as the right part of~$\hat{R}\setminus\{b,d\}$, and the right part of $\hat{R}\setminus\{b,c\}$ is the same (up to forced lozenges) as the left part of $\hat{R}\setminus\{a,c\}$.

\vskip0.05in
$(iii)$ The remaining two parts of $\hat{R}$ and $\hat{R}\setminus\{a,b,c,d\}$ (the ones which are not the same as any of the parts of $\hat{R}\setminus\{b,c\}$) can be identified with the two parts of $R$ obtained when we split~$R$ using one of the two zigzag lines along $\ell'$, as prescribed by the factorization theorem (compare the two picture at the center in Figure \ref{fbe} and the left picture in Figure \ref{fbf}).

\vskip0.05in
$(iv)$ The remaining two parts of $\hat{R}\setminus\{a,c\}$ and $\hat{R}\setminus\{b,d\}$ (the ones which cannot be identified with any part of $\hat{R}\setminus\{b,c\}$), can be identified (up to forced lozenges) with the two parts of $R$ when we split $R$ using the other zigzag line along $\ell'$, as prescribed by the factorization theorem (compare the two pictures on the right in Figure \ref{fbe} and the right picture in Figure~\ref{fbf}).


By applying Theorem \ref{tba} to each of the five regions in \eqref{Kuorec} besides $\hat{R}\setminus\{a,d\}$ (which, as we pointed out, after removing the forced lozenges becomes precisely the region $R$ whose tilings we want to enumerate), equation \eqref{Kuorec} becomes an equality in which the left hand side is the product of three factors, while the right hand side is the sum of two products, each having four factors.

By observations $(i)$ and $(ii)$ above, both factors on the left hand side besides $\M(\hat{R}\setminus\{a,d\})$ also appear in each of the two terms on the right hand side. By assumption, $\hat{R}\setminus\{b,c\}$ (which is the same as the region $R_0$) is tileable; hence, by Theorem \ref{tba}, both these factors are non-zero. Dividing by them, the left hand side of the equation discussed in the previous paragraph becomes just $\M(R)$. Furthermore, by $(iii)$, the two remaining factors in the first term on the right hand side are precisely $\M(R^+)$ and $\M(R^-)$. Finally, by $(iv)$, the two remaining factors in the second term on the right hand side are precisely $\M(\widehat{R}^+)$ and $\M(\widehat{R}^-)$. Thus, the equation at which we have arrived is precisely \eqref{ebd}, and the proof is complete.
\end{proof}

\section{Solution of Four Open Problems of Propp}

\begin{figure}
\vskip0.2in
\centerline{
\hfill
{\includegraphics[width=0.4\textwidth]{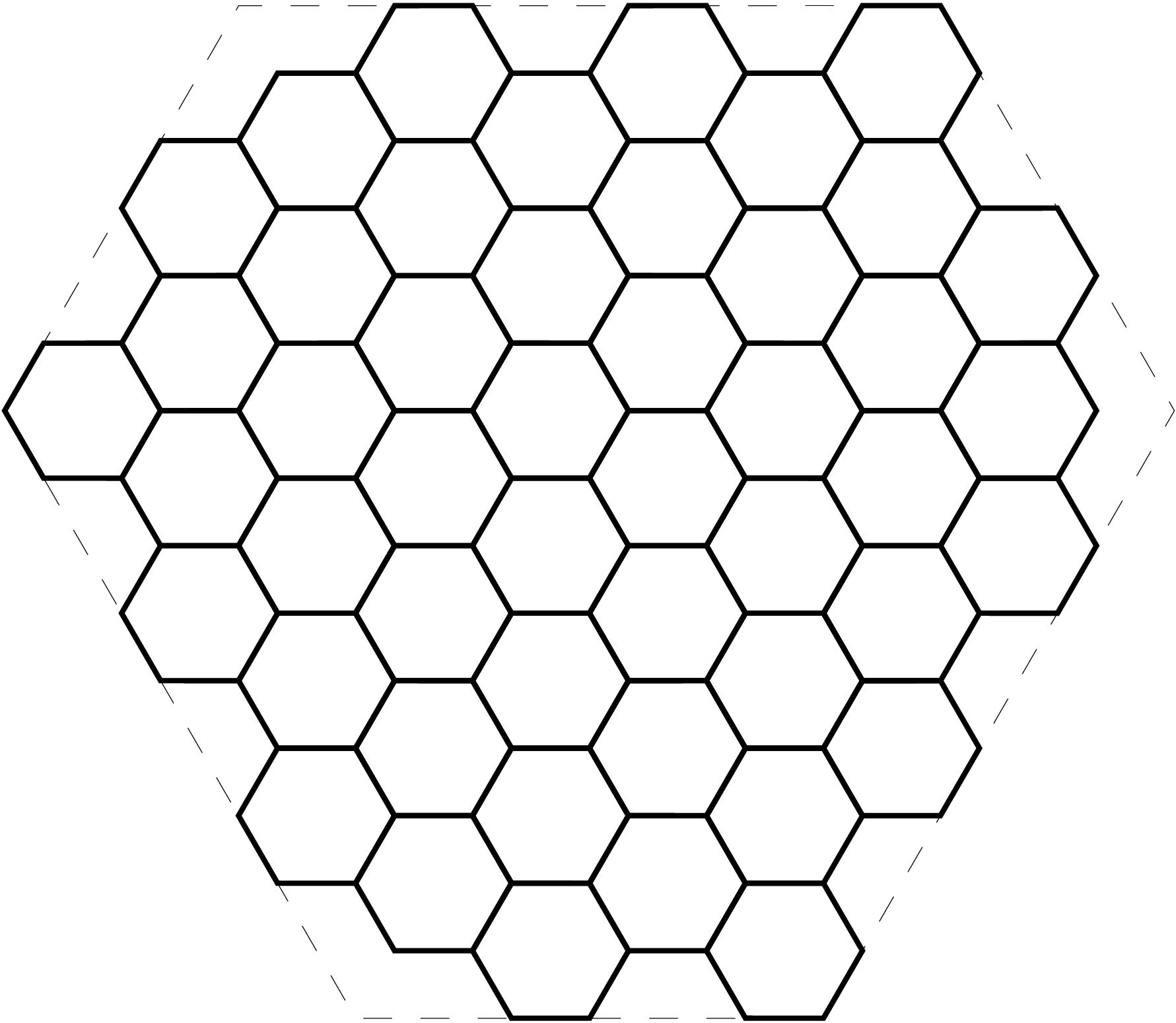}}
\hfill
}
\caption{\label{fca} The $(a,b)$-benzel for $a=7$, $b=8$.}
\end{figure}

In this section we solve four open problems posed by Propp \cite{propp2022trimer} on the enumeration of trimer coverings of regions called \textit{benzels}. To state the problems, we first define the benzel regions. For any positive integers $a$ and $b$ such that $2\leq a\leq 2b$ and $2\leq b\leq 2a$, we defined the $(a,b)$-benzel as follows. Consider a hexagonal grid on the plane consisting of regular hexagons of side-length $1$, each having two horizontal sides (see Figure \ref{fca}; this can be thought of as a tessellation of the plane by unit hexagons). We call these unit hexagons \textit{cells}. Regard the plane as the complex plane so that the origin is the left vertex of one of the cells. Let $H$ be the hexagon whose vertices are $\omega^{j}(a\omega+b)$ and $-\omega^{j}(a+b\omega)$ for $j\in\{0,1,2\}$, where $\omega=e^{2\pi i/3}$. The {\it $(a,b)$-benzel} is defined to be the union of all the cells that are contained in the hexagon $H$; Figure \ref{fca} shows the $(7,8)$-benzel.

\begin{figure}
\vskip0.2in
\centerline{
\includegraphics[width=0.7\textwidth]{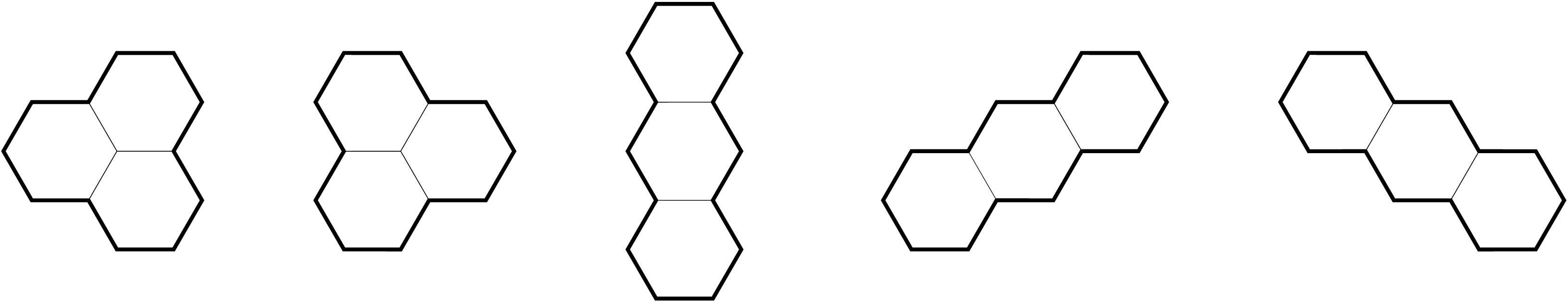}
}
\caption{\label{fcb} Left stone, right stone, vertical bone, rising bone, and falling bone (left to right).}
\end{figure}

Propp considered tilings of the benzel using the following two types of tiles. A \textit{stone} is the union of three cells that are pairwise adjacent. A \textit{bone} is the union of three contiguous cells whose centers are collinear. Taking orientation into account, there are five different types of tiles: the \textit{left stone}, the \textit{right stone}, the \textit{vertical bone}, the \textit{rising bone}, and the \textit{falling bone} (see Figure~\ref{fcb}). Given a benzel, a \textit{trimer cover} of the region is a collection of stones and bones that covers the region without gaps or overlaps.
By considering various restrictions on what tiles are allowed to be used, and the fact that the detailed shape of the $(a,b)$-benzel depends on the residues of $a$ and $b$ modulo~3, Propp was led to twenty open problems, stated in  \cite{propp2022trimer}. 
Some of these conjecture either explicit formulas or explicit recurrences for the number of trimer covers, while others concern combinatorial, respectively number theoretic properties of these numbers.

One  example is: ``How many trimer covers does the $(a,b)$-benzel have, in which all the three types of bones, but not both types of stones are allowed?" (Problem 3 in \cite{propp2022trimer}). See \cite{defant2024tilings}\cite{defant2023tilings}\cite{kim2023pentagonal} for recent progress on some of these problems. The open problems that we solve in this section are about the enumeration of trimer covers of benzels that can use both kinds of stones, the rising bone and the falling bone (in other words, the only restriction is that they cannot use any vertical bone). In this paper, we call such tilings
\textit{vertical-bone-free tilings}.
Based on numerical data, Propp conjectured that the number of such trimer covers of the $(3n,3n)$-benzel (Problem 8), the $(3n+1,3n+1)$-benzel (Problem 9), the $(3n+1,3n+2)$-benzel (Problem 10), and the $(3n-1,3n)$-benzel (Problem~11) satisfy explicit recurrence relations.
Proving such statements is in principle very hard, as, in contrast to the situation of dimer coverings, there are not many enumeration techniques available for trimer covers.

Recently, Defant et al.\ \cite{defant2024tilings} found a very useful bijection, which they call {\it compression}, allowing one to convert trimer enumeration problems into dimer enumeration problems (see Section~3 in \cite{defant2024tilings} for more details about this bijection). This opens up the possibility of proving some of Propp's conjectures using dimer enumeration techniques. We now introduce four families of regions whose domino tiling enumeration (equivalently, perfect matching enumeration of their planar dual graphs) are equivalent, thanks to the results in  \cite{defant2024tilings}, to Problems 8--11 from Propp's list.

Recall that for a positive integer $n$, the \textit{Aztec diamond of order $n$}, denoted $AD_n$, is the region consisting of all the unit squares on $\Z^2$ whose centers $(x,y)$ satisfy the inequality $|x|+|y|\leq n$ (Figure \ref{fcc} illustrates the Aztec diamond of order $7$).


\begin{figure}
\vskip0.2in
\centerline{
\includegraphics[width=0.4\textwidth]{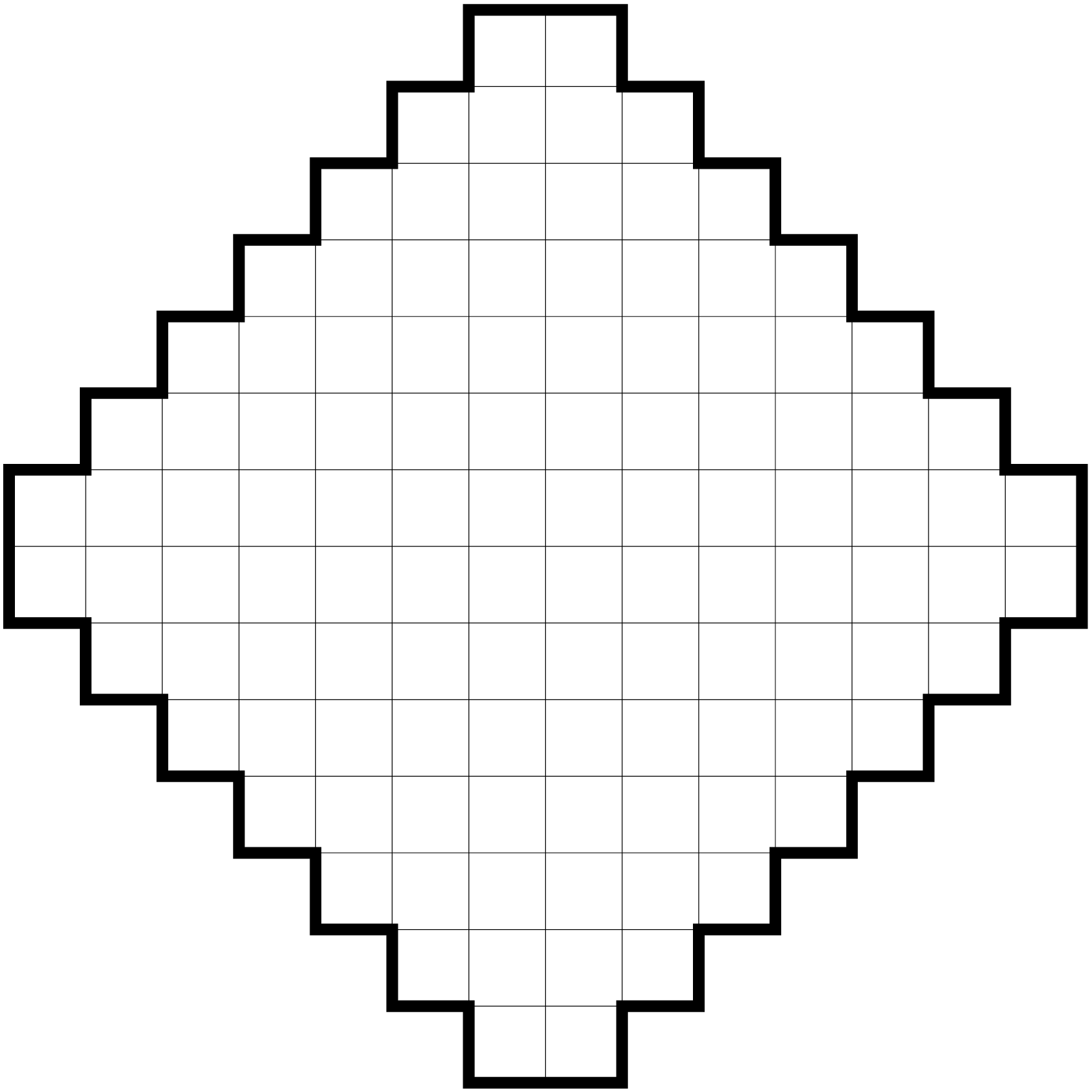}
}
\caption{\label{fcc} Aztec diamond $AD_n$ for $n=7$.}
\end{figure}

\begin{figure}
\vskip0.2in
\centerline{
\includegraphics[width=0.6\textwidth]{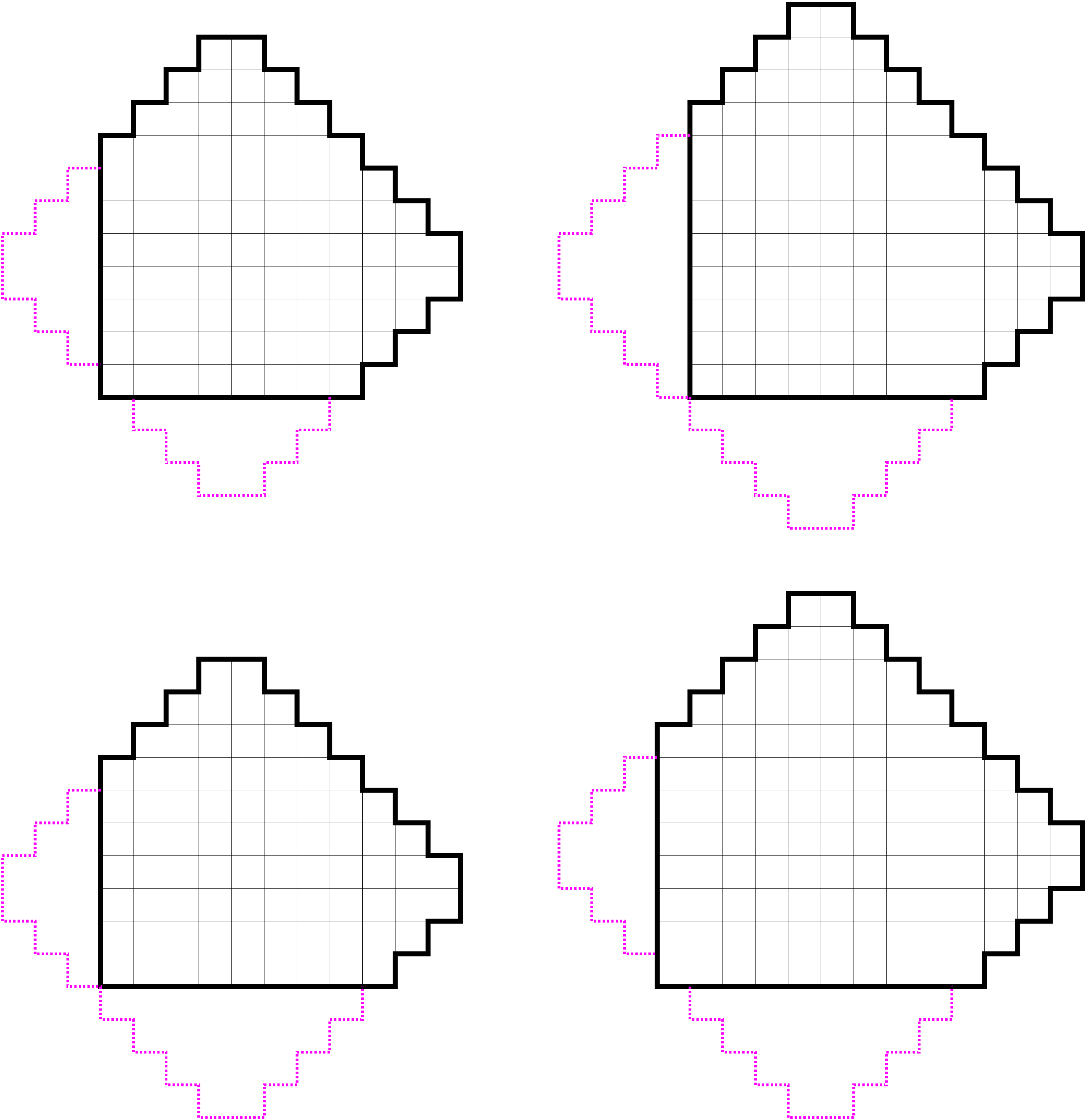}}
\caption{\label{fcd} $TAD_{7}$ (top left), $TAD_{8}$ (top right), $TAD'_{7}$ (bottom left), and $TAD'_{8}$ (bottom right).}
\end{figure}

The four families of regions mentioned above, which we denote by $TAD_{2n-1}$, $TAD_{2n}$, $TAD'_{2n-1}$, $TAD'_{2n}$, are simple truncations of the Aztec diamonds, obtained as follows.

To obtain the region $TAD_{2n-1}$, start with the Aztec diamond $AD_{2n-1}$, and consider the lattice point $P$ on its southwestern boundary which lies on its southwest-to-northeast going symmetry axis. The portion of $AD_{2n-1}$ contained in the ``northeast quadrant'' centered at $P$ is the region $TAD_{2n-1}$ (the top left picture in Figure \ref{fcd} shows $TAD_7$). The region $TAD_{2n}$ is defined in the same way, starting with $AD_{2n}$ ($TAD_8$ is shown on the top right of Figure \ref{fcd}); the reason we consider separately the cases of even and odd indices is because the details of their boundaries are slightly different, and this will be reflected in their tiling enumeration formulas.

The region $TAD_{2n-1}'$ is defined almost exactly like  $TAD_{2n-1}$, with the one difference that the boundary lattice point from which the above-described truncation is performed is not the point $P$ that lies on the southwest-to-northeast going symmetry axis, but is chosen instead to be the boundary point $P'$ which is one unit step north of $P$ (the bottom left picture in Figure \ref{fcd} illustrates $TAD_7'$). Similarly, the only difference between $TAD'_{2n}$ and $TAD_{2n}$ is that for the former the truncation of $AD_{2n}$ is not performed from $P$, but from the lattice point which is one unit step west of $P$; the bottom right picture in Figure \ref{fcd} illustrates $TAD'_8$.

In \cite{defant2024tilings} Defant et al.\ proved that the trimer enumeration questions in Problems 8--11 on Propp's list can be restated in terms of domino tilings of truncated Aztec diamonds as follows\footnote{ For a lattice region $R$ on the square lattice, $\M(R)$ denotes the number of its domino tilings.} (see Conjectures 5.1--5.3 and Question 5.4 in \cite{defant2024tilings}).

\begin{thm}[Defant et al.\ \cite{defant2024tilings}]
\label{tcaa}
Let $n$ be a positive integer.

$(${\rm a}$)$.\ The number of vertical-bone-free tilings of the $(3n,3n)$-benzel is $\M(TAD_{2n-1})$.  

$(${\rm b}$)$.\ The number of vertical-bone-free tilings of the $(3n+1,3n+1)$-benzel is $\M(TAD_{2n})$.

$(${\rm c}$)$.\ The number of vertical-bone-free tilings of the $(3n-1,3n)$-benzel is $\M(TAD'_{2n-1})$.

$(${\rm d}$)$.\ The number of vertical-bone-free tilings of the $(3n+1,3n+2)$-benzel is $\M(TAD'_{2n})$.
  
\end{thm}
%


Thus, to prove Propp's original conjectures, it suffices to enumerate the  domino tilings of $TAD_{2n-1}$, $TAD_{2n}$, $TAD'_{2n-1}$, and $TAD'_{2n}$. The main result of this section, Theorem \ref{tca} below, provides explicit, simple product formulas for the number of domino tilings of these regions.


\begin{thm}
\label{tca}
For positive integers $n$, the numbers of domino tilings of the regions $TAD_{2n-1}$, $TAD_{2n}$, $TAD'_{2n-1}$, and $TAD'_{2n}$ are given by the following product formulas\footnote{ Throughout this paper, an empty product is defined to be $1$ and $(2m-1)!!\coloneqq(2m-1)(2m-3)\cdots1$ for postive integers $m$.}:
\begin{equation}
\label{eca}
    \M(TAD_{2n-1})=2^{n^2}\cdot\frac{(2n-1)!!}{n!}\cdot\prod_{i=1}^{2n-1}\frac{(2i)!}{(n+i)!},
\end{equation}
\begin{equation}
\label{ecb}
    \M(TAD_{2n})=2^{n(n+1)}\cdot\prod_{i=0}^{n-1}\frac{(4i+2)!(4i+3)!}{(n+2i+1)!(n+2i+2)!},
\end{equation}
\begin{equation}
\label{ecc}
    \M(TAD'_{2n-1})=2^{n^2-1}\cdot\prod_{i=0}^{n-1}\frac{(4i+2)!}{(n+2i+1)!}\cdot\prod_{i=0}^{n-2}\frac{(4i+3)!}{(n+2i+1)!},
\end{equation}
and
\begin{equation}
\label{ecd}    
    \M(TAD'_{2n})=2^{n(n+1)}\cdot\frac{(2n+1)!!}{(n+1)!}\cdot\prod_{i=0}^{n-1}\frac{(4i+2)!(4i+4)!}{(n+2i+1)!(n+2i+3)!}.
\end{equation}
\end{thm}

\begin{cor}
\label{ccb}
    Let $n$ be a positive integer.

$(${\rm a}$)$.\ The number of vertical-bone-free tilings of the $(3n,3n)$-benzel is given by \eqref{eca}.

$(${\rm b}$)$.\ The number of vertical-bone-free tilings of the $(3n+1,3n+1)$-benzel is given by \eqref{ecb}.

$(${\rm c}$)$.\ The number of vertical-bone-free tilings of the $(3n-1,3n)$-benzel is given by \eqref{ecc}.

$(${\rm d}$)$.\ The number of vertical-bone-free tilings of the $(3n+1,3n+2)$-benzel is given by \eqref{ecd}.
\end{cor}

\begin{figure}[t]
\centerline{
\hfill
{\includegraphics[width=0.20\textwidth]{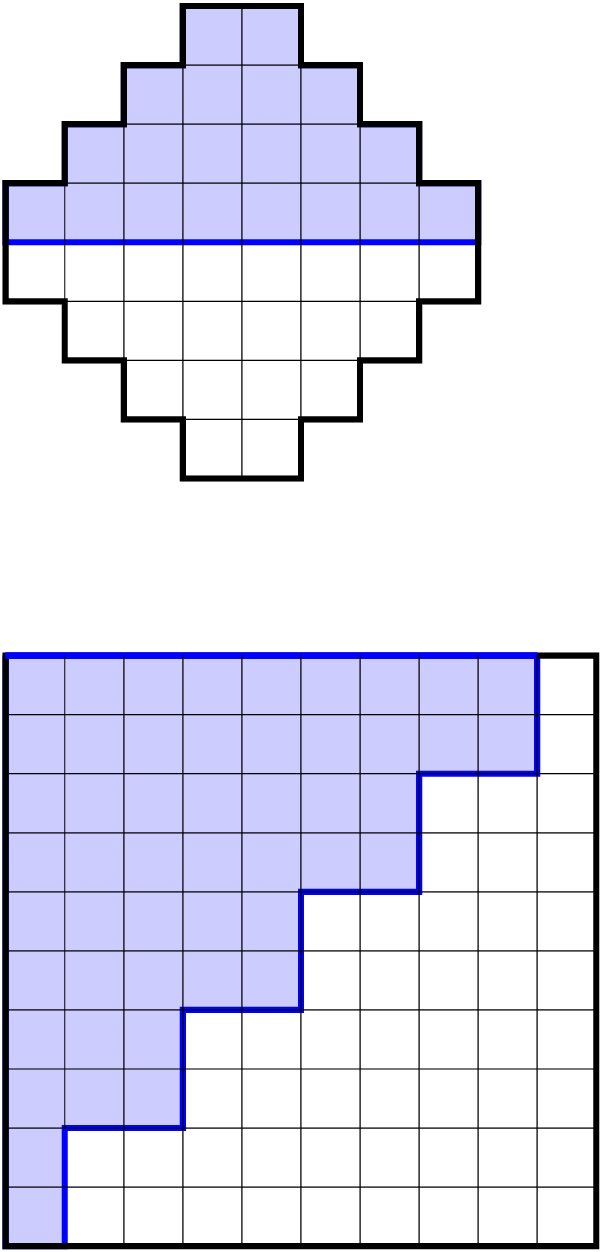}}
\hfill
{\includegraphics[width=0.18\textwidth]{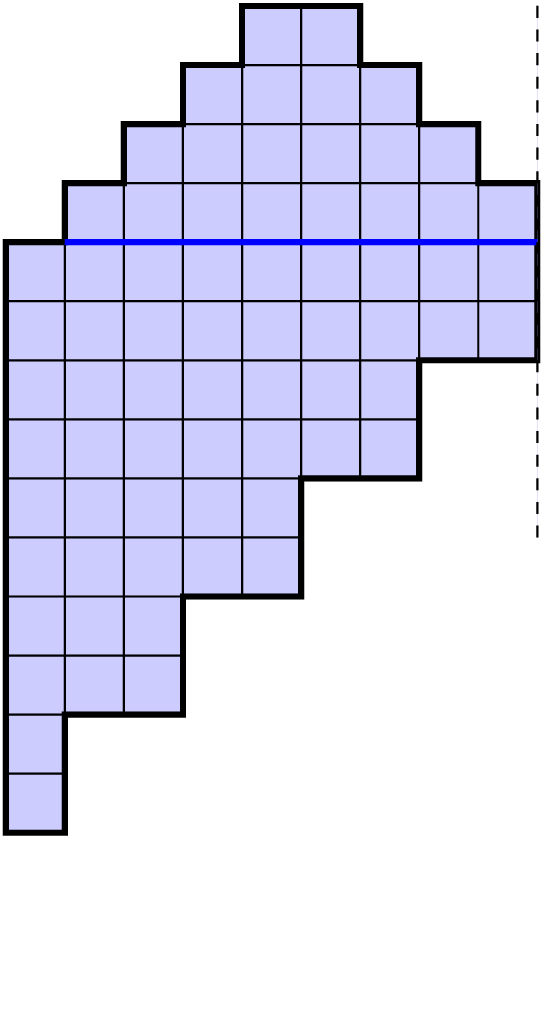}}
\hfill
{\includegraphics[width=0.18\textwidth]{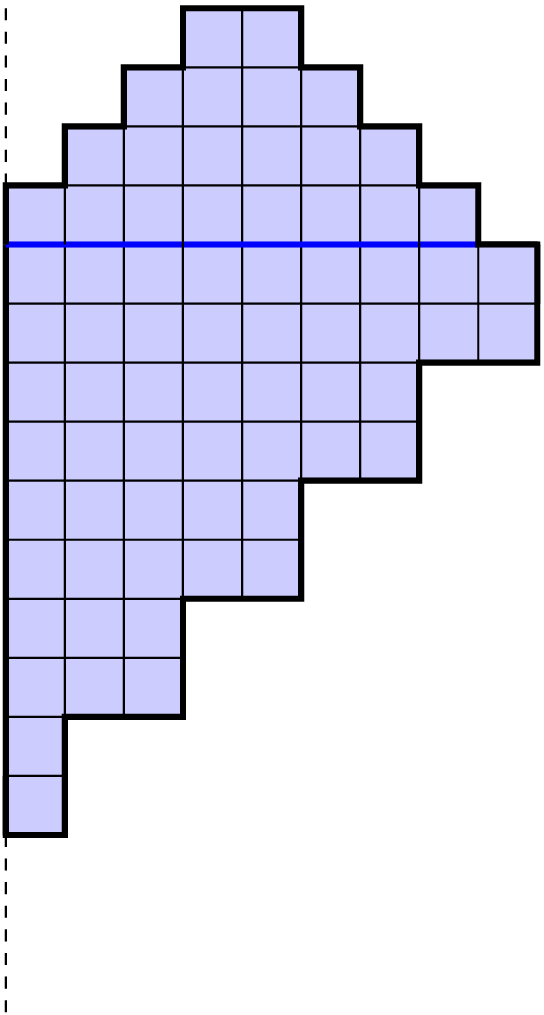}}
\hfill
}
\vskip0.1in
\caption{\label{fce} Using half of $AD_{n-1}$ and half of the square $S_{2n}$ (left) to get the Aztec triangles ${\mathcal T}_n$ (center) and ${\mathcal T}'_n$ (right), for $n=5$.}
\centerline{
\hfill
{\includegraphics[width=0.22\textwidth]{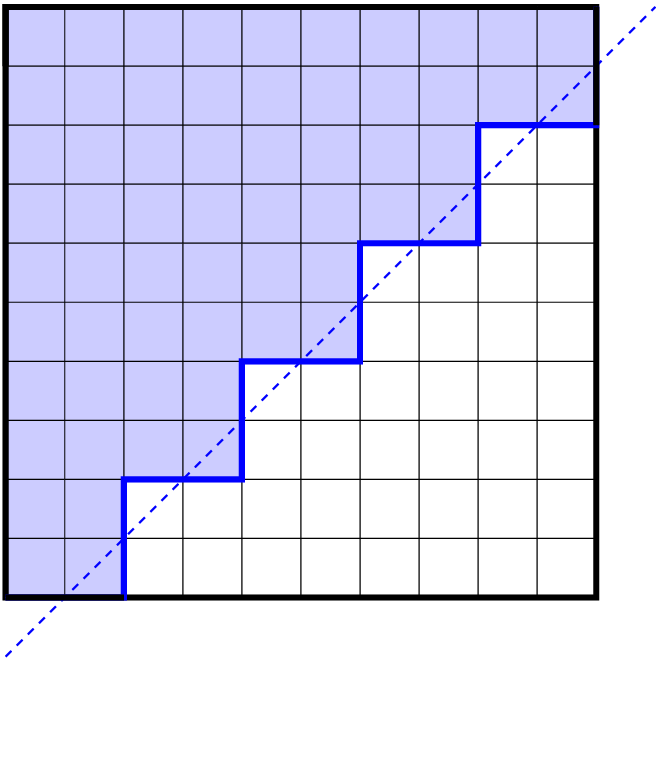}}
\hfill
{\includegraphics[width=0.20\textwidth]{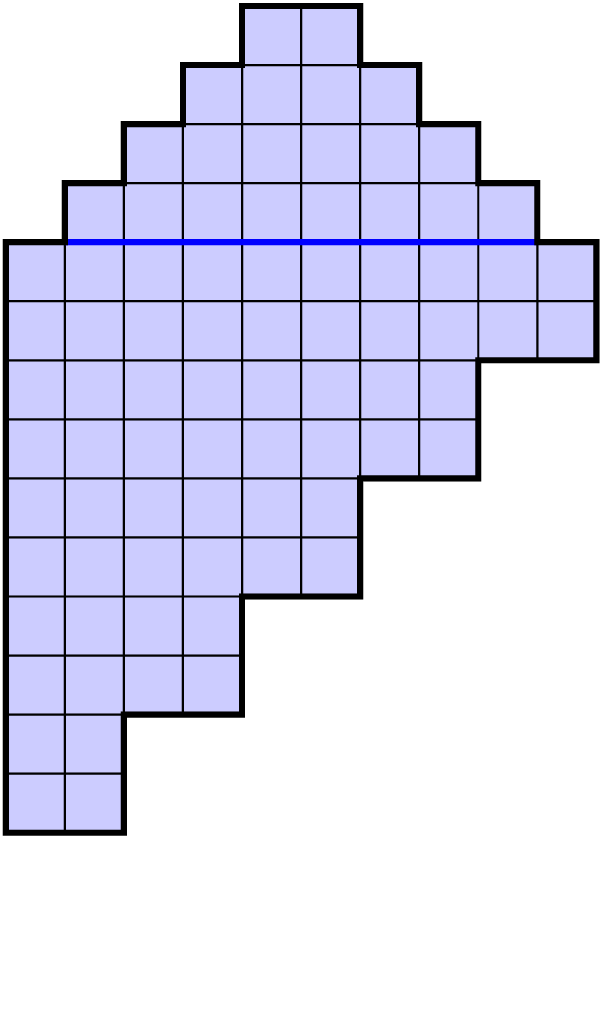}}
\hfill
{\includegraphics[width=0.20\textwidth]{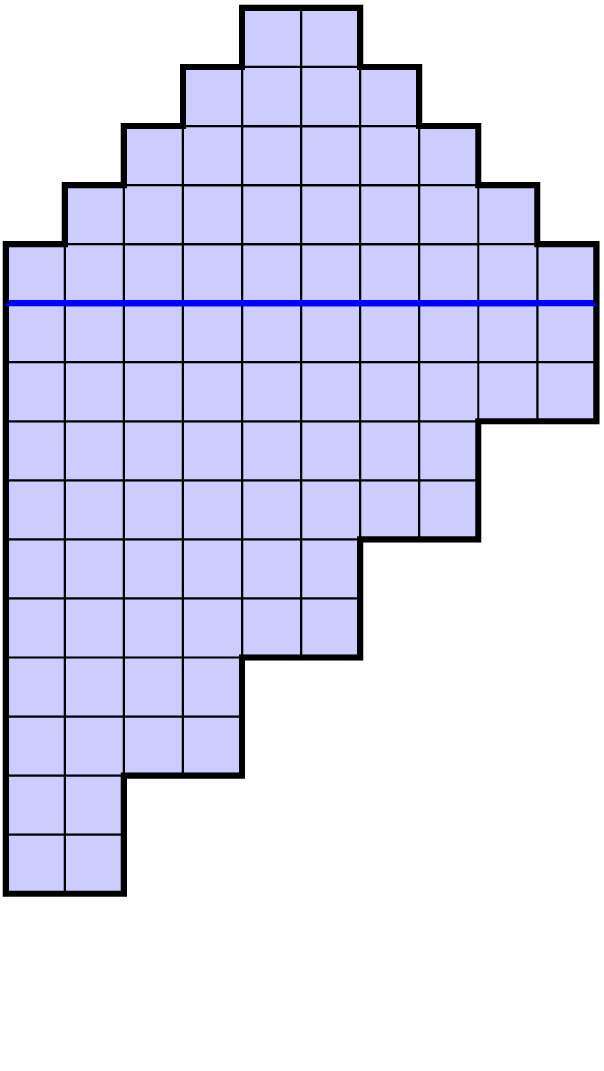}}
\hfill
}
\vskip-0.3in
\caption{\label{fcf} Using an augmented half of the square $S_{2n}$ (left) and half of $AD_{n-1}$ to get the Aztec triangle ${\mathcal T}''_n$ (center), and the same augmented half of the square $S_{2n}$, but together with half of $AD_{n}$, to obtain the Aztec triangle ${\mathcal T}'''_{n+1}$ (right), for $n=5$.}
\vskip-0.1in

\end{figure}



\medskip
The proof of Theorem \ref{tca} is based on Theorem \ref{tba}(a) (for \eqref{eca} and \eqref{ecb}), Theorem \ref{tba}(b) (for \eqref{ecc} and \eqref{ecd}), and some enumeration results of Corteel et al.\ presented in \cite{corteel2023domino}. As a result of applying Theorem \ref{tba}, four new families of regions arise.


One of these families consists of the Aztec triangles introduced by Di Francesco \cite{di2021twenty}, and the remaining three are variants of it. All of them can be regarded as hybrids between Aztec diamonds and squares. Namely, let $S_{2n}$ be the square of side-length $2n$ drawn on the square grid. Cut it in two congruent parts\footnote{ This way of cutting up the square appears also in the second author's paper\cite{CiucuMatchingFactorization} and in Pachter's work \cite{Pachter}.} by a zig-zag lattice path of step length 2, as shown in Figure \ref{fce}, and glue together the top half with the top half of $AD_{n-1}$. If this is done in a right-justified manner, one obtains Di Francesco's Aztec triangle ${\mathcal T}_n$ (see the picture in the center in Figure \ref{fce}); the left-justified fashion results in the Aztec triangle  ${\mathcal T}'_n$ (see the picture on the right in Figure \ref{fce}).

To obtain the remaining two variants, use the ``augmented half'' of the square $S_{2n}$ obtained by cutting $S_{2n}$ with the translation of the above zig-zag cut one unit square diagonal in the southeast direction (see the picture on the left in Figure \ref{fcf}). Gluing the top half of $AD_{n-1}$ to this augmented half results in the Aztec triangle ${\mathcal T}''_n$ (see the picture in the center in Figure \ref{fcf}), while gluing the top half of $AD_{n}$ to it produces the Aztec triangle ${\mathcal T}'''_{n+1}$ (see the picture on the right in Figure \ref{fcf}). Note that for all the four variants, the order of the Aztec triangle is one more than the order of the Aztec diamond whose top half was used to construct it.

We will find it convenient to have the indexing of the regions in the last family shifted by one unit, as in the above definition. But due to this, we need to define ${\mathcal T}'''_1$ separately: we set ${\mathcal T}'''_1$ to be the empty region, which has one domino tiling (the empty set of dominos).

These four families of
regions turn out to be special cases of regions considered by Corteel et al. in \cite{corteel2023domino}.  In that paper, the authors consider a family of regions parametrized by positive integers $l$ and $k$. The region ${\mathcal T}_n$ represents the special case $l=2n$ and $k=n$.
Furthermore, ${\mathcal T}'_n$ is the special case $l=2n+1$ and $k=n-1$, ${\mathcal T}''_n$ the special case $l=2n+1$ and $k=n$, and ${\mathcal T}'''_n$ is obtained by setting $l=2n$ and $k=n-1$.

By specializing the values of $l$ and $k$ in Theorem 1.2 of \cite{corteel2023domino} as indicated in the previous paragraph, we obtain the following enumeration results for the number of domino tilings of these four families of regions.

 \begin{lemma}
 \label{lcc}
For positive integers $n$, the number of domino tilings of the
four types of Aztec triangles
is given by
    \begin{equation}
    \label{ece}
        \M({\mathcal T}_n)=\M({\mathcal T}'_n)=2^{n(n-1)/2}\cdot\prod_{i=0}^{n-1}\frac{(4i+2)!}{(n+2i+1)!},
    \end{equation}
    \begin{equation}
    \label{ecf}
        \M({\mathcal T}''_n)=2^{n(n+1)/2}\cdot\prod_{i=0}^{n-1}\frac{(4i+3)!}{(n+2i+2)!},
    \end{equation}
    and
\begin{equation}
\label{ecg}
\M({\mathcal T}'''_n)=2^{n(n-1)/2}\cdot\frac{(2n-1)!!}{n!}\cdot\prod_{i=1}^{n-1}\frac{(4i)!}{(n+2i)!}.
\end{equation}

 \end{lemma}
The equality $\M({\mathcal T}_n)=\M({\mathcal T}'_n)$ was proved by Corteel et al. in \cite{corteel2023domino} by showing that the two numbers are given by the same formula. Recently, a bijective proof was given in \cite{ByunCiucu2024squarish} by the first two authors of the current paper.

\begin{proof}[Proof of Theorem \ref{tca}] The first two equalities \eqref{eca} and \eqref{ecb} turn out to be direct consequences of the factorization theorem for perfect matchings (Theorem \ref{tba}(a)) and the specializations of \cite[Theorem 1.2]{corteel2023domino} stated in Lemma \ref{lcc}. Indeed, note that the regions $TAD_{2n-1}$ and $TAD_{2n}$ are symmetric.
If we apply the factorization theorem to their planar duals, we get, after removing some forced dominos (see Figure \ref{fcg})

\begin{figure}
\vskip0.2in
\centerline{
\includegraphics[width=0.7\textwidth]{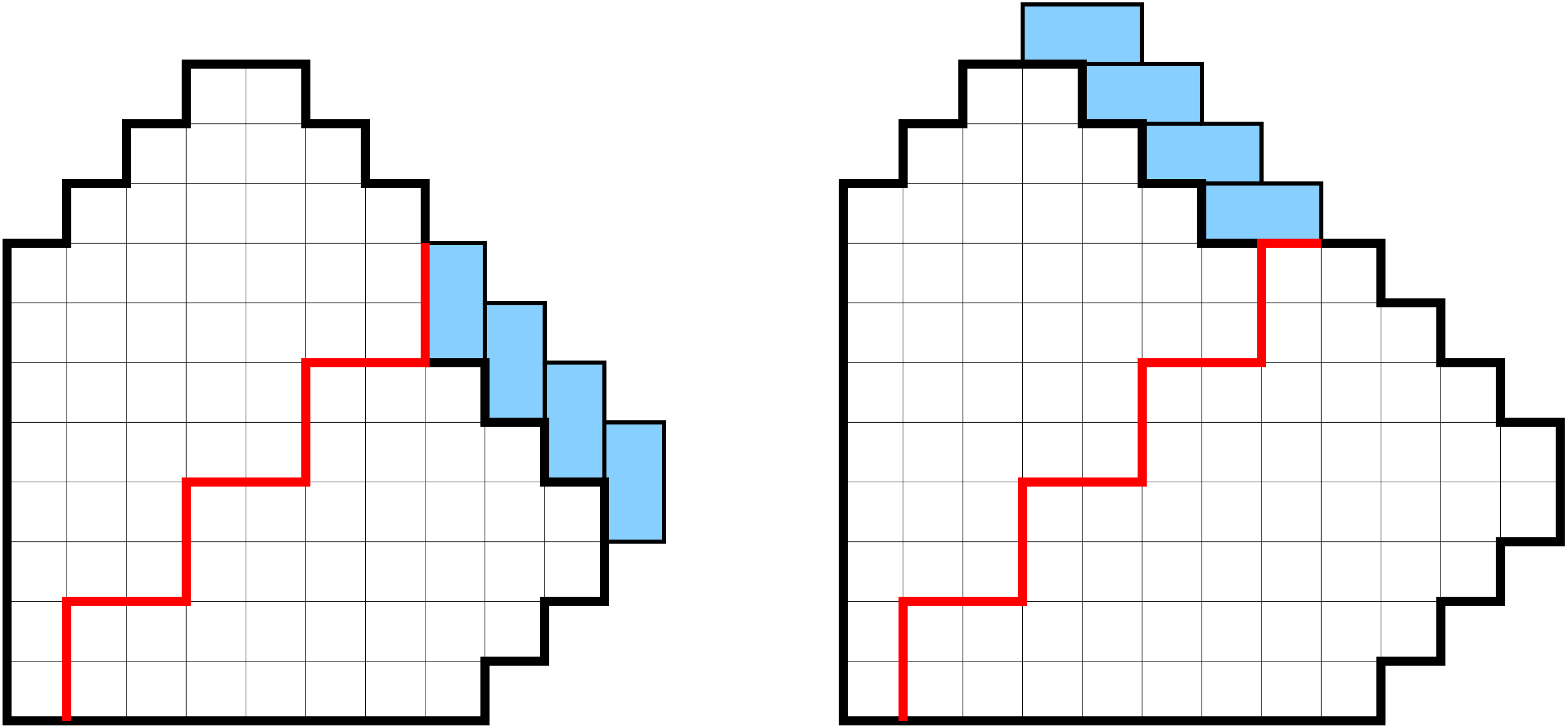}
}
\caption{\label{fcg} {\it Left.} Factorization of $TAD_{7}$ into ${\mathcal T}_4$ and ${\mathcal T}'''_4$. {\it Right.} Factorization of $TAD_{8}$ into ${\mathcal T}'_4$ and ${\mathcal T}''_4$.}
\end{figure}

\begin{equation}
\label{ech}
    \M(TAD_{2n-1})=2^{n} \M({\mathcal T}_n)  \M({\mathcal T}'''_n),
\end{equation}
and
\begin{equation}
\label{eci}
    \M(TAD_{2n})=2^{n} \M({\mathcal T}'_n) \M({\mathcal T}''_n).    
\end{equation}

The situation is different for
the regions $TAD'_{2n-1}$ and $TAD'_{2n}$, as they are not symmetric. To apply Theorem \ref{tba}(b), we first ``symmetrize'' these regions about their diagonals; denote by $\overline{TAD'}_{2n-1}$ and $\overline{TAD'}_{2n}$, respectively, the resulting symmetric regions (in Figure \ref{fch}, these are the regions enclosed by the outermost contours). These new regions have no domino tilings, as they have an odd number of vertices.
Let $v$ be the rightmost unit square in the top row of these regions ($v$ is colored black in Figure \ref{fch}). Since after removing the forced dominos from $\overline{TAD'}_{2n-1}\setminus v$ one obtains the region $TAD'_{2n-1}$, these two regions have the same number of tilings. The same is true for the regions $\overline{TAD'}_{2n}\setminus v$ and $TAD'_{2n}$. Note that the dual graphs of $\overline{TAD'}_{2n-1}\setminus v$ and $\overline{TAD'}_{2n}\setminus v$ are symmetric graphs with a boundary defect, so we can apply Theorem \ref{tba}(b) to them. Applying it we get (after removing the forced dominos; see Figure \ref{fch})

\begin{figure}
\vskip0.2in
\centerline{
\includegraphics[width=0.7\textwidth]{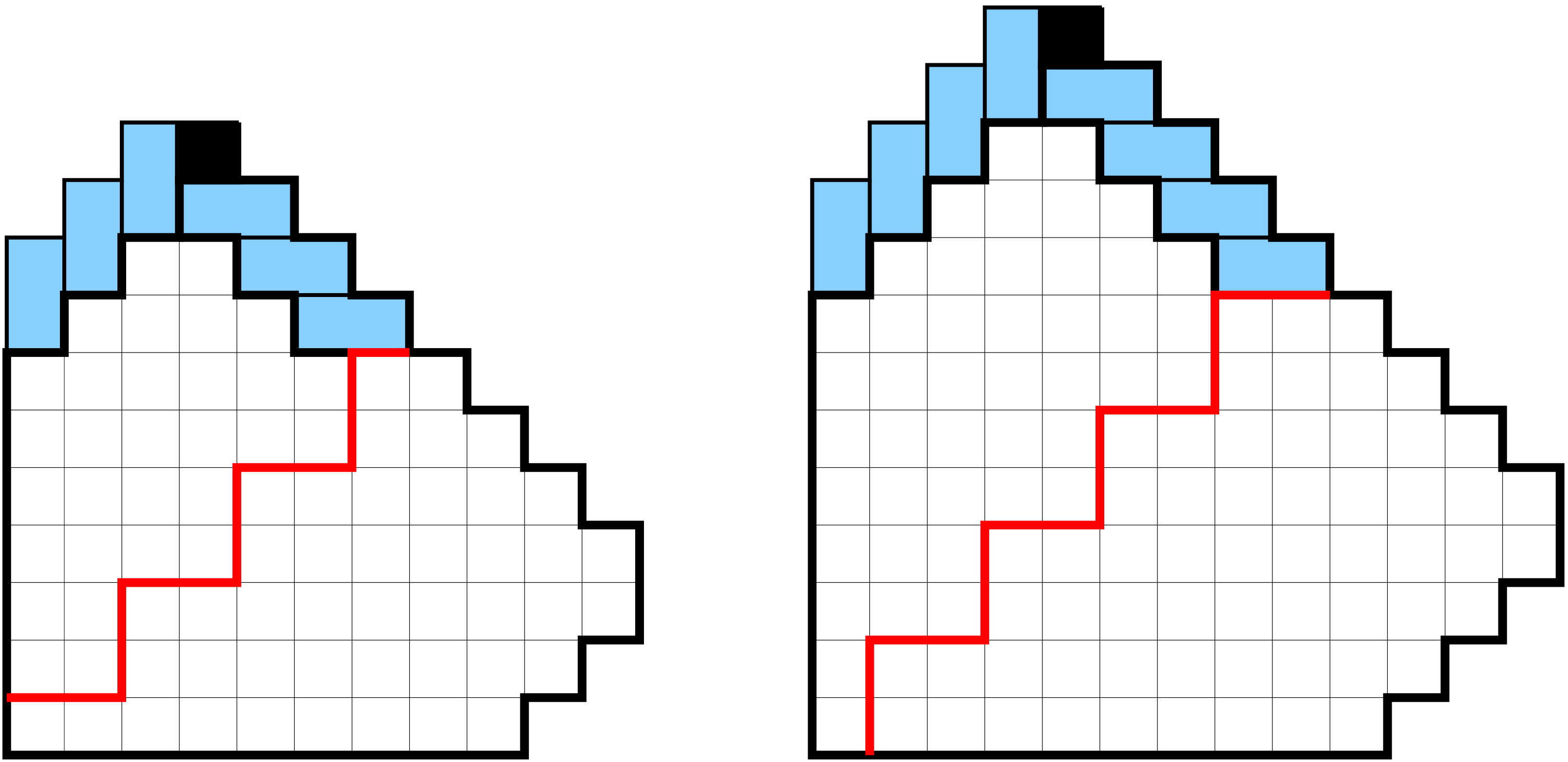}
}
\caption{\label{fch} Factorization of $\overline{TAD'}_{7}\setminus v$ into ${\mathcal T}''_3$ and  ${\mathcal T}'_4$ (left) and that of $\overline{TAD'}_{8}\setminus v$ into ${\mathcal T}_4$ and ${\mathcal T}'''_5$ (right); $v$ is indicated by a black unit square.}
\end{figure}

\begin{equation}
\label{ecj}
    \M(TAD'_{2n-1})=\M(\overline{TAD'}_{2n-1}\setminus v)=2^{n-1} \M({\mathcal T}''_{n-1}) \M({\mathcal T}'_n),
\end{equation}
and
\begin{equation}
\label{eck}
    \M(TAD'_{2n})=\M(\overline{TAD'}_{2n}\setminus v)=2^{n} \M({\mathcal T}_n) \M({\mathcal T}'''_{n+1}).    
\end{equation}
It is straightforward to check that one gets expressions \eqref{eca}--\eqref{ecd} by combining equations \eqref{ece}--\eqref{ecg} and \eqref{ech}--\eqref{eck}. This completes the proof.
\end{proof}

{\it Remark $2$.}
The original conjectures of Propp did not state explicit formulas for the number of vertical-bone-free tilings of benzels. Instead, the conjectures stated that those numbers satisfy certain recurrence relations. More precisely, via the compression bijection of \cite{defant2024tilings}, Propp's conjectures (open Problems 8--11 in \cite{propp2022trimer}) are equivalent to the following equalities (\eqref{ecl} and \eqref{eco} were given in \cite{propp2022trimer} and the other two in \cite{defant2024tilings}):
\begin{equation}
\label{ecl}
    \frac{\M(TAD_{2n-1})\M(TAD_{2n+3})}{\M(TAD_{2n+1})^2}=\frac{256(2n+3)^{2}(4n+1)(4n+3)^{2}(4n+5)}{27(3n+1)(3n+2)^{2}(3n+4)^{2}(3n+5)}
\end{equation}
for all $n\geq1$,
\begin{equation}
\label{ecm}
    \frac{\M(TAD_{2n})}{\M(TAD_{2n-2})}=\frac{2^{2n}(4n-1)!(4n-2)!n!}{(3n)!(3n-1)!(3n-2)!}
\end{equation}
for all $n\geq2$,
\begin{equation}
\label{ecn}
    \frac{\M(TAD'_{2n-1})}{\M(TAD'_{2n-3})}=\frac{2^{2n-3}(4n-2)!(4n-4)!(n-1)!!(n-3)!!}{(3n-1)!!(3n-2)!(3n-3)!(3n-5)!!}
\end{equation}
for all $n\geq2$, and
\begin{equation}
\label{eco}
\begin{aligned}
    &\frac{\M(TAD'_{2n})\M(TAD'_{2n+6})}{\M(TAD'_{2n+2})\M(TAD'_{2n+4})}\\
    =&\frac{65536(2n+3)(2n+5)^{2}(2n+7)(4n+3)(4n+5)^{2}(4n+7)^{2}(4n+9)^{2}(4n+11)}{729(3n+2)(3n+4)^{2}(3n+5)^{2}(3n+7)^{2}(3n+8)^{2}(3n+10)}
\end{aligned}
\end{equation}
for all $n\geq1$. It is straightforward to check that formulas \eqref{eca}-\eqref{ecd} satisfy the above equalities \eqref{ecl}-\eqref{eco}; thus Corollary \ref{ccb} indeed solves the four open problems of Propp.

Note also that deducing the simple product formulas \eqref{eca}--\eqref{ecd} directly from the above recurrences would not be an easy task.

\section{Nearly symmetric hexagons with collinear holes}






In this section we solve an open problem posed by Lai --- namely, Problem 28 in \cite{LaiOpenProblems} --- by providing an explicit product formula for the number of lozenge tilings of a certain family of hexagonal regions with holes on the triangular lattice (the regions $H'_{\mathbf{l}}(a,b,k)$; see subsection \ref{subsection4.1} and Theorems \ref{tdb}(a) and \ref{tdc}(a)).
This family is closely related to one of three families of symmetric hexagonal regions with collinear triangular holes (the regions $H_{\mathbf{l}}(a,b,k)$, also described in subsection \ref{subsection4.1}) whose number of lozenge tilings was determined in \cite{CiucuPP1} by the second author of the current paper: it is obtained from the latter by translating all the holes one unit in the southeast lattice direction.
In addition to solving Lai's open problem, we also provide explicit product formulas for two more families of regions, obtained from the two remaining families from \cite{CiucuPP1} by the same procedure;
see Theorems \ref{tdb}(b)(c) and \ref{tdc}(b)(c).

\subsection{Six families of hexagonal regions with collinear holes}
\label{subsection4.1}

\begin{figure}[t]
\centerline{
\hfill
{\includegraphics[width=0.40\textwidth]{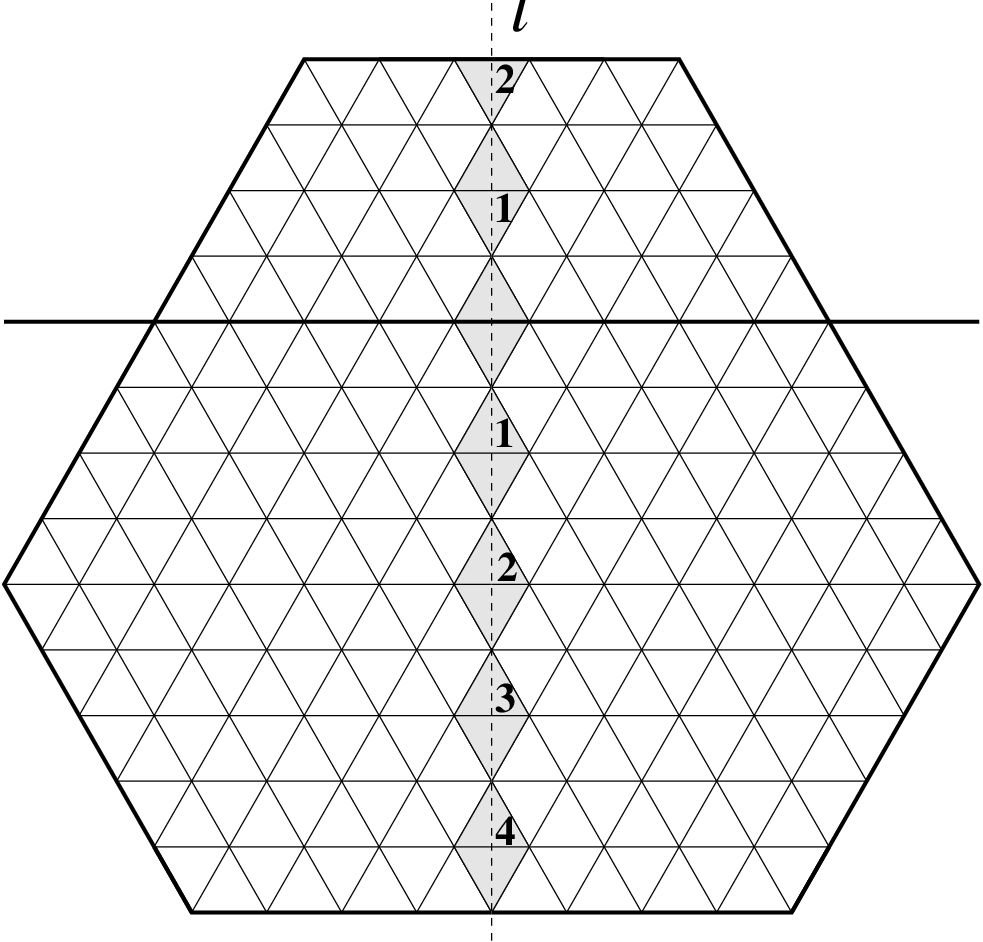}}
\hfill
}
\vskip0.1in
\caption{The labeling of the slots along a vertical line, starting from a reference horizontal lattice line; slots that are only partially contained in $H$ are also labeled.}
\vskip-0.1in
\label{fda}
\end{figure}

We now recall the definitions of the three families of regions $H_{\mathbf{l}}(a,b,k)$, $H_{{\mathbf{l}},{\mathbf{q}}}(a,b,k)$ and $\bar{H}_{{\mathbf{l}},{\mathbf{q}}}(a,b,k)$ introduced in \cite{CiucuPP1}; the other three families, $H'_{\mathbf{l}}(a,b,k)$, $H'_{{\mathbf{l}},{\mathbf{q}}}(a,b,k)$ and $\bar{H}'_{{\mathbf{l}},{\mathbf{q}}}(a,b,k)$ are defined by a simple variation of these.

Draw the triangular lattice so that one of the families of lattice lines is horizontal, and let $H$ be a hexagon on this lattice. Let $\ell$ be a vertical line containing lattice points and crossing the interior of $H$.
A {\it slot} is the union of two unit triangles crossed by $\ell$ which share an edge.
Essential in the definition of our regions is the concept of labeling the slots contained in $H$, starting from a reference horizontal lattice line $L$. This is very simple: label the slots on both sides of $L$ successively by $1, 2, 3,\dotsc$ (the two closest ones by 1, the next two by 2, and so on), including also any slots that are only partially contained in $H$; this is illustrated in Figure \ref{fda}.

\begin{figure}[t]
\centerline{
\hfill
{\includegraphics[width=0.40\textwidth]{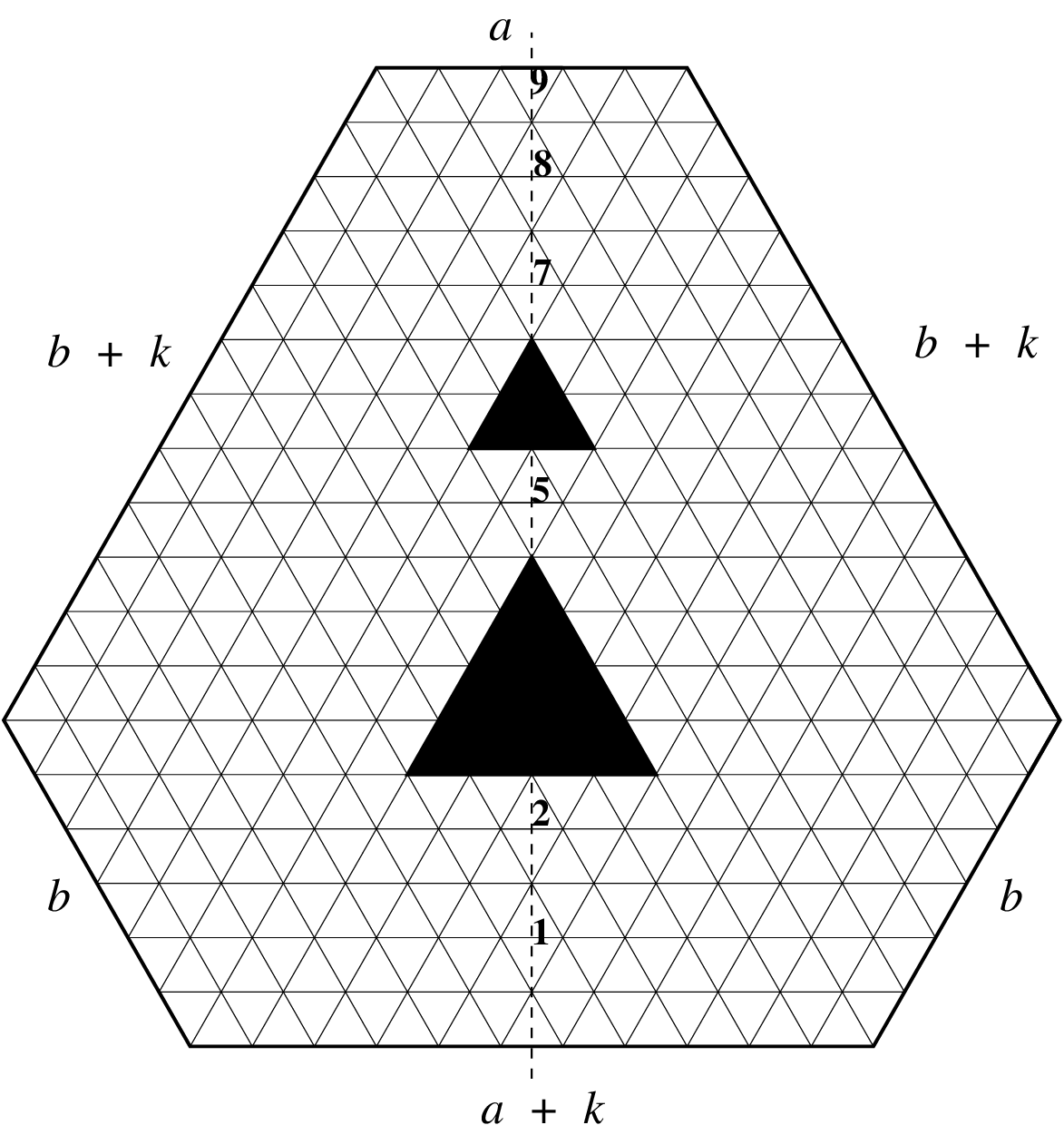}}
\hfill
{\includegraphics[width=0.40\textwidth]{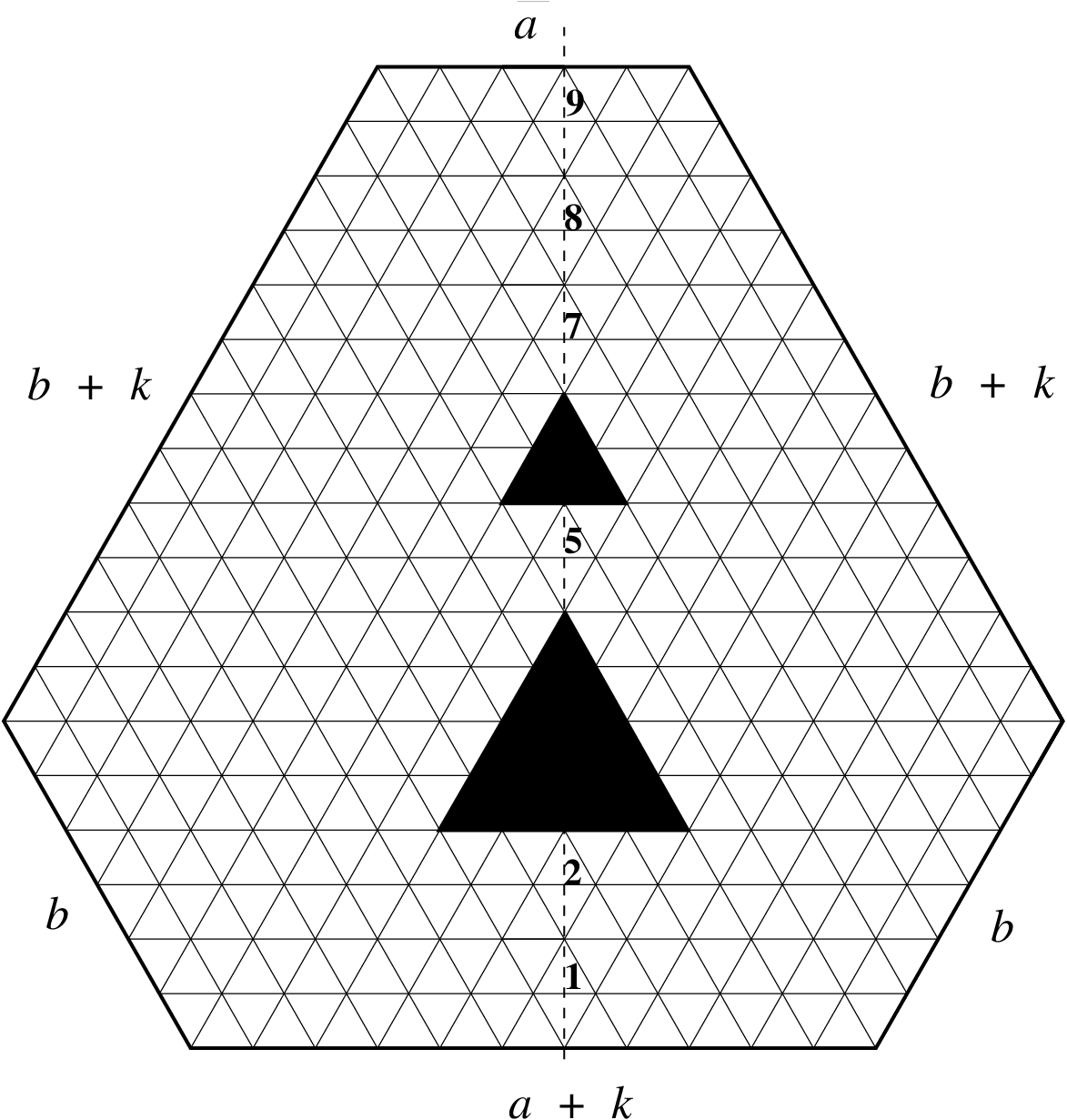}}
\hfill
}
\vskip0.1in
\caption{{\it Left.} The symmetric hexagon with holes $H_{\mathbf{l}}(a,b,k)$ for $a=5$, $b=6$, $k=6$, and ${\mathbf{l}}=(1,2,5,7,8,9)$. {\it Right.} The corresponding nearly symmetric hexagon with holes $H'_{\mathbf{l}}(a,b,k)$ for $a=5$, $b=6$, $k=6$, and ${\mathbf{l}}=(1,2,5,7,8,9)$. }
\vskip-0.1in
\label{fdb}
\end{figure}

Given non-negative integers $a$, $b$ and $k$, denote by $H(a,b,k)$ the hexagonal region on the triangular lattice whose side-lengths are $a$, $b+k$, $b$, $a+k$, $b$, $b+k$, clockwise from top. Thus $H(a,b,k)$ has $k$ more up-pointing unit triangles than down-pointing unit triangles. Our regions will be obtained from $H(a,b,k)$ by making some holes in it, so that the union of the holes has $k$ more up- than down-pointing triangles; this way, the resulting regions will have the same number of unit triangles of the two kinds, a necessary condition for the existence of lozenge tilings.

Let $\ell$ be the vertical symmetry axis of $H(a,b,k)$, and label the slots along $\ell$ choosing the reference line $L$ to be the base of $H(a,b,k)$ (see the picture on the left in Figure \ref{fdb} for an example). Choose an arbitrary subset of the labeled slots along $\ell$ (avoiding the possible partial slot at the top), and for each chosen slot remove the up-pointing lattice triangle of side-length two containing that slot. If the resulting region has forced lozenges, remove them\footnote{ By this, if a run of $k$ triangular holes of side-length two are contiguous, this gets replaced by a single triangular hole of side-length $2k$. In addition, it follows that the top triangular hole in $H_{\mathbf{l}}(a,b,k)$ does not touch the top side of the boundary.}. Let ${\mathbf{l}}$ be the list of the labels of slots not covered by removed triangles in the resulting region. Then we denote the resulting region by $H_{\mathbf{l}}(a,b,k)$; an example is shown on the left in Figure \ref{fdb}.

The regions $H'_{\mathbf{l}}(a,b,k)$
featured in Lai's open problem that we solve in this section 
are defined very similarly. Indeed, the only difference is that instead of the symmetry axis $\ell$, we consider its translation $\ell'$ half a unit to the right, and we label the slots along $\ell'$, with the reference line still being the base of $H(a,b,k)$. Letting ${\mathbf{l}}$ be the list of slot labels not covered by the removed triangles, the resulting region is defined to be $H'_{\mathbf{l}}(a,b,k)$ (see the picture on the right in Figure \ref{fdb} for an example). Lai noticed, based on data, that the number of lozenge tilings of these regions always seems to factor into relatively small prime factors, a fact that points to the existence of a simple product formula. In \cite[Problem 28]{LaiOpenProblems} he posed the open problem of finding and proving such a formula. We provide this in Theorem \ref{tdb}(a). Furthermore, in Theorem \ref{tdb}(b) we give simple product formulas for the number of tilings of two more families of regions, namely the regions $H'_{{\mathbf{l}},{\mathbf{q}}}(a,b,k)$ and $\bar{H}'_{{\mathbf{l}},{\mathbf{q}}}(a,b,k)$ defined below.

The main new ingredient in the last two pairs of families of regions is that their definition also involves making one triangular hole of odd side-length. More precisely, let the reference line $L$ be a horizontal lattice line that meets the region $H(a,b,k)$, and consider again the symmetry axis~$\ell$. Label the slots along $\ell$ with $L$ as the reference line, and remove an up-pointing triangle of odd size, symmetrically about $\ell$, whose base is along $L$. Then, as for the previous family, choose arbitrarily some labeled slots --- now generating two subsets of labels, one for slots below $L$ and one for those above $L$. Then remove triangles of side two straddling the chosen slots, with an important difference: for the slots above $L$, choose them to be up-pointing, but for the slots below $L$ make them {\it down-pointing}. Use the same convention about removing any forced lozenges. If in the resulting region the leftover labels below $L$ form the list ${\mathbf{l}}$, and those above $L$ the list  ${\mathbf{q}}$, we define the resulting region to be $H_{{\mathbf{l}},{\mathbf{q}}}(a,b,k)$; an example is illustrated on the left in Figure \ref{fdc}. We define the region $H'_{{\mathbf{l}},{\mathbf{q}}}(a,b,k)$ by applying exactly the same procedure, but with $\ell$ replaced by its half-unit translation to the right, $\ell'$; see the picture on the right in Figure \ref{fdc} for an example.

\begin{figure}[t]
\centerline{
\hfill
{\includegraphics[width=0.40\textwidth]{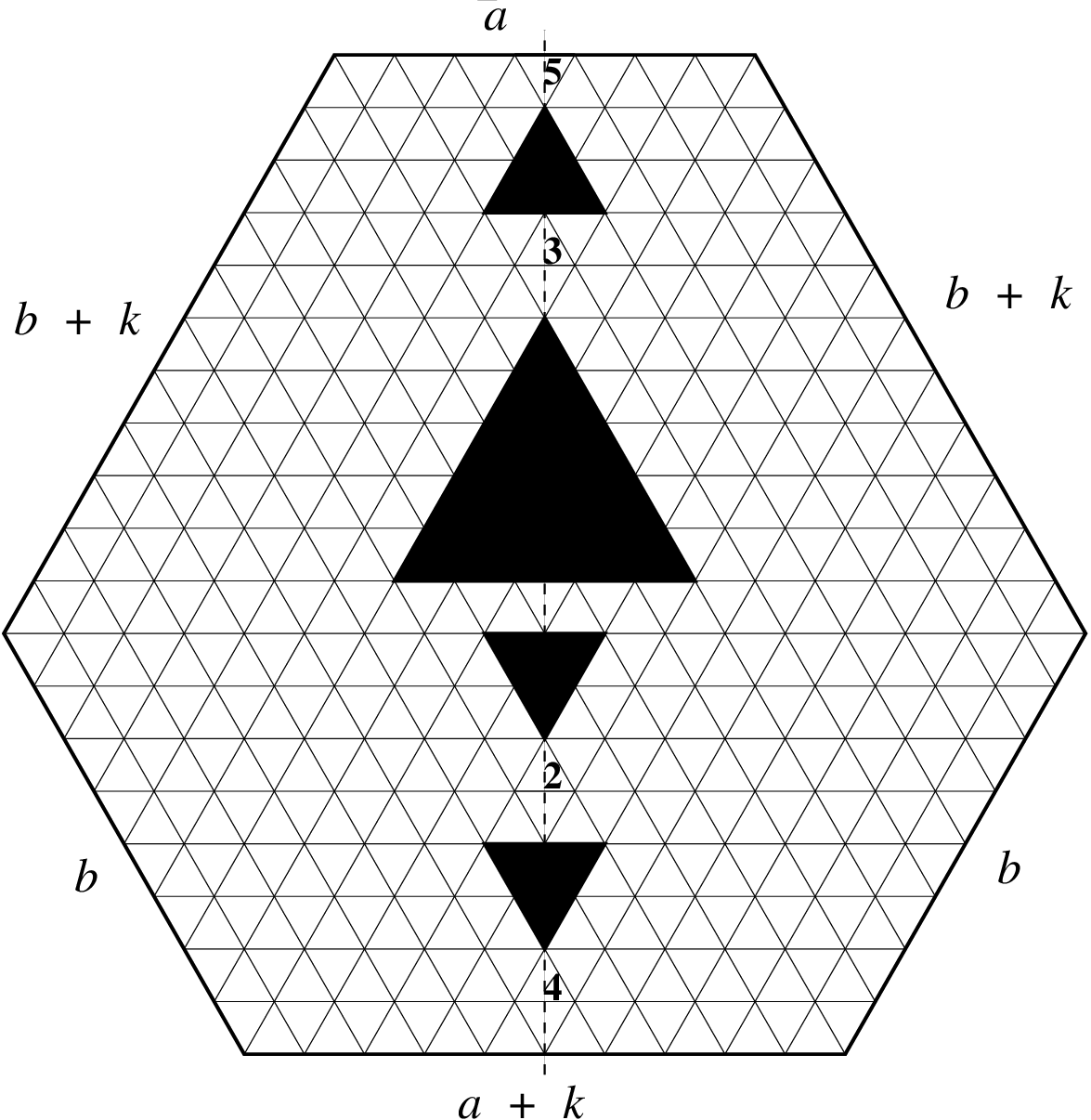}}
\hfill
{\includegraphics[width=0.40\textwidth]{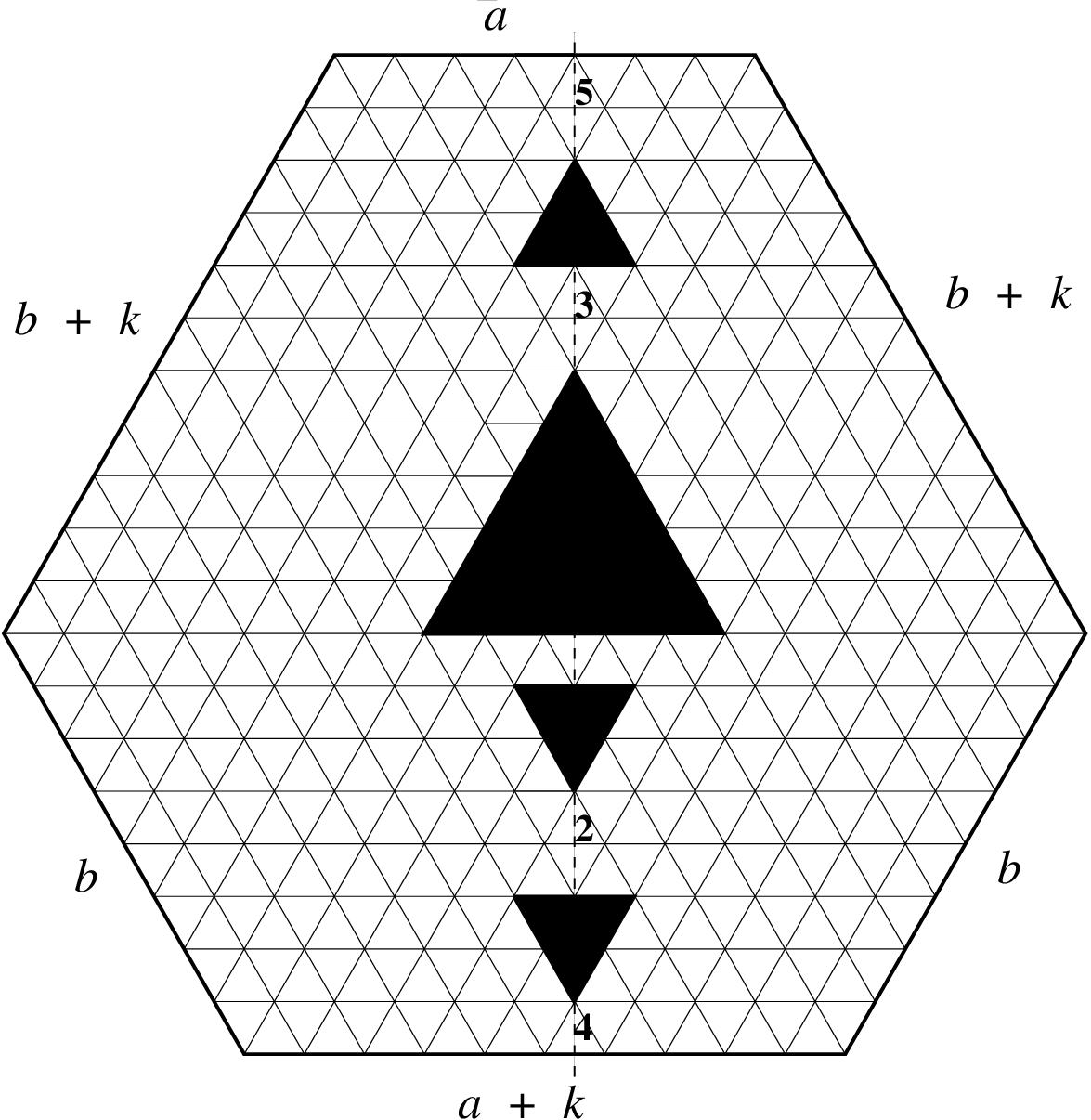}}
\hfill
}
\vskip0.1in
\caption{{\it Left.} The symmetric hexagon with holes $H_{{\mathbf{l}},{\mathbf{q}}}(a,b,k)$ for $a=7$, $b=8$, $k=3$, ${\mathbf{l}}=(2,4)$ and ${\mathbf{q}}=(3,5)$. {\it Right.} The corresponding nearly symmetric hexagon with holes $H'_{{\mathbf{l}},{\mathbf{q}}}(a,b,k)$ for $a=7$, $b=8$, $k=3$,  ${\mathbf{l}}=(2,4)$ and ${\mathbf{q}}=(3,5)$.}
\vskip-0.1in
\label{fdc}
\end{figure}

The last pair of families of regions is very similar to the previous one; the only difference is that the triangular hole of odd side-length is chosen to point {\it downward}\footnote{ Note that such a region cannot be obtained by rotating an $H_{{\mathbf{l}},{\mathbf{q}}}(a,b,k)$ by $180^{\circ}$, because we are assuming that $k$, which is the difference between the number of up-pointing and down-pointing unit triangles in $H(a,b,k)$, is non-negative.}. We denote the resulting regions by $\bar{H}_{{\mathbf{l}},{\mathbf{q}}}(a,b,k)$ and $\bar{H}'_{{\mathbf{l}},{\mathbf{q}}}(a,b,k)$, respectively (these are illustrated in Figure \ref{fdd}).

\begin{figure}[t]
\centerline{
\hfill
{\includegraphics[width=0.40\textwidth]{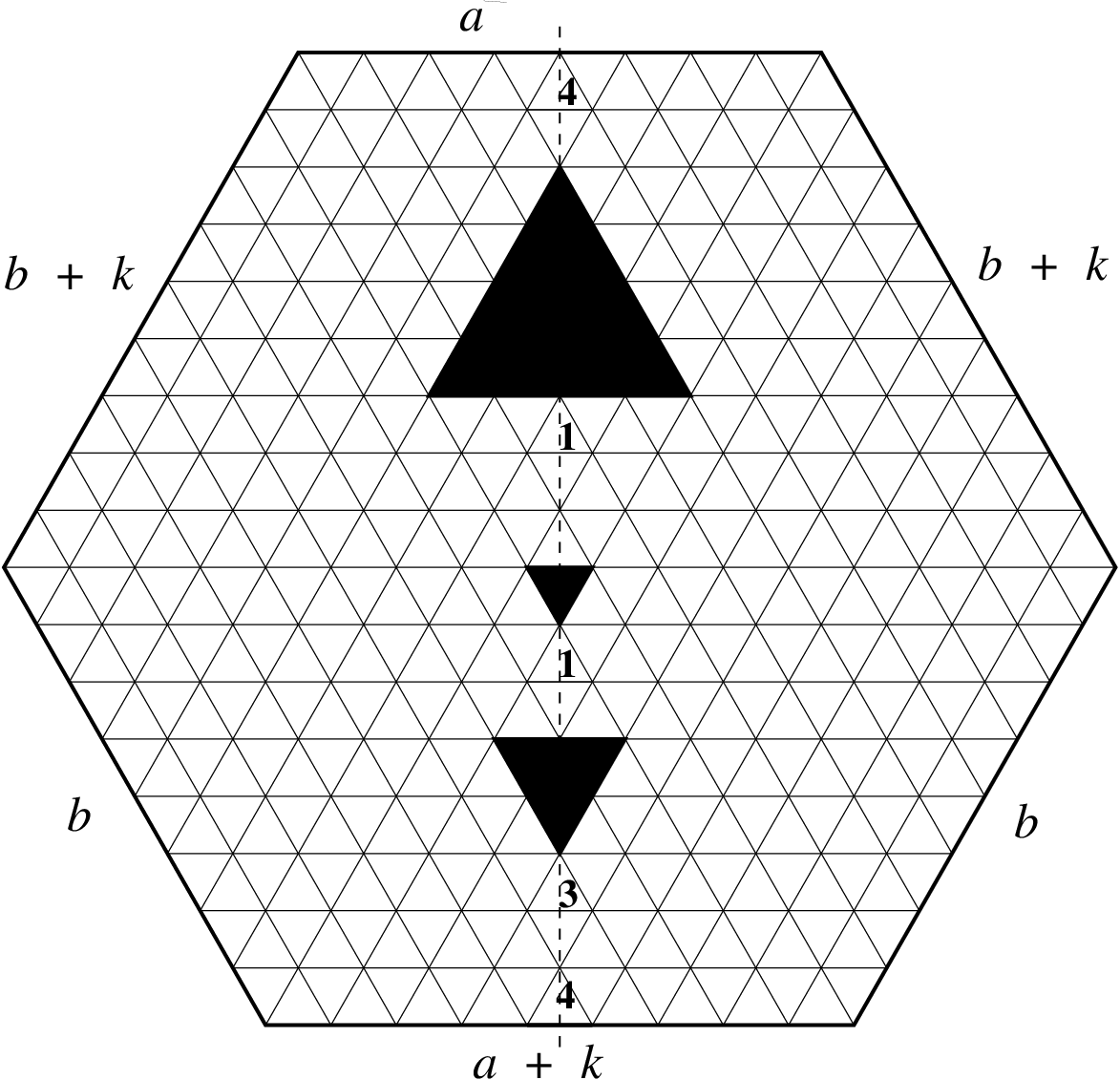}}
\hfill
{\includegraphics[width=0.40\textwidth]{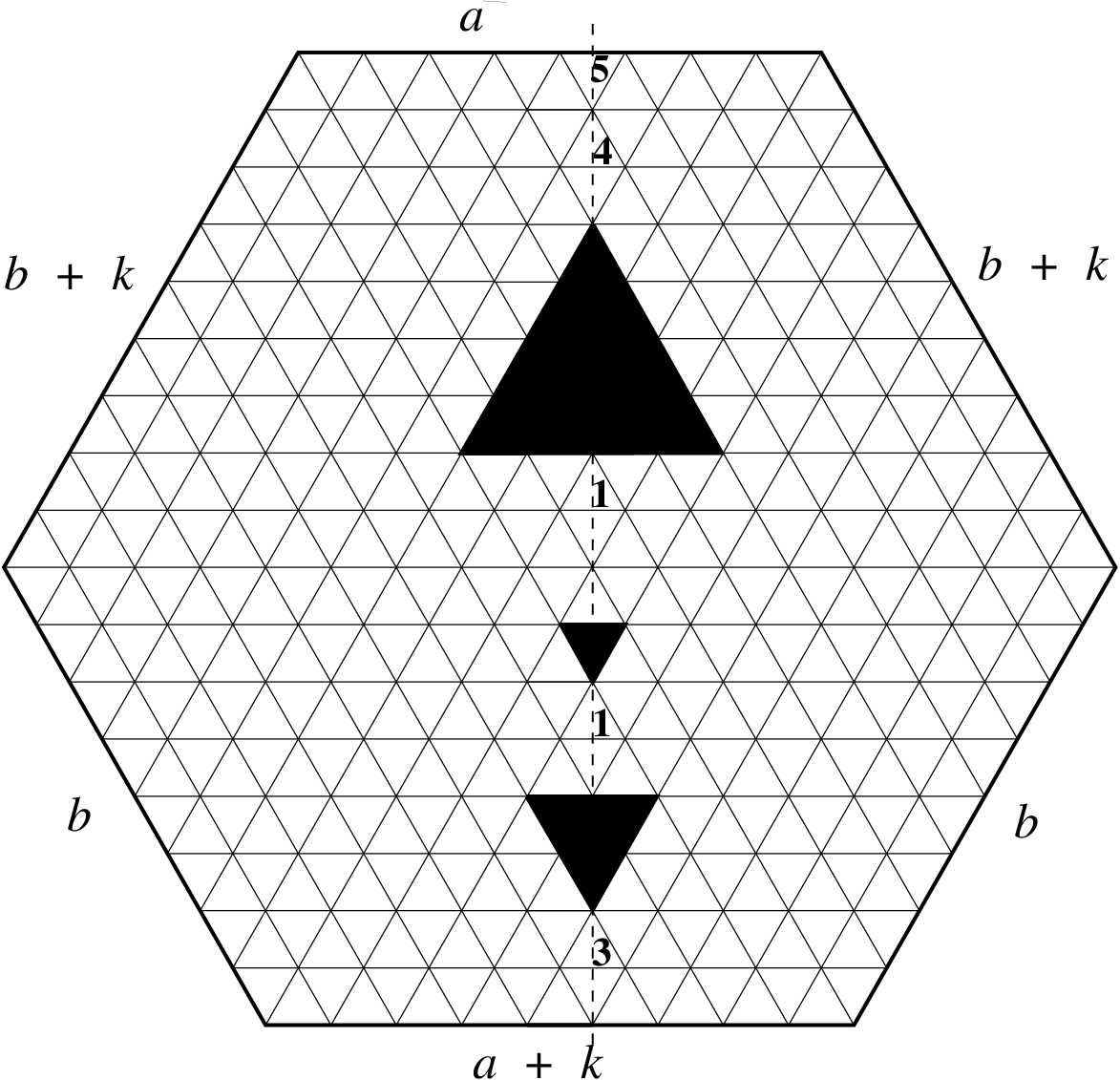}}
\hfill
}
\vskip0.1in
\caption{{\it Left.} The symmetric hexagon with holes $\bar{H}_{{\mathbf{l}},{\mathbf{q}}}(a,b,k)$ for $a=b=8$, $k=1$, ${\mathbf{l}}=(1,3,4)$ and ${\mathbf{q}}=(1,4)$. {\it Right.} The corresponding nearly symmetric hexagon with holes $\bar{H}'_{{\mathbf{l}},{\mathbf{q}}}(a,b,k)$ for $a=b=8$, $k=1$,  ${\mathbf{l}}=(1,3)$ and ${\mathbf{q}}=(1,4,5)$.}
\vskip-0.1in
\label{fdd}
\end{figure}

\subsection{Two families of dented regions}
\label{subsection4.2}
We recall here the definitions of two more families of regions whose lozenge tilings were enumerated in \cite{CiucuPP1}; they will be essential for solving Lai's open problem.
The detailed definitions are given in \cite{CiucuPP1}; we rely here on Figures \ref{fde}--\ref{fdh} to define our regions.

\begin{figure}[t]
\centerline{
\hfill
{\includegraphics[width=0.25\textwidth]{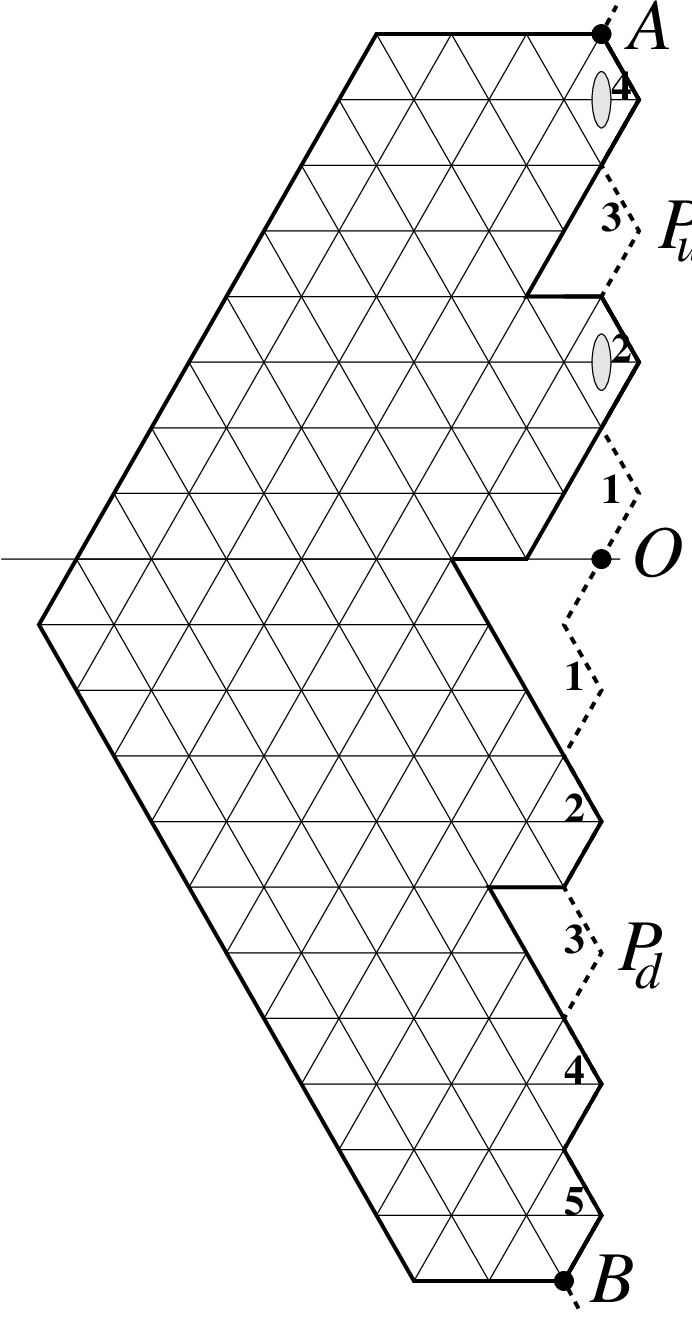}}
\hfill
}
\vskip0.1in
\caption{$R_{(2,4,5),(2,4)}(2)$}
\vskip-0.1in
\label{fde}
%
\centerline{
\hfill
{\includegraphics[width=0.20\textwidth]{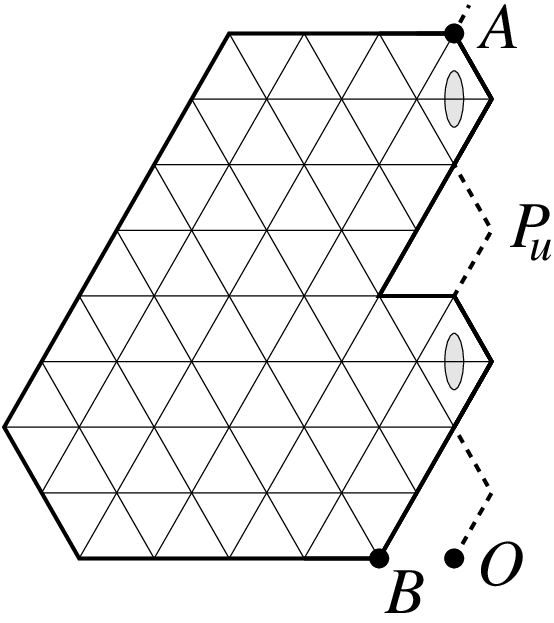}}
\hfill
{\includegraphics[width=0.20\textwidth]{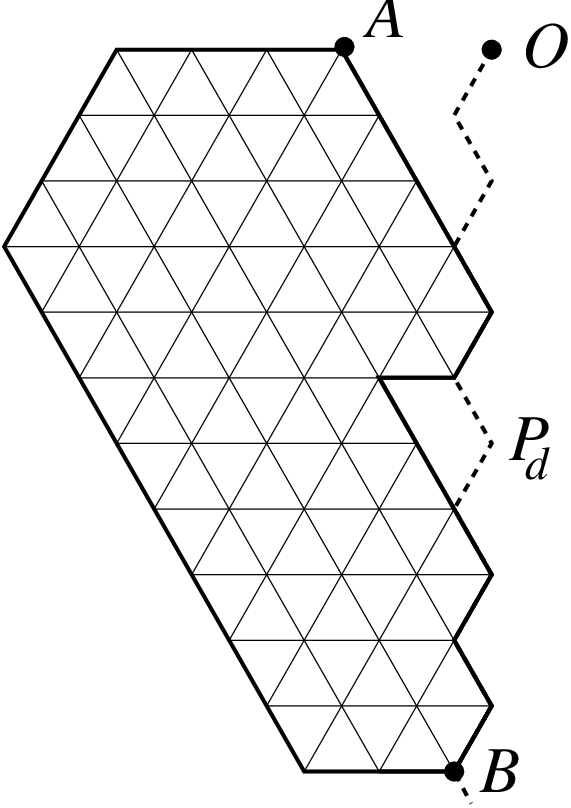}}
\hfill
}
\vskip0.1in
\caption{$R_{\emptyset,(2,4)}(4)$ (left) and $R_{(2,4,5),\emptyset}(2)$ (right).}
\vskip-0.1in
\label{fdf}
\end{figure}

Consider two semi-infinite vertical zigzag paths $P_u$ and $P_d$, meeting at the reference point $O$ as shown in Figure \ref{fde} (in that figure, $P_u$ connects $O$ to $A$, while $P_b$ connects $O$ to $B$). Label the slots that fit in their folds (call these bumps) as shown in the figure. The tile positions corresponding to the bumps above $O$ are weighted by $1/2$ (this is indicated by shaded ellipses in the figures).

Given non-empty lists of strictly increasing positive integers ${\mathbf{l}}=(l_1,\dotsc,l_m)$ and ${\mathbf{q}}=(q_1,\dotsc,q_n)$ and a non-negative integer $x$, the region $R_{{\mathbf{l}},{\mathbf{q}}}(x)$ is defined as indicated in Figure \ref{fde}. The lists ${\mathbf{l}}$ and ${\mathbf{q}}$ specify which bumps on the zigzag lines are kept. The kept bumps on $P_u$ are joined together by creating up-pointing dents, and those on $P_d$ by down-pointing dents; these two portions are then joined together via a horizontal ray left of $O$, as shown; $x$ is the length of the base.
Note that this information, together with the fact that we want a region which can be tiled by lozenges, determines the rest of the boundary of $R_{{\mathbf{l}},{\mathbf{q}}}(x)$\footnote{ Indeed, if a lozenge tiling exists, due to a family of paths of lozenges that it determines, the length of the southwest side must be equal to the number of unit segments facing northeast on the right boundary.}.

In case one of the lists ${\mathbf{l}}$ or ${\mathbf{q}}$ is empty, the definition is slightly different, as shown in Figure \ref{fdf}; note in particular that if ${\mathbf{l}}=\emptyset$, $x$ is not the length of the base, but one unit less than the distance from the left end of the base to the reference point $O$. Define $R_{\emptyset,\emptyset}(x)=\emptyset$ for all $x$.

\begin{figure}[t]
\centerline{
\hfill
{\includegraphics[width=0.25\textwidth]{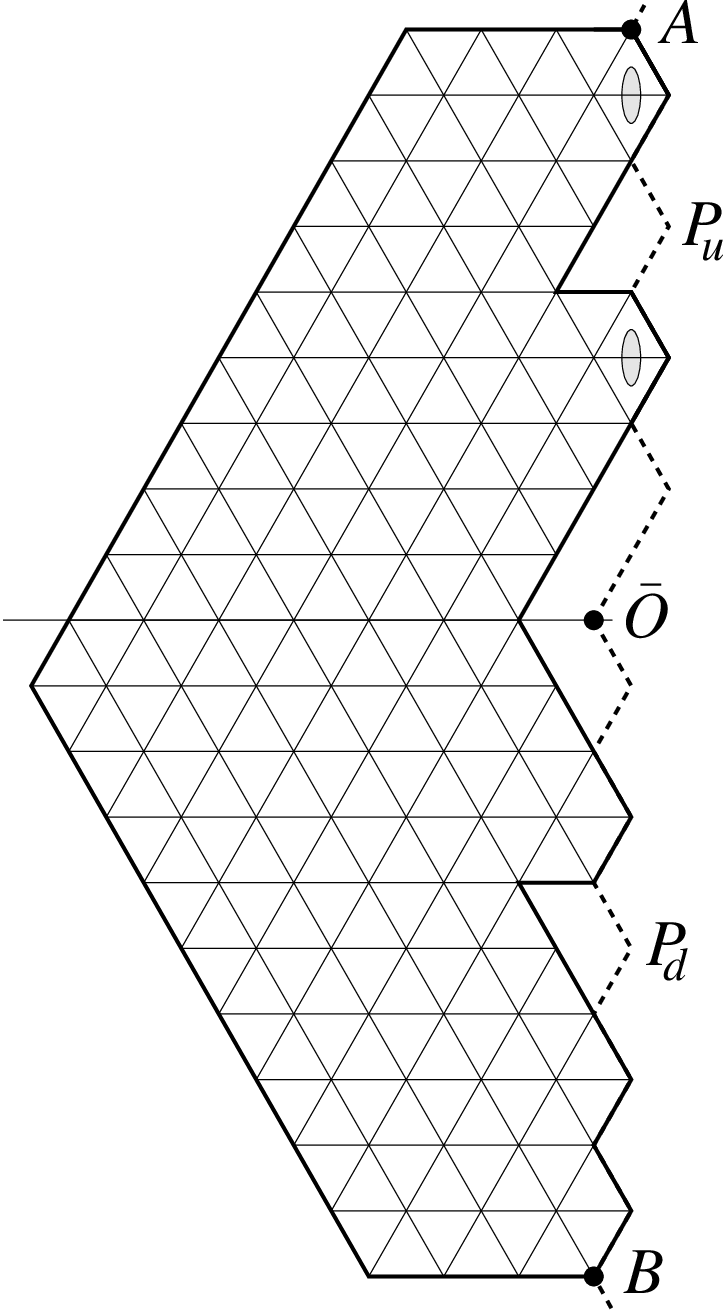}}
\hfill
}
\vskip0.1in
\caption{$\bar{R}_{(2,4,5),(2,4)}(3)$}
\vskip-0.1in
\label{fdg}
%
\centerline{
\hfill
{\includegraphics[width=0.20\textwidth]{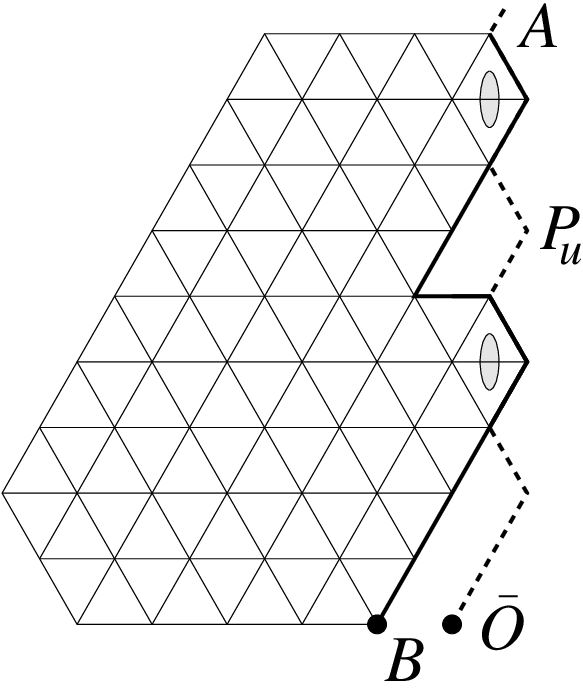}}
\hfill
{\includegraphics[width=0.20\textwidth]{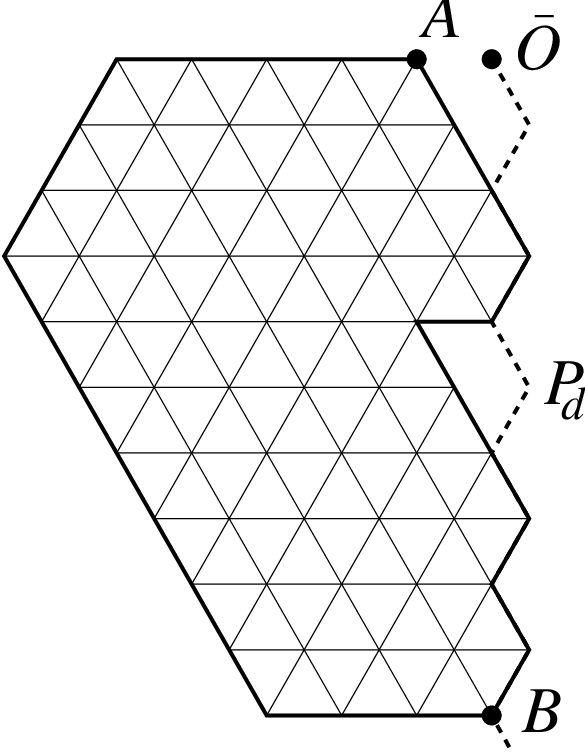}}
\hfill
}
\vskip0.1in
\caption{$\bar{R}_{\emptyset,(2,4)}(5)$ (left) and $\bar{R}_{(2,4,5),\emptyset}(3)$ (right).}
\vskip-0.1in
\label{fdh}
\end{figure}

Our second family of regions, denoted $\bar{R}_{{\mathbf{l}},{\mathbf{q}}}(x)$, is defined
almost identically (see Figures \ref{fdg} and \ref{fdh}). The only difference is that, when connecting the upper and lower boundary portions determined by the selected bumps,
instead of the horizontal ray starting from $O$, we use the
one starting from the lattice point $\bar{O}$, one step southwest of $O$ (and, in
case ${\mathbf{l}}$ is empty, when defining the bottom side of 
$\bar{R}_{\emptyset,{\mathbf{q}}}(x)$, we move left until we are $x$ units away from 
$\bar{O}$).
Again, we define $\bar{R}_{\emptyset,\emptyset}(x)=\emptyset$ for all $x$.

\subsection{Two families of polynomials and their connection to the dented regions}
\label{subsection4.3}
Given integers $m,n\geq0$ and lists ${\mathbf l}=(l_1,\ldots,l_m)$ and ${\mathbf q}=(q_1,\ldots,q_n)$ of strictly increasing positive integers, we define the polynomials $P_{\mathbf{l},\mathbf{q}}(x)$ and $\overline{P}_{\mathbf{l},\mathbf{q}}(x)$ as follows.

We first define the monic polynomials $B_{m,n}(x)$ and $\bar{B}_{m,n}(x)$.
Recall that for any
real number~$a$ and non-negative integer $k$, the  Pochhammer symbol
$(a)_k$ is defined by  setting $(a)_{0}\coloneqq1$, and $(a)_k\coloneqq \prod_{i=1}^{k}(a+i-1)$ for $k>0$. In addition, for $a$ and $k$ like above, we will find it convenient to use the notation

\begin{equation*}
\begin{aligned}
    \langle a,a+k\rangle&\coloneqq
      \displaystyle\prod_{i=0}^{k}(a+i)^{\text{min}(i+1,k+1-i)}
      \\
   &=a(a+1)^{2}\cdots(a+k-1)^{2}(a+k)\\
   &=\prod_{i=0}^{\lfloor\frac{k}{2}\rfloor}(a+i)_{k+1-2i},\ \ \ k\geq0,
\end{aligned}
\end{equation*}
and  $\langle a,a+k\rangle:=1$ for $k<0$.

With this notation, for non-negative integers $m$ and $n$, $B_{m,n}(x)$ and $\bar{B}_{m,n}(x)$ are defined to be the monic polynomials given by
\begin{equation*}
\begin{aligned}
    B_{m,n}(x)= ~&2^{-mn-m(m-1)/2}(x+n+1)_{m}(x+n+2)_{m}\langle x+2,x+n\rangle\left\langle x+\frac{3}{2},x+\frac{2n+1}{2}\right\rangle\\
    &\ \ \ \ \ \ \ \ \ \ \ \ \ \ \ \ \times\prod_{i=1}^{n}\frac{(x+i)_{m}}{(x+i+1/2)_{m}}\prod_{i=1}^{m}(2x+n+i+2)_{n+i-1}
\end{aligned}
\end{equation*}
and
\begin{equation*}
\begin{aligned}
    \bar{B}_{m,n}(x)= ~&2^{-mn-n(n+1)/2}(x+m+1)_{n}\langle x+1,x+m\rangle\left\langle x+\frac{3}{2},x+\frac{2m-1}{2}\right\rangle\\
    &\ \ \ \ \ \ \ \ \ \ \ \ \ \ \ \ \times\prod_{i=1}^{m}\frac{(x+i)_{n}}{(x+i+1/2)_{n}}\prod_{i=1}^{n}(2x+m+i+1)_{m+i}.
\end{aligned}
\end{equation*}
Next, given lists of strictly increasing positive integers ${\mathbf l}=(l_1,\ldots,l_m)$ and ${\mathbf q}=(q_1,\ldots,q_n)$, we define the constants $c_{\mathbf{l},\mathbf{q}}$ and $\bar{c}_{\mathbf{l},\mathbf{q}}$ by
\begin{equation*}
    c_{\mathbf{l},\mathbf{q}}=2^{\binom{n-m}{2}-m}\prod_{i=1}^{m}\frac{1}{(2l_{i})!}\prod_{i=1}^{n}\frac{1}{(2q_{i}-1)!}\frac{\prod_{1\leq i<j\leq m}(l_{j}-l_{i})\prod_{1\leq i<j\leq n}(q_{j}-q_{i})}{\prod_{i=1}^{m}\prod_{j=1}^{n}(l_{i}+q_{j})}
\end{equation*}
and
\begin{equation*}
    \bar{c}_{\mathbf{l},\mathbf{q}}=2^{\binom{n-m}{2}-m}\prod_{i=1}^{m}\frac{1}{(2l_{i}-1)!}\prod_{i=1}^{n}\frac{1}{(2q_{i})!}\frac{\prod_{1\leq i<j\leq m}(l_{j}-l_{i})\prod_{1\leq i<j\leq n}(q_{j}-q_{i})}{\prod_{i=1}^{m}\prod_{j=1}^{n}(l_{i}+q_{j})}.
\end{equation*}
We are now ready to define our polynomials $P_{\mathbf{l},\mathbf{q}}(x)$ and $\bar{P}_{\mathbf{l},\mathbf{q}}(x)$.
For lists of strictly increasing positive integers ${\mathbf l}=(l_1,\ldots,l_m)$ and ${\mathbf q}=(q_1,\ldots,q_n)$,
$P_{\mathbf{l},\mathbf{q}}(x)$ and $\bar{P}_{\mathbf{l},\mathbf{q}}(x)$ are given by
\footnote{These are equations (5.1) and (5.2) in \cite[Part B]{CiucuPP1}.}\footnote{ For $m=0$, one needs to take $l_m=0$ in equations \eqref{eda} and \eqref{edb}; also, when $m$ or $n$ is zero at the upper limit of an outside product, that double product is defined to be 1.}
\begin{equation}
\label{eda}
\begin{aligned}
    P_{\mathbf{l},\mathbf{q}}(x)=c_{\mathbf{l},\mathbf{q}}B_{m,n}(x+l_{m}-m)&\prod_{i=1}^{m}\prod_{j=i}^{l_{i}-1}(x+l_{m}-j)(x+l_{m}-m+n+j+2)\\
    \times&\prod_{i=1}^{n}\prod_{j=i}^{q_{i}-1}(x+l_{m}-m+n-j+1)(x+l_{m}+j+1)
\end{aligned}
\end{equation}
and
\begin{equation}
\label{edb}
\begin{aligned}
    \bar{P}_{\mathbf{l},\mathbf{q}}(x)=\bar{c}_{\mathbf{l},\mathbf{q}}\bar{B}_{m,n}(x+l_{m}-m)&\prod_{i=1}^{m}\prod_{j=i}^{l_{i}-1}(x+l_{m}-j)(x+l_{m}-m+n+j+1)\\
    \times&\prod_{i=1}^{n}\prod_{j=i}^{q_{i}-1}(x+l_{m}-m+n-j)(x+l_{m}+j+1).
\end{aligned}    
\end{equation}

In \cite{CiucuPP1}, the second author showed that the tiling generating functions of the regions $R_{{\mathbf{l}},{\mathbf{q}}}(x)$ and $\bar{R}_{{\mathbf{l}},{\mathbf{q}}}(x)$ introduced in subsection \ref{subsection4.2} are given by the above polynomials $P_{\mathbf{l},\mathbf{q}}(x)$ and $\bar{P}_{\mathbf{l},\mathbf{q}}(x)$.

\begin{thm} [Ciucu \cite{CiucuPP1}]
\label{propositionpp1} We have
\begin{equation}
\label{edc}
    \M(R_{{\mathbf{l}},{\mathbf{q}}}(x))=P_{\mathbf{l},\mathbf{q}}(x)
\end{equation}
and
\begin{equation}
\label{edd}
    \M(\bar{R}_{{\mathbf{l}},{\mathbf{q}}}(x))=\bar{P}_{\mathbf{l},\mathbf{q}}(x).
\end{equation}

\end{thm}

\subsection{Lozenge tilings of symmetric hexagons with holes along the symmetry axis.}
\label{subsection4.4}


The number of lozenge tilings of the regions $H_{\mathbf{l}}(a,b,k)$, $H_{\mathbf{l},\mathbf{q}}(a,b,k)$, and $\bar{H}_{\mathbf{l},\mathbf{q}}(a,b,k)$ (three of the six families of regions introduced in subsection \ref{subsection4.1}) was determined by the second author in  \cite{CiucuPP1}.
The proof is based on the factorization theorem for matchings \cite[Theorem 1.2]{CiucuMatchingFactorization} and equalities \eqref{edc}--\eqref{edd}. The result (stated in a slightly different form) is the following.

\begin{thm}[Theorem 1.1 in \cite{CiucuPP1}]
\label{tda}
Let $H$ be one of the regions $H_{\mathbf{l}}(a,b,k)$, $H_{{\mathbf{l}},{\mathbf{q}}}(a,b,k)$ or $\bar{H}_{{\mathbf{l}},{\mathbf{q}}}(a,b,k)$. Then we have\footnote{ Here $H^+$ and $H^-$ are the regions obtained from the symmetric region $H$ by the procedure described before the statement of Theorem \ref{tba}.}
\begin{equation}
\M(H)=2^{\w_\ell(H)}\M(H^+)\M(H^-),
\label{ede}
\end{equation}
and all of the regions $H^+$ and $H^-$ are $R$- or $\bar{R}$-regions, as follows\footnote{ Given two regions $R$ and $Q$, $[R,Q]$ is the (ordered) list consisting of these regions.}:

$(${\rm a}$)$. For $k$ even, we have
\begin{equation}
[H_{\mathbf{l}}(a,b,k)^+,H_{\mathbf{l}}(a,b,k)^-]=
\begin{cases}
    [\bar{R}_{{\mathbf{l}}-1,\emptyset}(a/2), R_{\emptyset,{\mathbf{l}}}((a+k-2)/2)], & \text{$a$ even, $l_1=1$,} \\
    [{R}_{{\mathbf{l}}-1,\emptyset}(a/2), R_{\emptyset,{\mathbf{l}}}((a+k-2)/2)], & \text{$a$ even, $l_1>1$,} \\
    [\bar{R}_{{\mathbf{l}},\emptyset}((a-1)/2), R_{\emptyset,{\mathbf{l}}^{(m)}}((a+k-1)/2)], & \text{$a$ odd,}\
\label{edf}
\end{cases}
\end{equation}
where ${\mathbf l}-1$ denotes the list obtained from ${\mathbf l}$ by decrementing each of its elements by one unit $($discarding the first element if $l_1=1$$)$.

$(${\rm b}$)$. For $k$ odd, we have

\begin{equation}
[H_{{\mathbf{l}},{\mathbf{q}}}(a,b,k)^+,H_{{\mathbf{l}},{\mathbf{q}}}(a,b,k)^-]=
\begin{dcases}
    [{R}_{{\mathbf{q}},{\mathbf{l}}^{(m)}}(a/2), \bar{R}_{{\mathbf{l}},{\mathbf{q}}}((a+k-1)/2)], & \text{$a$ even,} \\
    [{R}_{{\mathbf{q}},{\mathbf{l}}}((a-1)/2), \bar{R}_{{\mathbf{l}},{\mathbf{q}}^{(n)}}((a+k)/2)], & \text{$a$ odd,}\
\end{dcases}
\label{edg}
\end{equation}
\ \ \ and
\begin{equation}
[\bar{H}_{{\mathbf{l}},{\mathbf{q}}}(a,b,k)^+,\bar{H}_{{\mathbf{l}},{\mathbf{q}}}(a,b,k)^-]=
\begin{dcases}
    [\bar{R}_{{\mathbf{q}},{\mathbf{l}}^{(m)}}(a/2), {R}_{{\mathbf{l}},{\mathbf{q}}}((a+k-1)/2)], & \text{$a$ even,} \\
    [\bar{R}_{{\mathbf{q}},{\mathbf{l}}}((a-1)/2), {R}_{{\mathbf{l}},{\mathbf{q}}^{(n)}}((a+k)/2)], & \text{$a$ odd,}\
\end{dcases}
\label{edh}
\end{equation}
where ${\mathbf l}^{(k)}$ denotes the list obtained from ${\mathbf l}^{(k)}$ by discarding its $k$th element.

\end{thm}

\subsection{Shifting the holes one unit southeast. 
}
\label{subsection4.5}

The first part of the following result, together with equations \eqref{edc} and \eqref{edd}, provides a simple product formula for the number of lozenge tilings of the region $H'_{\mathbf{l}}(a,b,k)$ --- thus solving Lai's open problem mentioned in subsection \ref{subsection4.1} --- while the second part gives analogous formulas for the regions $H'_{\mathbf{l},\mathbf{q}}(a,b,k)$ and $\bar{H}'_{\mathbf{l},\mathbf{q}}(a,b,k)$.

\begin{thm}
\label{tdb}
Let $H$ be one of the regions $H'_{\mathbf{l}}(a,b,k)$, $H'_{{\mathbf{l}},{\mathbf{q}}}(a,b,k)$ or $\bar{H}'_{{\mathbf{l}},{\mathbf{q}}}(a,b,k)$, and let $\ell'$ be the axis with respect to which the holes in $H$ are symmetric. Then we have\footnote{ Here ${H}^+$ and ${H}^-$ are 
the ``half-regions'' obtained from the nearly symmetric hexagon with holes $H$ by the procedure described before the statement of Theorem \ref{tbb} (see the middle picture in Figure \ref{fdi} for an example).}
\begin{equation}
\M(H)=
\begin{cases}
\dfrac{a+b+k}{a+2b+k}\cdot 2^{\w_{\ell'}(H)}\M(H^+)\M(H^-), & \text{$a$ odd},\\
\\[-10pt]
\dfrac{a+b+k}{a}\cdot 2^{\w_{\ell'}(H)}\M(H^+)\M(H^-), & \text{$a$ even},\
\end{cases}
\label{edi}
\end{equation}
and all of the regions $H^+$ and $H^-$ are $R$- or $\bar{R}$-regions, as follows:

$(${\rm a}$)$. For $k$ even, we have

\begin{equation}
[H'_{\mathbf{l}}(a,b,k)^+,H'_{\mathbf{l}}(a,b,k)^-]=
\begin{cases}
    [\bar{R}_{{\mathbf{l}}-1,\emptyset}((a-1)/2), R_{\emptyset,{\mathbf{l}}}((a+k-1)/2)], & \text{$a$ odd, $l_1=1$,} \\
    [R_{{\mathbf{l}}-1,\emptyset}((a-1)/2), R_{\emptyset,{\mathbf{l}}}((a+k-1)/2)], & \text{$a$ odd, $l_1>1$,} \\
    [\bar{R}_{{\mathbf{l}},\emptyset}((a-2)/2), R_{\emptyset,{\mathbf{l}}^{(m)}}((a+k)/2)], & \text{$a$ even.}\
\label{edj}
\end{cases}
\end{equation}

$(${\rm b}$)$. For $k$ odd, we have

\begin{equation}
[H'_{{\mathbf{l}},{\mathbf{q}}}(a,b,k)^+,H'_{{\mathbf{l}},{\mathbf{q}}}(a,b,k)^-]=
\begin{dcases}
    [{R}_{{\mathbf{q}},{\mathbf{l}}^{(m)}}((a-1)/2), \bar{R}_{{\mathbf{l}},{\mathbf{q}}}((a+k)/2)], & \text{$a$ odd,} \\
    [{R}_{{\mathbf{q}},{\mathbf{l}}}((a-2)/2), \bar{R}_{{\mathbf{l}},{\mathbf{q}}^{(n)}}((a+k+1)/2)], & \text{$a$ even,}\
\end{dcases}
\label{edk}
\end{equation}
\ \ \ and
\begin{equation}
[\bar{H}'_{{\mathbf{l}},{\mathbf{q}}}(a,b,k)^+,\bar{H}'_{{\mathbf{l}},{\mathbf{q}}}(a,b,k)^-]=
\begin{dcases}
    [\bar{R}_{{\mathbf{q}},{\mathbf{l}}^{(m)}}((a-1)/2), {R}_{{\mathbf{l}},{\mathbf{q}}}((a+k)/2)], & \text{$a$ odd,} \\
    [\bar{R}_{{\mathbf{q}},{\mathbf{l}}}((a-2)/2), {R}_{{\mathbf{l}},{\mathbf{q}}^{(n)}}((a+k+1)/2)], & \text{$a$ even.}\
\end{dcases}
\label{edl}
\end{equation}

\end{thm}
%


One may wonder why the prefactors in \eqref{edi} have the particular fraction form indicated.
Our proof will show that we can
in fact rephrase the theorem in an alternate form as follows.

\begin{thm}
\label{tdc}
Let $H$ be one of the regions $H'_{\mathbf{l}}(a,b,k)$, $H'_{{\mathbf{l}},{\mathbf{q}}}(a,b,k)$ or $\bar{H}'_{{\mathbf{l}},{\mathbf{q}}}(a,b,k)$, and let $\ell'$ be the axis with respect to which the holes in $H$ are symmetric.
Then we have\footnote{ Here $\widehat{H}^+$ and $\widehat{H}^-$ are 
the ``half-regions'' obtained from the nearly symmetric hexagon with holes $H$ by the procedure described before the statement of Theorem \ref{tbb} (see the middle picture in Figure \ref{fdj} for an example).}
\begin{equation}
\M(H)=
\begin{cases}
\dfrac{a+b+k}{a+k}\cdot 2^{\w_{\ell'}(H)}\M(\widehat{H}^+)\M(\widehat{H}^-), & \text{$a$ odd},\\
\\[-10pt]
\dfrac{a+b+k}{a+2b+2k}\cdot 2^{\w_{\ell'}(H)}\M(\widehat{H}^+)\M(\widehat{H}^-), & \text{$a$ even},\
\end{cases}
\label{edm}
\end{equation}
and all of the regions $\widehat{H}_{\ell'}^+$ and $\widehat{H}_{\ell'}^-$ are $R$- or $\bar{R}$-regions, as follows:

\vskip0.1in
$(${\rm a}$)$. For $k$ even, we have

\begin{equation}
[\widehat{H'}_{\mathbf{l}}(a,b,k)^+,\widehat{H'}_{\mathbf{l}}(a,b,k)^-]=
\begin{cases}
    [R_{\emptyset,{\mathbf{l}}}((a+k-3)/2), \bar{R}_{{\mathbf{l}}-1,\emptyset}((a+1)/2)], & \text{$a$ odd, $l_1=1$,} \\
    [R_{\emptyset,{\mathbf{l}}}((a+k-3)/2), R_{{\mathbf{l}}-1,\emptyset}((a+1)/2)], & \text{$a$ odd, $l_1>1$,} \\
    [R_{\emptyset,{\mathbf{l}}^{(m)}}((a+k-2)/2), \bar{R}_{{\mathbf{l}},\emptyset}(a/2)], & \text{$a$ even.}\
\label{edn}
\end{cases}
\end{equation}

$(${\rm b}$)$. For $k$ odd, we have

\begin{equation}
[\widehat{H'}_{{\mathbf{l}},{\mathbf{q}}}(a,b,k)^+,\widehat{H'}_{{\mathbf{l}},{\mathbf{q}}}(a,b,k)^-]=
\begin{dcases}
    [\bar{R}_{{\mathbf{l}},{\mathbf{q}}}((a+k-2)/2), {R}_{{\mathbf{q}},{\mathbf{l}}^{(m)}}((a+1)/2)], & \text{$a$ odd,} \\
    [\bar{R}_{{\mathbf{l}},{\mathbf{q}}^{(n)}}((a+k-1)/2), {R}_{{\mathbf{q}},{\mathbf{l}}}(a/2)], & \text{$a$ even,}\
\end{dcases}
\label{edo}
\end{equation}
\ \ \ and
\begin{equation}
[\widehat{\bar{H}'}_{{\mathbf{l}},{\mathbf{q}}}(a,b,k)^+,\widehat{\bar{H}'}_{{\mathbf{l}},{\mathbf{q}}}(a,b,k)^-]=
\begin{dcases}
    [{R}_{{\mathbf{l}},{\mathbf{q}}}((a+k-2)/2), \bar{R}_{{\mathbf{q}},{\mathbf{l}}^{(m)}}((a+1)/2)], & \text{$a$ odd,} \\
    [{R}_{{\mathbf{l}},{\mathbf{q}}^{(n)}}((a+k-1)/2), \bar{R}_{{\mathbf{q}},{\mathbf{l}}}(a/2)], & \text{$a$ even.}\
\end{dcases}
\label{edp}
\end{equation}

\end{thm}

{\it Remark $3$.} The most striking part of Theorem \ref{tdb} is that the number of lozenge tilings is {\it almost} as if given by the factorization theorem --- although the factorization theorem does not apply, as the regions are not symmetric! More precisely, the number of lozenge tilings is obtained by multiplying the quantity on the right hand side of the factorization theorem (which makes sense, with the definition of the regions $H^+$ and  $H^-$ given before the statement of Theorem \ref{tba}) by one of the simple fractions $\frac{a+b+k}{a+2b+k}$ or $\frac{a+b+k}{a}$, according as $a$ is odd or even, respectively; a similar statement holds for Theorem \ref{tdc}. This phenomenon seems to be essentially dependent on the structure of our holes. Even for a very simple change in their definition (for instance, having just two oppositely oriented triangular holes of side two), numerical data strongly indicates that such a simple relationship does not hold anymore. This also invites one to find a more direct proof.

\vskip0.1in
It turns out that equations \eqref{edi} and \eqref{edm} can be combined to obtained the following result, which expresses the number of tilings of the regions $H'_{\mathbf{l}}(a,b,k)$, $H'_{{\mathbf{l}},{\mathbf{q}}}(a,b,k)$ or $\bar{H}'_{{\mathbf{l}},{\mathbf{q}}}(a,b,k)$ directly
in terms 
of the number of tilings of two very closely related regions of the corresponding unprimed type.

\begin{thm}
\label{tdd}
Consider one of the regions $H'_{\mathbf{l}}(a,b,k)$, $H'_{{\mathbf{l}},{\mathbf{q}}}(a,b,k)$ or $\bar{H}'_{{\mathbf{l}},{\mathbf{q}}}(a,b,k)$, and denote it for short by $H'(a)$. Denote by $H(a)$ the corresponding unprimed region of the same parameters $($the number of lozenge tilings of $H(a)$ is given by Theorem $\ref{tda}$ and the formulas in Theorem $\ref{propositionpp1}$$)$. Then we have
\begin{equation}
\M(H'(a))=
\begin{dcases}
\frac{a+b+k}{\sqrt{(a+k)(a+2b+k)}}\sqrt{\M(H(a-1))\M(H(a+1))},  & \text{$a$ odd,} \\
\frac{a+b+k}{\sqrt{a(a+2b+2k)}}\sqrt{\M(H(a-1))\M(H(a+1))},  & \text{$a$ even.}
\end{dcases}
\label{edpp}
\end{equation}
\end{thm}

{\it Remark $4$.} The above result has the following pleasing geometrical interpretation. Suppose we place a two-sided mirror along the axis $\ell'$ with respect to which the holes in $H'(a)$ are symmetric. Consider the region seen in the mirror from its left (this is precisely $H(a+1)$), and also the region seen in the mirror from its right (which is precisely $H(a-1)$). Then, up to a simple factor which depends on the parity of $a$, $\M(H'(a))$ is just the geometric mean of the numbers of tilings of these two regions.

\vskip0.1in
Our proof of Theorems \ref{tdb} and \ref{tdc} consists of two ingredients: (1) the lemma below, which can be obtained directly from the definitions of the polynomials $P_{\mathbf{l},\mathbf{q}}(x)$ and $\bar{P}_{\mathbf{l},\mathbf{q}}(x)$, and (2) Theorem~\ref{tbb}.

\begin{lemma}
\label{ldd}
    For lists of strictly increasing positive integers ${\mathbf l}=(l_1,\ldots,l_m)$ and ${\mathbf q}=(q_1,\ldots,q_n)$, we have\footnote{ When $m=0$ or $n=0$, the same conventions apply as for equations \eqref{eda} and \eqref{edb}.}
    \begin{equation}
    \label{edq}
    \begin{aligned}
        \frac{P_{\mathbf{l},\mathbf{q}}(x+1)}{P_{\mathbf{l},\mathbf{q}}(x)}=~&\frac{(2x+2l_{m}+2)!(2x+2l_{m}-2m+2n+3)!}{(2x+2l_{m}-m+n+2)!(2x+2l_{m}-m+n+3)!}\\
        &\cdot\prod_{i=1}^{m}\Bigg[\frac{x+l_{m}+l_{i}-m+n+2}{x+l_{m}-l_{i}+1}\Bigg]\prod_{i=1}^{n}\Bigg[\frac{x+l_{m}+q_{i}+1}{x+{l_{m}-q_{i}-m+n+2}}\Bigg]
    \end{aligned}
    \end{equation}
    and
    \begin{equation}
    \label{edr}
    \begin{aligned}
        \frac{\bar{P}_{\mathbf{l},\mathbf{q}}(x+1)}{\bar{P}_{\mathbf{l},\mathbf{q}}(x)}=~&\frac{(2x+2l_{m}+1)!(2x+2l_{m}-2m+2n+2)!}{(2x+2l_{m}-m+n+1)!(2x+2l_{m}-m+n+2)!}\\
        &\cdot\prod_{i=1}^{m}\Bigg[\frac{x+l_{m}+l_{i}-m+n+1}{x+l_{m}-l_{i}+1}\Bigg]\prod_{i=1}^{n}\Bigg[\frac{x+l_{m}+q_{i}+1}{x+l_{m}-q_{i}-m+n+1}\Bigg].
    \end{aligned}
    \end{equation}
\end{lemma}
\begin{proof}
    We present the proof for \eqref{edq}; the proof of \eqref{edr} is completely analogous. From the definition of $P_{\mathbf{l},\mathbf{q}}(x)$,
    \begin{equation}
    \label{eds}
    \begin{aligned}
        \frac{P_{\mathbf{l},\mathbf{q}}(x+1)}{P_{\mathbf{l},\mathbf{q}}(x)}=~&\frac{c_{\mathbf{l},\mathbf{q}}B_{m,n}(x+l_{m}-m+1)}{c_{\mathbf{l},\mathbf{q}}B_{m,n}(x+l_{m}-m)}\cdot\frac{\prod_{i=1}^{m}\prod_{j=i}^{l_{i}-1}(x+l_{m}-j+1)(x+l_{m}-m+n+j+3)}{\prod_{i=1}^{m}\prod_{j=i}^{l_{i}-1}(x+l_{m}-j)(x+l_{m}-m+n+j+2)}\\
        &\cdot\frac{\prod_{i=1}^{n}\prod_{j=i}^{q_{i}-1}(x+l_{m}-m+n-j+2)(x+l_{m}+j+2)}{\prod_{i=1}^{n}\prod_{j=i}^{q_{i}-1}(x+l_{m}-m+n-j+1)(x+l_{m}+j+1)}\\
        =~&\frac{B_{m,n}(x+l_{m}-m+1)}{B_{m,n}(x+l_{m}-m)}\cdot\prod_{i=1}^{m}\Bigg[\frac{(x+l_{m}-i+1)(x+l_{m}-m+n+l_{i}+2)}{(x+l_{m}-l_{i}+1)(x+l_{m}-m+n+i+2)}\Bigg]\\
        &\cdot\prod_{i=1}^{n}\Bigg[\frac{(x+l_{m}-m+n-i+2)(x+l_{m}+q_{i}+1)}{(x+l_{m}-q_{i}-m+n+2)(x+l_{m}+i+1)}\Bigg]\\
        =~&\frac{B_{m,n}(x+l_{m}-m+1)}{B_{m,n}(x+l_{m}-m)}\cdot\frac{(x+l_{m}-m+1)_{m}(x+l_{m}-m+2)_{n}}{(x+l_{m}-m+n+3)_{m}(x+l_{m}+2)_{n}}\\
        &\cdot\prod_{i=1}^{m}\Bigg[\frac{x+l_{m}+l_{i}-m+n+2}{x+l_{m}-l_{i}+1}\Bigg]\prod_{i=1}^{n}\Bigg[\frac{x+l_{m}+q_{i}+1}{x+{l_{m}-q_{i}-m+n+2}}\Bigg].
    \end{aligned}
    \end{equation}
On the other hand, by the definition of $B_{m,n}(x)$,
\begin{equation}
\label{edt}
\begin{aligned}
    \frac{B_{m,n}(x+1)}{B_{m,n}(x)}= ~&\frac{2^{-mn-m(m-1)/2}(x+n+2)_{m}(x+n+3)_{m}\langle x+3,x+n+1\rangle\langle x+\frac{5}{2},x+\frac{2n+3}{2}\rangle}{2^{-mn-m(m-1)/2}(x+n+1)_{m}(x+n+2)_{m}\langle x+2,x+n\rangle\langle x+\frac{3}{2},x+\frac{2n+1}{2}\rangle}\\
    &\cdot\frac{\prod_{i=1}^{n}\frac{(x+i+1)_{m}}{(x+i+3/2)_{m}}\prod_{i=1}^{m}(2x+n+i+4)_{n+i-1}}{\prod_{i=1}^{n}\frac{(x+i)_{m}}{(x+i+1/2)_{m}}\prod_{i=1}^{m}(2x+n+i+2)_{n+i-1}}\\
    = ~&\frac{(x+n+3)_{m}}{(x+n+1)_{m}}\cdot\frac{\langle x+3,x+n+1\rangle\langle x+\frac{5}{2},x+\frac{2n+3}{2}\rangle}{\langle x+2,x+n\rangle\langle x+\frac{3}{2},x+\frac{2n+1}{2}\rangle}\\
    &\cdot\frac{(x+n+1)_{m}(x+3/2)_{m}}{(x+1)_{m}(x+n+3/2)_{m}}\cdot\frac{(2x+2n+3)_{2m}}{(2x+n+3)_{m}(2x+n+4)_{m}}.
\end{aligned}
\end{equation}
Furthermore\footnote{In the last equality, we use the fact that for any integer $n$, $(\lfloor\frac{n-2}{2}\rfloor+1)+(\lfloor\frac{n-1}{2}\rfloor+1)=n$.},
\begin{equation}
\label{edu}
\begin{aligned}
    \frac{\langle x+3,x+n+1\rangle\langle x+\frac{5}{2},x+\frac{2n+3}{2}\rangle}{\langle x+2,x+n\rangle\langle x+\frac{3}{2},x+\frac{2n+1}{2}\rangle}=~&\frac{\displaystyle\prod_{i=0}^{\lfloor\frac{n-2}{2}\rfloor}(x+3+i)_{n-1-2i}\prod_{i=0}^{\lfloor\frac{n-1}{2}\rfloor}(x+\frac{5}{2}+i)_{n-2i}}{\displaystyle\prod_{i=0}^{\lfloor\frac{n-2}{2}\rfloor}(x+2+i)_{n-1-2i}\prod_{i=0}^{\lfloor\frac{n-1}{2}\rfloor}(x+\frac{3}{2}+i)_{n-2i}}\\
    =&\prod_{i=0}^{\lfloor\frac{n-2}{2}\rfloor}\frac{x+n+1-i}{x+2+i}\prod_{i=0}^{\lfloor\frac{n-1}{2}\rfloor}\frac{x+n+\frac{3}{2}-i}{x+\frac{3}{2}+i}\\
    =&\frac{(2x+n+4)_{n}}{(2x+3)_{n}}.
\end{aligned}
\end{equation}
Combining \eqref{eds}--\eqref{edu} and performing some algebraic manipulation, one obtains \eqref{edq}.
\end{proof}


\begin{proof}[Proof of Theorems $\ref{tdb}$ and $\ref{tdc}$]
 Our proof is based on Theorem \ref{tbb} and Lemma \ref{ldd}.
 The details of the arguments depend on the parities of $k$ and $a$, as well as on whether $l_1=1$ or $l_1>1$. 
 We only present the proof in one of these cases,
 as the proofs of the remaining cases are very similar.

Suppose that $k$ is even, $a$ is odd, and $l_{1}=1$.

Apply the semi-factorization theorem (Theorem \ref{tbb}) to the region $H=H'_{\mathbf{l}}(a,b,k)$. Note that in the case that we are assuming we have (1) $m=b$ and (2) $l_{m}=l_{b}=b+\frac{k}{2}$. These equalities can be obtained by using the fact that the number of up-pointing and down-pointing unit triangles in the region are the same (a necessary condition for the existence of a tiling, as each lozenge covers one unit triangle of each type) or by expressing the same length in two different ways.

This application of Theorem \ref{tbb} involves two different zigzag cuts (as described before the statement of Theorem \ref{tbb}) and produces four half-regions $H^+$, $H^+$, $\widehat{H}^+$ and  $\widehat{H}^-$. It is not hard to see that all these four belong to the two families defined in subsection \ref{subsection4.2}. More precisely, we have
\begin{align*}
H^+&=R_{\emptyset,{\mathbf{l}}}((a+k-1)/2),\\
H^-&=\bar{R}_{{\mathbf{l}}-1,\emptyset}((a-1)/2),\\
\widehat{H}^+&=R_{\emptyset,{\mathbf{l}}}((a+k-3)/2),\\
\widehat{H}^-&=\bar{R}_{{\mathbf{l}}-1,\emptyset}((a+1)/2).
\end{align*}
%

\begin{figure}
\vskip0.2in
\centerline{
\includegraphics[width=1\textwidth]{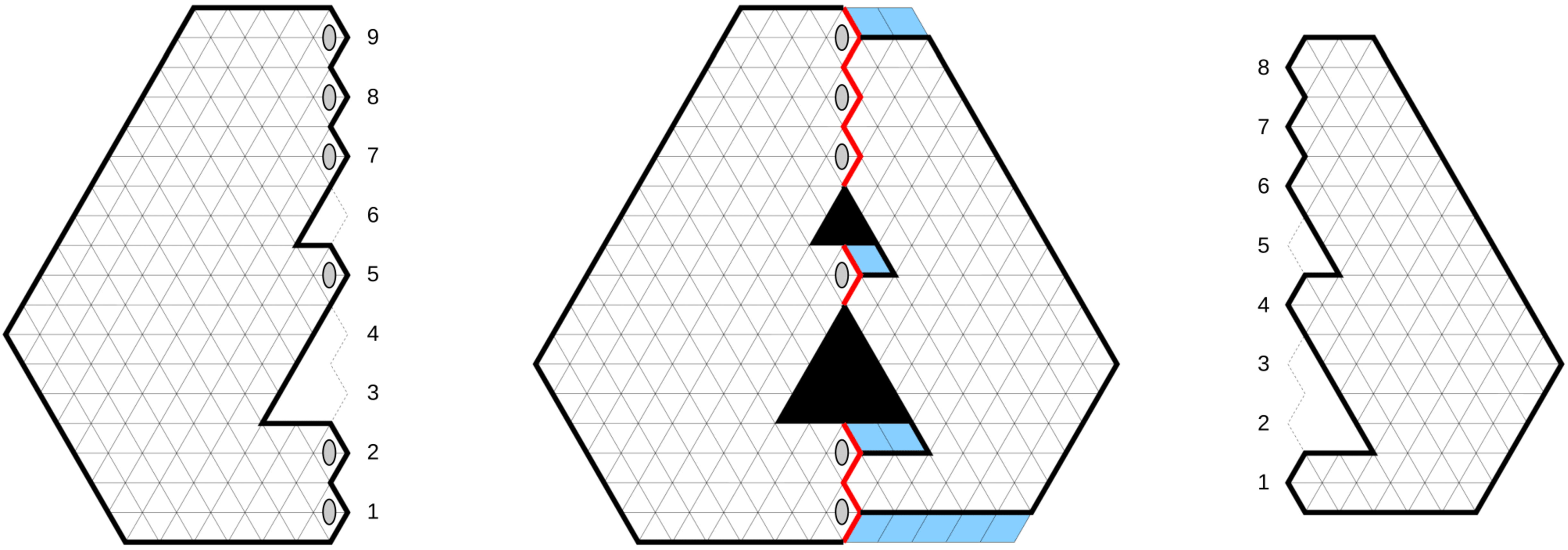}
}
\caption{\label{fdi} Decomposition of $H=H'_{\mathbf{l}}(a,b,k)$ for $a=5$, $b=6$, $k=6$ and ${\mathbf{l}}=(1,2,5,7,8,9)$
  into the half-regions $H^+$ and $H^-$ (see the notation in Theorem \ref{tbb}).
  The left half, $H^+$, is precisely our region $R_{\emptyset,{\mathbf{l}}}((a+k-1)/2)$ (see subsection \ref{subsection4.2}), while the one on the right, $H^-$, is congruent to our region  $\bar{R}_{{\mathbf{l}}-1,\emptyset}((a-1)/2)$.}
\end{figure}

\begin{figure}
\vskip0.2in
\centerline{
\includegraphics[width=1\textwidth]{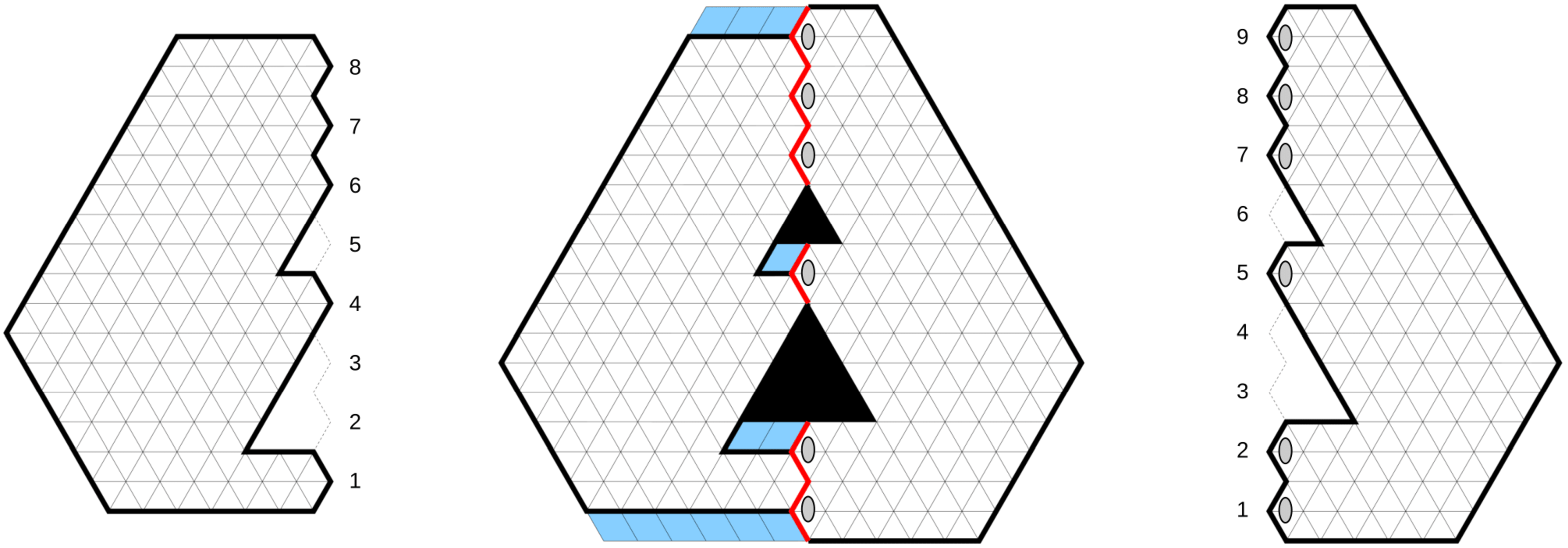}
}
\caption{\label{fdj} Decomposition of $H'_{\mathbf{l}}(a,b,k)$ for $a=5$, $b=6$, $k=6$ and ${\mathbf{l}}=(1,2,5,7,8,9)$ into the half-regions $\widehat{H}^+$ (the left half-region in the center picture) and $\widehat{H}^-$ (the right half-region). We have $\widehat{H}^+=\bar{R}_{{\mathbf{l}}-1,\emptyset}((a+1)/2)$, while $\widehat{H}^-$ is the reflection across the vertical of $R_{\emptyset,{\mathbf{l}}}((a+k-3)/2)$.}
\end{figure}

Thus, by Theorem \ref{tbb}, $\M(H)=\M(H'_{\mathbf{l}}(a,b,k))$ can be expressed as
\begin{equation}
\label{edv}
    2^{\w_{l'}(H)-1}(\M(R_{\emptyset,{\mathbf{l}}}((a+k-1)/2))\M(\bar{R}_{{\mathbf{l}}-1,\emptyset}((a-1)/2))+\M(R_{\emptyset,{\mathbf{l}}}((a+k-3)/2))\M(\bar{R}_{{\mathbf{l}}-1,\emptyset}((a+1)/2))).
\end{equation}
On the other hand, by Lemma \ref{ldd}, we have
\begin{equation}
\begin{aligned}
    \frac{\M(R_{\emptyset,{\mathbf{l}}}((a+k-1)/2))}{\M(R_{\emptyset,{\mathbf{l}}}((a+k-3)/2))}
    =&\frac{P_{\emptyset,{\mathbf{l}}}((a+k-1)/2)}{P_{\emptyset,{\mathbf{l}}}((a+k-3)/2)}\\
    =&\frac{(a+k-1)!(a+2b+k)!}{(a+b+k-1)!(a+b+k)!} \cdot \prod_{i=1}^{b}\Bigg[\frac{(a+k-3)/2+l_{i}+1}{(a+k-3)/2-l_{i}+b+2}\Bigg]\\
    =&\frac{(a+k-1)!(a+2b+k)!}{(a+b+k-1)!(a+b+k)!} \cdot \prod_{i=1}^{b}\Bigg[\frac{a+k+2l_{i}-1}{a+2b+k-2l_{i}+1}\Bigg]
\end{aligned}
\end{equation}
and
\begin{equation}
\begin{aligned}
    \frac{\M(\bar{R}_{{\mathbf{l}}-1,\emptyset}((a+1)/2))}{\M(\bar{R}_{{\mathbf{l}}-1,\emptyset}((a-1)/2)))}=&\frac{\bar{P}_{{\mathbf{l}}-1,\emptyset}((a+1)/2)}{\bar{P}_{{\mathbf{l}}-1,\emptyset}((a-1)/2))}\\
    =&\frac{(a+2l_{b}-2)!(a+2l_{b}-2b+1)!}{(a+2l_{b}-b-1)!(a+2l_{b}-b)!}\prod_{i=1}^{b-1}\Bigg[\frac{(a-1)/2+l_{b}+l_{i+1}-b}{(a-1)/2+l_{b}-l_{i+1}+1}\Bigg]\\
    =&\frac{(a+2l_{b}-2)!(a+2l_{b}-2b+1)!}{(a+2l_{b}-b-1)!(a+2l_{b}-b)!}\prod_{i=1}^{b-1}\Bigg[\frac{a+2l_{b}+2l_{i+1}-2b-1}{a+2l_{b}-2l_{i+1}+1}\Bigg]\\
    =&\frac{(a+2b+k-2)!(a+k+1)!}{(a+b+k-1)!(a+b+k)!}\prod_{i=1}^{b-1}\Bigg[\frac{a+k+2l_{i+1}-1}{a+2b+k-2l_{i+1}+1}\Bigg]
\end{aligned}
\end{equation}
Therefore, if we use the two equalities above and the fact that $l_{1}=1$, we obtain
\begin{equation}
\begin{aligned}
\frac{\M(R_{\emptyset,{\mathbf{l}}}((a+k-1)/2))\M(\bar{R}_{{\mathbf{l}}-1,\emptyset}((a-1)/2))}{\M(R_{\emptyset,{\mathbf{l}}}((a+k-3)/2))\M(\bar{R}_{{\mathbf{l}}-1,\emptyset}((a+1)/2))}&=\frac{(a+k-1)!}{(a+k+1)!}\cdot\frac{(a+2b+k)!}{(a+2b+k-2)!}\cdot\frac{a+k+1}{a+2b+k-1}\\
&=\frac{a+2b+k}{a+k}.
\end{aligned}
\end{equation}
This implies that, if we factor out $\M(R_{\emptyset,{\mathbf{l}}}((a+k-3)/2))\M(\bar{R}_{{\mathbf{l}}-1,\emptyset}((a+1)/2))$ from \eqref{edv}, we obtain
\begin{equation}
\begin{aligned}
    \M(H)
    =\M(H'_{\mathbf{l}}(a,b,k))
    &=2^{w_{l'}(H)-1}\Bigg[\frac{a+2b+k}{a+k}+1\Bigg]\M(R_{\emptyset,{\mathbf{l}}}((a+k-3)/2))\M(\bar{R}_{{\mathbf{l}}-1,\emptyset}((a+1)/2))\\
    &=\frac{2a+2b+2k}{a+k}\cdot2^{w_{l'}(H)-1}\M(R_{\emptyset,{\mathbf{l}}}((a+k-3)/2))\M(\bar{R}_{{\mathbf{l}}-1,\emptyset}((a+1)/2))\\
    &=\frac{a+b+k}{a+k}\cdot2^{w_{l'}(H)}\M(R_{\emptyset,{\mathbf{l}}}((a+k-3)/2))\M(\bar{R}_{{\mathbf{l}}-1,\emptyset}((a+1)/2)).
\end{aligned}
\end{equation}
On the other hand, if we factor out $\M(R_{\emptyset,{\mathbf{l}}}((a+k-1)/2))\M(\bar{R}_{{\mathbf{l}}-1,\emptyset}((a-1)/2))$ from \eqref{edv}, we get
\begin{equation}
\begin{aligned}
    \M(H)
    =\M(H'_{\mathbf{l}}(a,b,k))
    &=2^{w_{l'}(H)-1}\Bigg[\frac{a+k}{a+2b+k}+1\Bigg]\M(R_{\emptyset,{\mathbf{l}}}((a+k-1)/2))\M(\bar{R}_{{\mathbf{l}}-1,\emptyset}((a-1)/2))\\
    &=\frac{2a+2b+2k}{a+2b+k}\cdot2^{w_{l'}(H)-1}\M(R_{\emptyset,{\mathbf{l}}}((a+k-1)/2))\M(\bar{R}_{{\mathbf{l}}-1,\emptyset}((a-1)/2))\\
    &=\frac{a+b+k}{a+2b+k}\cdot2^{w_{l'}(H)}\M(R_{\emptyset,{\mathbf{l}}}((a+k-1)/2))\M(\bar{R}_{{\mathbf{l}}-1,\emptyset}((a-1)/2)).
\end{aligned}
\end{equation}
This completes the proof in the case when $k$ even, $a$ odd, and $l_{1}=1$. The proof of the remaining cases works in exactly the same manner: we apply Theorem \ref{tbb} to $\M(H)$, use Lemma \ref{ldd}, and finally use some identities that the indices satisfy (which we can obtain by using the fact that the region consists of the same number of up- and down-pointing unit triangles or by expressing the same length in two different ways; for example, when $k$ even, $a$ odd, and $l_{1}=1$, we have $m=b$ and $l_{m}=l_{b}=b+\frac{k}{2}$).
\end{proof}

\vskip0.1in




\begin{proof}[Proof of Theorem $\ref{tdd}$]
Let $H$ be one of the regions $H'_{\mathbf{l}}(a,b,k)$, $H'_{{\mathbf{l}},{\mathbf{q}}}(a,b,k)$ or $\bar{H}'_{{\mathbf{l}},{\mathbf{q}}}(a,b,k)$. Note that $H^+$ and $\widehat{H}^-$ are the same $R$ or $\bar{R}$ region, just the $x$-argument in the former is one more that in the latter. The same holds for $\widehat{H}^+$ and $H^-$. This guarantees that $\widehat{H}^+$ fits to the right of $H^+$ to allow a reverse application of the factorization theorem, and $H^-$ fits to the left of $\widehat{H}^-$ to do the same. The proof follows then by multiplying together equations \eqref{edi} and \eqref{edm}, pairing together the factors $\M(H^+)$ and $\M(\widehat{H}^+)$, pairing also $\M(H^-)$ and $\M(\widehat{H}^-)$, and applying the factorization theorem in reverse twice.          
\end{proof}


\section{Nearly symmetric hexagons with a shamrock or fern hole}

An {\it $S$-cored hexagon} is a hexagon from which a {\it shamrock} --- the union of an up-pointing triangular core and three down-pointing triangular lobes touching the core at its three vertices --- is removed (see \cite{CiucuDual}).
Simple product formulas for the number of lozenge tilings of a symmetric $S$-cored hexagon were given in~\cite{CiucuAxialShamrock} and~\cite{LaiAxialShamrock}.

In this section we solve another open problem of Lai (namely, \cite[Problem 27]{LaiOpenProblems}), which asks to find and prove a simple product formula for the number of lozenge tilings of the region obtained from a symmetric $S$-cored hexagon by translating the shamrock hole one unit southeast; we call such a region a {\it nearly symmetric $S$-cored hexagon} (see Theorem \ref{tea}). In addition, we prove a similar formula for {\it $F$-cored hexagons} --- regions obtained from hexagons by removing a {\it fern}\footnote{ Ferns were introduced in \cite{CiucutheOtherDual}, where a simple product formula was proved for the number of tilings of a hexagon with a fern removed from its center. The number of tilings of an $F$-cored hexagon which is symmetric about a vertical axis was determined in \cite{LaiAxialFern}.} --- the union of contiguous triangular lobes of alternating orientation lined up along a horizontal lattice line --- instead of a shamrock (such a region will be called a {\it nearly symmetric $F$-cored hexagon}; see Theorem \ref{tec}).

\begin{thm}
\label{tea}
Let $H'$ be a nearly symmetric $S$-cored hexagon in which the side-lengths of the triangular lobes are $\al$ for the top one, $\be$ for the two bottom ones, and $\mu$ for the central one; let the side-lengths of the outer hexagon be $x+\al+2\be$, $y+\mu$, $y+\al+2\be$, $x+\mu$, $y+\al+2\be$, $y+\mu$ $($clockwise from top; see the picture on the left in Figure $\ref{fea}$ for an example$)$. Let $\ell'$ be the vertical symmetry axis of the removed shamrock.

Denote by $H_1$ the region consisting of the union of the portion of $H'$ to the right of $\ell'$ with its reflection across $\ell'$ $($i.e., the region seen if $\ell'$ was a mirror and one looks in it from the right; see the picture on the right in Figure $\ref{fea}$ for an illustration$)$; let~$H_2$ be the analogous region determined by the portion of $H'$ to the left of $\ell'$. Then we have
\begin{equation}
\M(H')^2=f\M(H_1)\M(H_2),
\label{eea}
\end{equation}
where
\begin{equation}
f=
\begin{dcases}
\frac{(x+y+\al+2\be+\mu)^2}{(x+\al+2\be+\mu)(x+2y+\al+2\be+\mu)}, & \text{$x$ odd},\\
\frac{(x+y+\al+2\be+\mu)^2}{x(x+2y+2\al+4\be+2\mu)}, & \text{$x$ even}.\\
\end{dcases} 
\label{eeb}
\end{equation}

\end{thm}
With the formulas for the number of lozenge tilings of $H_1$ and $H_2$ --- which are, by construction, symmetric $S$-cored hexagons --- given in \cite{CiucuAxialShamrock} and~\cite{LaiAxialShamrock}, this provides an explicit product formula for $\M(H')$, and thus solves Lai's Problem 27 in the list of open problems \cite{LaiOpenProblems}.

\begin{figure}[t]
\centerline{
\hfill
{\includegraphics[width=0.45\textwidth]{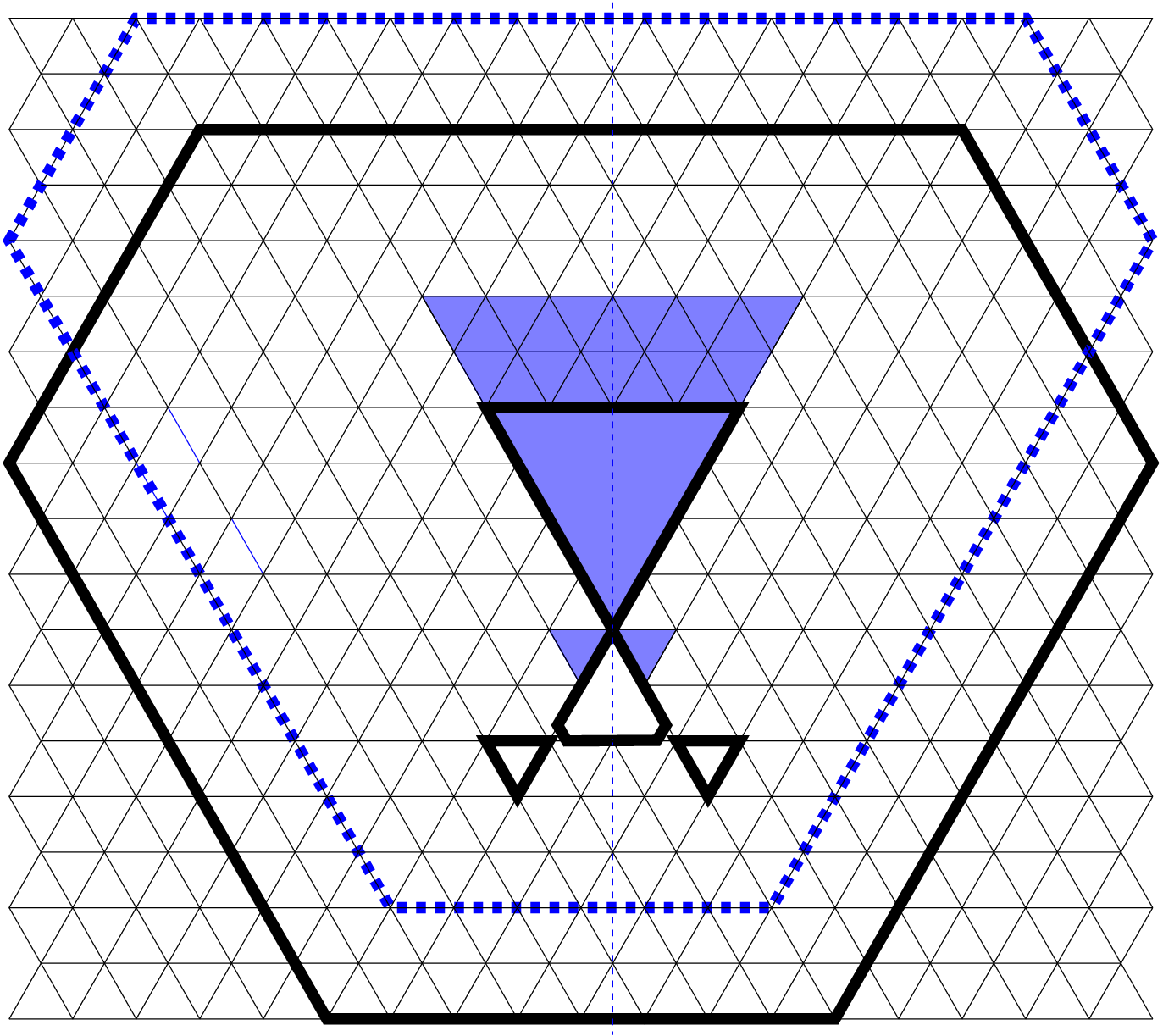}}
\hfill
{\includegraphics[width=0.47\textwidth]{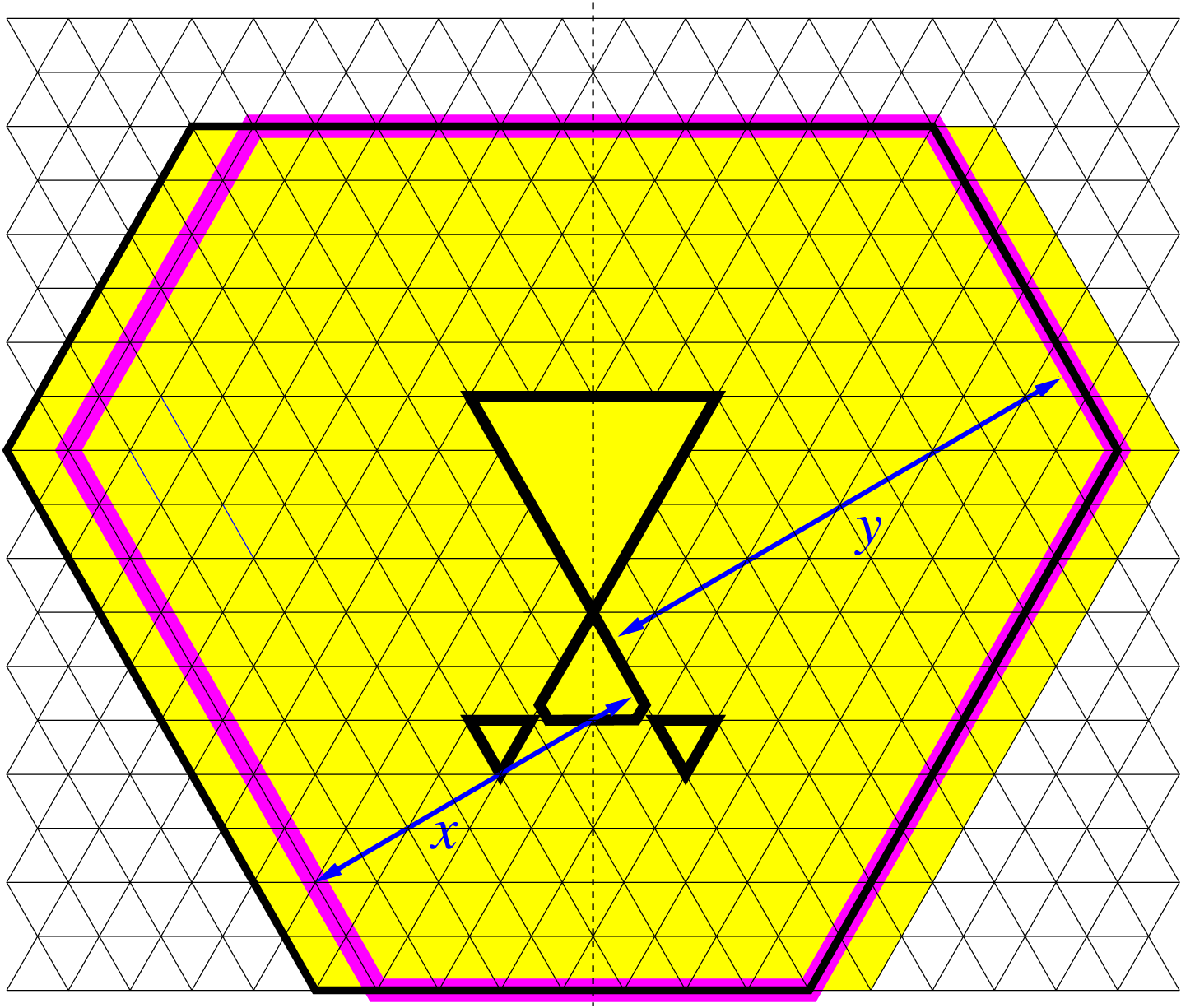}}
\hfill
}
\vskip0.1in
\caption{\label{fea} {\it Left.} Squeezing out the central lobe of the shamrock in a nearly symmetric $S$-cored hexagon; the original region is shown in black, while the region obtained from it after the squeezing is indicated in blue. {\it Right.} The distances $x$ and $y$ between a support line and the opposite sides of the outer hexagon in the thin symmetrization $H_1$ (shown inside the magenta contour) of $H'$; the thick symmetrization $H_2$ is shaded in yellow.}
\vskip-0.2in
\end{figure}

\begin{proof}
 We recall the ``bowtie squeezing theorem'' \cite[Theorem 1]{ciucuBowtieSqueezing}. A triad hexagon is a region obtained from a hexagon by removing three bowtie-shaped holes positioned as illustrated in \cite[Figure 1]{ciucuBowtieSqueezing}. Bowtie squeezing is an operation that turns a triad hexagon into another triad hexagon by squeezing part (or all) of a lobe of a bowtie into the other lobe; the position of the node of the squeezed bowtie as well as the shape of the other two bowties remain the same, but the relative position of the bowties and the boundary hexagon change (an illustrative example is shown \cite[Figure 1]{ciucuBowtieSqueezing} --- see that paper for the details; the picture on the left in Figure \ref{fea} shows a special case). Then \cite[Theorem 1]{ciucuBowtieSqueezing} states that if the region $Q$ is obtained from the triad hexagon $R$ by a sequence of bowtie squeezing operations, then we have\footnote{ Here $A$, $B$ and $C$ (resp., $A_1$, $B_1$ and $C_1$) are the nodes --- i.e., the points where the two lobes meet --- of the bowties in the triad hexagon $R$ (resp., $Q$), counterclockwise from top.}
\begin{equation}
\frac{\M(R)}{\M(Q)}=\dfrac
{\w^{(R)}\dfrac{\ka_A^{(R)}\ka_B^{(R)}\ka_C^{(R)}}{\ka_{BC}^{(R)}\ka_{AC}^{(R)}\ka_{AB}^{(R)}}}
{\w^{(Q)}\dfrac{\ka_{A_1}^{(Q)}\ka_{B_1}^{(Q)}\ka_{C_1}^{(Q)}}{\ka_{B_1C_1}^{(Q)}\ka_{A_1C_1}^{(Q)}\ka_{A_1B_1}^{(Q)}}},
\label{eec}
\end{equation}
where the weight $\w^{(R)}$ depends only on the shapes of the three bowties and their relative positions (the precise expression is given in \cite[(2.2)]{ciucuBowtieSqueezing}, but we don't need it here), and the point couples like $\ka_A^{(R)}$ (resp., the line couples like $\ka_{BC}^{(R)}$) have the following common definition: if $\h(n)$ is the hyperfactorial $\h(n)=0!\,1!\,\cdots\,(n-1)!$, the point couple $\ka_A^{(R)}$ is defined to be $\h(u)\h(d)$, where $u$ and $d$ are the distances from $A$ to the pair of opposite sides of the outer hexagon towards which the $60^\circ$ degree double cone at $A$ contained in the bowtie points (for line couples, $u$ and $d$ are the distances from the {\it line} to the sides of the hexagon {\it parallel to it}).

Consider our nearly symmetric $S$-cored hexagon $H'$ (indicated in black in Figure \ref{fea}). Its shamrock-shaped hole can be regarded as a triad of bowties: the top bowtie has upper lobe of side $\al$ and lower lobe of side $\mu$, while the other two bowties consist of just their lower lobes, both of side $\be$. Let $T'$ be the region obtained from $H'$ by completely squeezing out the bottom lobe of the top bowtie ($T'$ is indicated in blue in Figure \ref{fea}). Applying equation \eqref{eec} for $R=H'$ and $Q=T'$ we get
\begin{equation}
\frac{\M(H')}{\M(T')}=\dfrac
{\w^{(H')}\dfrac{\ka_A^{(H')}\ka_B^{(H')}\ka_C^{(H')}}{\ka_{BC}^{(H')}\ka_{AC}^{(H')}\ka_{AB}^{(H')}}}
{\w^{(T')}\dfrac{\ka_{A_1}^{(T')}\ka_{B_1}^{(T')}\ka_{C_1}^{(T')}}{\ka_{B_1C_1}^{(T')}\ka_{A_1C_1}^{(T')}\ka_{A_1B_1}^{(T')}}}
=\dfrac
{\w^{(H')}\dfrac{\ka_A^{(H')}}{\ka_{AC}^{(H')}\ka_{AB}^{(H')}}}
{\w^{(T')}\dfrac{\ka_{A_1}^{(T')}}{\ka_{A_1C_1}^{(T')}\ka_{A_1B_1}^{(T')}}},
\label{eed}
\end{equation}
where at the last equality we used that $\ka_B^{(H')}=\ka_{B_1}^{(T')}$, $\ka_C^{(H')}=\ka_{C_1}^{(T')}$ and $\ka_{BC}^{(H')}=\ka_{B_1C_1}^{(T')}$ (these equalities hold because the pairs of distances involved in the couples claimed to be equal are identical).

Consider now the ``thin symmetrization'' $H_1$ described in the statement of the theorem, and let~$T_1$ be the region obtained from it by completely squeezing out the bottom lobe of the top bowtie. Equation \eqref{eec} then similarly gives
\begin{equation}
\frac{\M(H_1)}{\M(T_1)}=\dfrac
     {\w^{(H_1)}\dfrac{\ka_A^{(H_1)}}{\ka_{AC}^{(H_1)}\ka_{AB}^{(H_1)}}}
{\w^{(T_1)}\dfrac{\ka_{A_1}^{(T_1)}}{\ka_{A_1C_1}^{(T_1)}\ka_{A_1B_1}^{(T_1)}}}.
\label{eee}
\end{equation}
Note that the point couples $\ka_A^{(H')}$ are $\ka_A^{(H_1)}$ equal, since the distances involved in them are identical; similarly, $\ka_{A_1}^{(T')}=\ka_{A_1}^{(T_1)}$. Furthermore, we have $\w^{(H')}=\w^{(H_1)}$, because $\w^{(R)}$ depends only on the shapes and relative positions of the bowties in the triad hexagon; similarly, $\w^{(T')}=\w^{(T_1)}$.  Therefore, combining \eqref{eed} and \eqref{eee} we obtain 
\begin{equation}
\frac{\M(H')/\M(H_1)}{\M(T')/\M(T_1)}=
\dfrac{\ka_{AC}^{(H_1)}\ka_{AB}^{(H_1)}}
{\ka_{AC}^{(H')}\ka_{AB}^{(H')}}
\,
\dfrac{\ka_{A_1C_1}^{(T')}\ka_{A_1B_1}^{(T')}}
{\ka_{A_1C_1}^{(T_1)}\ka_{A_1B_1}^{(T_1)}}.
\label{eef}
\end{equation}
Repeating the argument with $H_1$ replaced by the thick symmetrization $H_2$ of $H'$ yields
\begin{equation}
\frac{\M(H')/\M(H_2)}{\M(T')/\M(T_2)}=
\dfrac{\ka_{AC}^{(H_2)}\ka_{AB}^{(H_2)}}
{\ka_{AC}^{(H')}\ka_{AB}^{(H')}}
\,
\dfrac{\ka_{A_1C_1}^{(T')}\ka_{A_1B_1}^{(T')}}
{\ka_{A_1C_1}^{(T_2)}\ka_{A_1B_1}^{(T_2)}}.
\label{eeg}
\end{equation}
Multiply the last two equations side by side to obtain
\begin{equation}
\frac{\M(H')^2/(\M(H_1)\M(H_2))}{\M(T')^2/(\M(T_1)\M(T_2))}=
\,
\dfrac{\ka_{AC}^{(H_1)}\ka_{AB}^{(H_1)}}
{\ka_{AC}^{(H')}\ka_{AB}^{(H')}}
\dfrac{\ka_{A_1C_1}^{(T')}\ka_{A_1B_1}^{(T')}}
{\ka_{A_1C_1}^{(T_1)}\ka_{A_1B_1}^{(T_1)}}
\,
\dfrac{\ka_{AC}^{(H_2)}\ka_{AB}^{(H_2)}}
{\ka_{AC}^{(H')}\ka_{AB}^{(H')}}
\,
\dfrac{\ka_{A_1C_1}^{(T')}\ka_{A_1B_1}^{(T')}}
{\ka_{A_1C_1}^{(T_2)}\ka_{A_1B_1}^{(T_2)}}.
\label{eeh}
\end{equation}
Denoting by $x$ and $y$ the distances from the line supporting the top and right lobes in $H_1$ to the northeast and southwest sides of the boundary (see the picture on the right in Figure \ref{fea}), the factors involving $H$-regions on the right hand side above combine to
\begin{equation}
\dfrac{\h(x)\h(y)}
{\h(x+1)\h(y)}
\,
\dfrac{\h(x)\h(y)}
{\h(x)\h(y+1)}
\,
\dfrac{\h(x+1)\h(y+1)}
{\h(x+1)\h(y)}
\,
\dfrac{\h(x+1)\h(y+1)}
{\h(x)\h(y+1)}
=1.
\label{eei}
\end{equation}
The same calculation shows that the factors involving $T$-regions combine to 1 as well. Therefore, the entire right hand side of \eqref{eeh} is 1, and we get
\begin{equation}
\frac{\M(H')^2}{\M(H_1)\M(H_2)}=
\frac{\M(T')^2}{\M(T_1)\M(T_2)}.
\label{eej}
\end{equation}
Since Theorem \ref{tdd} applies to $T'$ (indeed, due to forced lozenges, the removal of the three down-pointing lobes that meet at a point is equivalent, from the point of view of counting lozenge tilings, with the removal of a single down-pointing triangle of side $\al+2\be+\mu$), the statement follows from \eqref{eej} by noticing that the $a$-, $b$- and $k$-values of the parameters in Theorem \ref{tdd} corresponding to $T'$ are $x$, $y$ and $\al+2\be+\mu$, respectively. 
\end{proof} 


The operation of lobe squeezing can also be defined for a hexagon from which a collection of disjoint collinear ferns has been removed. The definition is analogous to the bowtie squeezing operation (which really squeezes one lobe of a given bowtie into the other) described in \cite{ciucuBowtieSqueezing}.

For simplicity, suppose that only one fern has been removed from the hexagon. The picture on the left in Figure \ref{feb} illustrates such an example: the black solid lines show the original region, while shown in blue is the region that results by completely squeezing the central lobe into the one to its right. In general, let $R$ be a hexagon with a fern removed, and let $A$ be a node of the fern (i.e., a point where two lobes meet). Then the region obtained from $R$ by squeezing the lobe to the left of $A$ into the lobe to the right of $A$ by $d$ units is the region $Q$ described by: (1) $Q$ is a hexagon with a fern removed (2) the boundary hexagon of $Q$ is obtained from that of $R$ by pushing in three of its sides by $d$ units, and pushing out the other three by $d$ units (which three are pushed in and which are pushed out is determined by the direction of the squeezing, which pulls in the three sides in the back and pushes out the three sides in the front), and (3) the removed fern of $Q$ is obtained from that of $R$ by keeping the node $A$ fixed, shrinking the lobe to its left by $d$ units (while ``dragging along'' the portion of the fern to the left of it), and inflating the lobe to the right of $A$ by $d$ units (while ``pushing away'' the portion of the fern to the right of this lobe).

The definition works in the same way for the case when a collection of ferns lined up along a common horizontal line has been removed from the hexagon --- the only change is that in (3) above when we shrink and correspondingly inflate the two lobes at $A$, we translate to the right by $d$ units the entire remaining portion of the system of ferns.

\begin{figure}[t]
\centerline{
\hfill
{\includegraphics[width=0.45\textwidth]{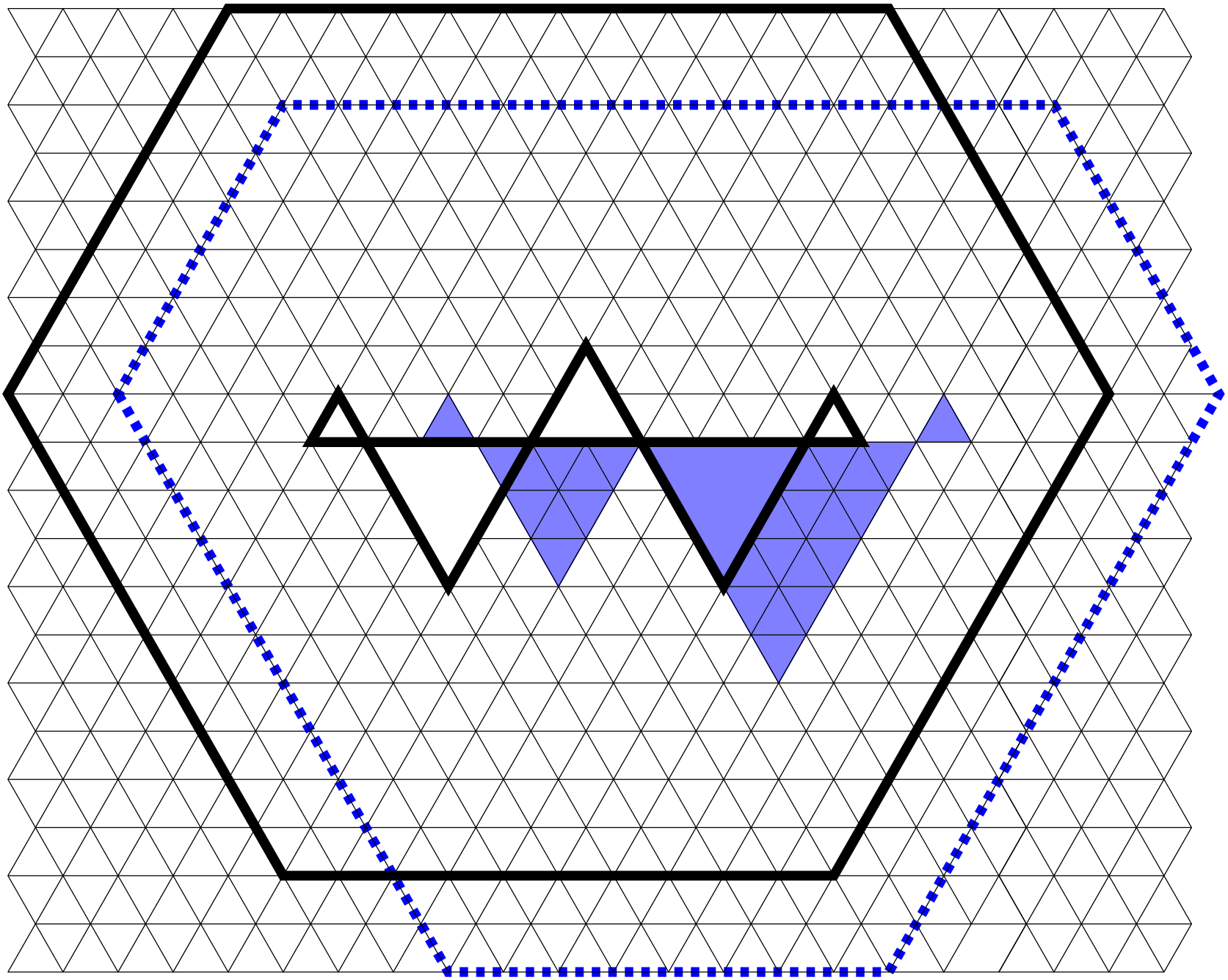}}
\hfill
{\includegraphics[width=0.45\textwidth]{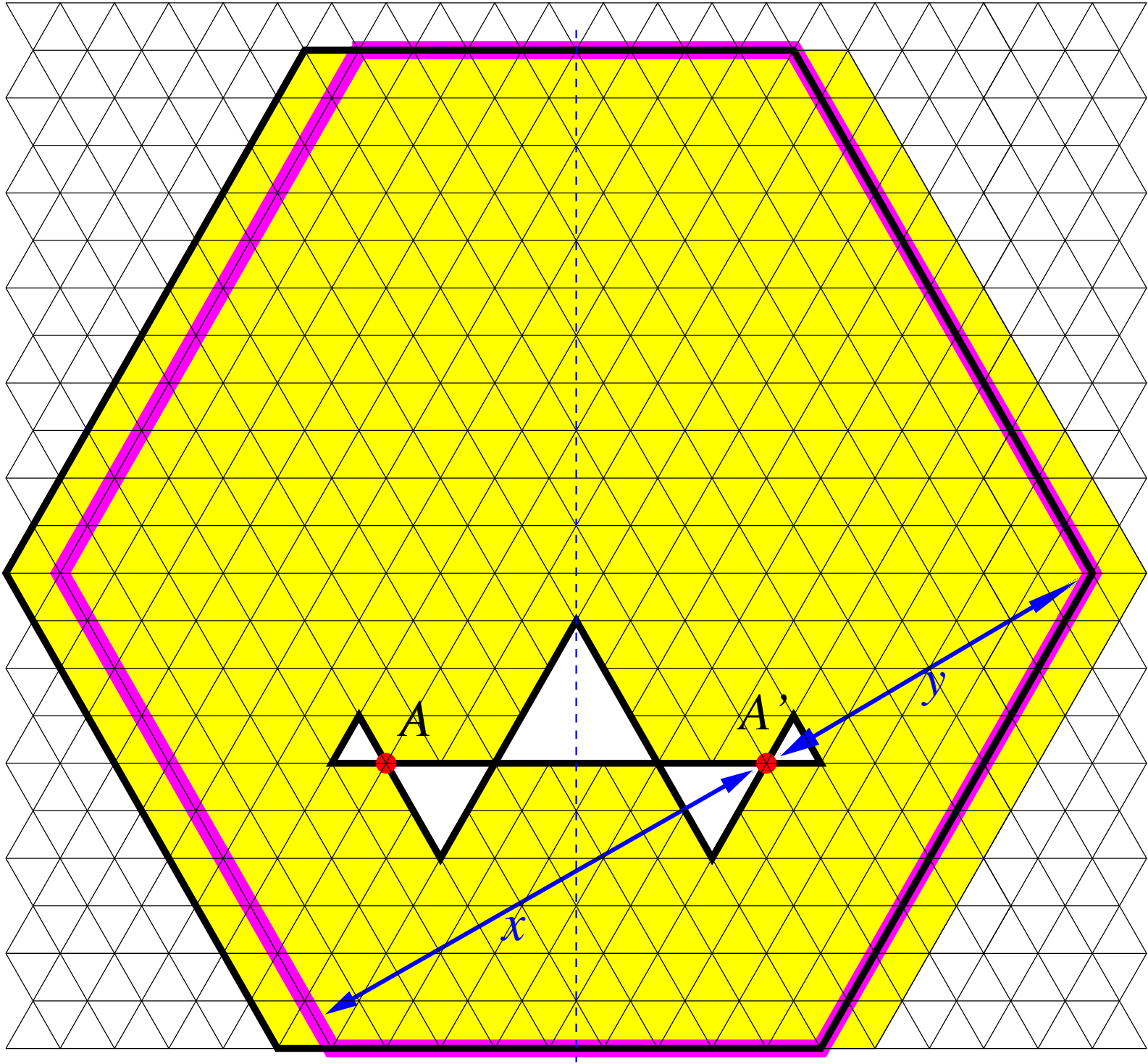}}
\hfill
}
\vskip0.1in
\caption{\label{feb} {\it Left.} Squeezing out the central lobe of a nearly symmetric $F$-cored hexagon. {\it Right.} The distances $x$ and $y$ between the node $A'$ and the opposite sides of the outer hexagon in the thin symmetrization $H_1$ (inside the magenta contour) of $H'$; the yellow shading indicates $H_2$.}
\vskip-0.2in
\end{figure}

Consider a region $R$ obtained from a hexagon by removing a collection of collinear ferns. A {\it node of $R$} is a point $A$ on one of the removed ferns where two lobes meet. A {\it support line of $R$} is a lattice line $L$ that contains an edge of at least two lobes (with the exception of the common axis of the ferns, all support lines support precisely two lobes). The point couples $\ka_A^{(R)}$ and line couples $\ka_L^{(R)}$ are defined exactly like in \cite{ciucuBowtieSqueezing} (these were also recalled at the beginning of the proof of Theorem~\ref{tea}). To define the {\it weight} $\w^{(R)}$, label the unit lattice segments inside $R$ which are on the common axis of the removed ferns by $1,2,3,\dotsc$, from left to right. If $I$ is the set of labels corresponding to the bases of the down-pointing lobes of the ferns, and $J$ consists of the labels on the bases of the up-pointing lobes, then the weight $\w^{(R)}$ is defined to be
\begin{equation}
\w^{(R)}=\Delta(I)\Delta(J),
\label{eek}
\end{equation}
where for a set $I=\{s_1,\dotsc,s_k\}$ with elements listed in increasing order we define
\begin{equation}
\Delta(I):=\prod_{1\leq i<j\leq k}(s_j-s_i).
\label{eekk}
\end{equation}
%



We have the following analog of the bowtie squeezing theorem.

\begin{thm}
\label{teb}
Let $R$ be a hexagonal region from which a system of co-axial ferns has been removed. Let $Q$ be a region obtained from $R$ by a sequence of lobe squeezings. Then we have
\begin{equation}
\frac{\M(R)}{\M(Q)}=\dfrac
{\w^{(R)}\dfrac{\prod_{\text{$A$ node of $R$}}\,\ka_A^{(R)}}{\prod_{\text{$L$ support line of $R$}}\,\ka_{L}^{(R)}}}
{\w^{(Q)}\dfrac{\prod_{\text{$A$ node of $Q$}}\,\ka_A^{(Q)}}{\prod_{\text{$L$ support line of $Q$}}\,\ka_{L}^{(Q)}}}.
\label{eel}
\end{equation}

\end{thm}

\begin{proof}

We can assume without loss of generality that $Q$ is obtained from $R$ by a single operation of lobe squeezing.
  
Let $\overline{R}$ be the region obtained from $R$ by the following procedure. Denote by $R_1$ the region obtained by starting with the outer hexagon $H$ of $R$, but removing from it, instead of the system of ferns, just those unit triangles contained in the system of ferns whose bases rest on their common axis $l$. Note that, due to forced lozenges, the region $R_1$ has the same number of lozenge tilings as $R$. 
Next, construct a vertically symmetric hexagon $\overline{H}$ so that $l$ is its horizontal diagonal, and four of its sides --- including the top and bottom ones --- are one the same lattice lines as the corresponding sides of $H$. Finally, remove from $\overline{H}$ the unit triangles that were removed to construct $R_1$, and in addition also runs of unit triangles along the portions of $l$ outside $H$, chosen so as to force a unique tiling of the remaining portion of $\overline{H}\setminus H$; the resulting region is what we define to be $\overline{R}$. Then we have $\M(\overline{R})=\M(R_1)=\M(R)$, and $\overline{R}$ is precisely the kind of region to which the shuffling theorem of \cite{Laishuffling} applies.

This theorem concerns pairs of regions $S$ and $T$ which can be described as follows. The region $S$ is a vertically symmetric hexagon with a set of unit triangles (both up-pointing and down-pointing being allowed) removed from along its horizontal diagonal $l$. The region $T$ is obtained from $S$ by ``flipping'' an arbitrary subset of the removed unit triangles, i.e.\ moving them to the position of their mirror image across $l$, and changing the boundary of the outer hexagon appropriately (the latter is needed in order to end up with a region with the same number of up- and down-pointing unit triangles, a necessary condition for the existence of lozenge tilings). Then the theorem states that $\M(S)/\M(T)$ has a simple expression in terms of $\Delta$ functions \eqref{eekk} of sets which record the positions of the removed unit triangles.

Let $A$ be a node of the system of ferns in the given region $R$, and assume that $Q$ was obtained from $R$ by squeezing out $d$ units of the lobe left of $A$ into the lobe right of $A$. Consider the region $\overline{R}$ defined above, and also the region $\overline{Q}$ obtained by applying the same procedure to $Q$.
Apply the shuffling theorem to the region $\overline{R}$, by flipping the run of $d$ contiguous unit triangles just to the left of $A$. It is not hard to verify that the resulting region is precisely $\overline{Q}$.

Therefore, the shuffling theorem of \cite{Laishuffling} applies to the regions $\overline{R}$ and $\overline{Q}$, and gives
\begin{equation}
\frac{\M(\overline{R})}{\M(\overline{Q})}=
\frac{1/\ka_{l}^{(\overline{R})} }{1/\ka_{l}^{(\overline{Q})} }
\frac{\Delta(X)\Delta(Y)}{\Delta(X')\Delta(Y')},
\label{eell}
\end{equation}
where $X$ (resp., $Y$) is the set of labels\footnote{ The unit segments on $l$ inside $\overline{R}$ are labeled $1,2,3,\dotsc$ from left to right.} of the up-pointing (resp., down-pointing) unit triangles with bases on $L$ which are removed from the outer hexagon of $\overline{R}$, while $X'$ and $Y'$ are the corresponding sets of labels for $\overline{Q}$.

After a straightforward (if lengthy) rearranging of the factors on the right hand side of \eqref{eell}, one can bring the expression on the right hand side of \eqref{eell} to the form shown on the right hand side of \eqref{eel}. Since $\frac{\M(R)}{\M(Q)}=\frac{\M(\overline{R})}{\M(\overline{Q})}$, this completes the proof.
\end{proof}

{\it Remark $6$.} As pointed out in \cite{CiucutheOtherDual}, it seems that shamrocks and ferns are the only shapes of holes that lead to simple product formulas when considering the number of lozenge tilings of a hexagon with a hole removed from its center. The above theorem and the statement \eqref{eec} of \cite[Theorem~1]{ciucuBowtieSqueezing} take the common features of these shapes one step further, by providing a pleasingly unified form for the quantities that govern the number of lozenge tilings in regions related by lobe squeezing operations --- up to the weights $\w^{(R)}$ and  $\w^{(Q)}$, the right hand side of \eqref{eec} is precisely the same expression as that of \eqref{eel}!

\vskip0.1in
As an application of Theorem \ref{teb}, we show that nearly symmetric hexagons with a fern-shaped hole are related to their symmetric counterparts by essentially the same identity as in Theorems \ref{tdd} and \ref{tea}.

\begin{thm}
\label{tec}
Let $H'$ be a nearly symmetric $F$-cored hexagon in which the side-lengths of the triangular lobes of the removed ferns are $\al_1,\al_2,\cdots$ for the down-pointing ones, and $\be_1,\be_2,\cdots$ for the up-pointing ones. Set $\al=\al_1+\al_2+\cdots$ and $\be=\be_1+\be_2+\cdots$, and let the side-lengths of the outer hexagon be $x+\al$, $y+\be$, $y+\al$, $x+\be$, $y+\al$, $y+\be$ $($clockwise from top$)$. Let $\ell'$ be the vertical symmetry axis of the removed fern.

Let the thin and thick symmetrizations of $H'$, $H_1$ and $H_2$, be defined as in the statement of Theorem $\ref{tea}$. Then we have

%
\begin{equation}
\M(H')^2=f\M(H_1)\M(H_2),
\label{eeu}
\end{equation}
where
\begin{equation}
f=
\begin{dcases}
\frac{(x+y+\al+\be)^2}{(x+\al+\be)(x+2y+\al+\be)}, & \text{$x$ odd},\\
\frac{(x+y+\al+\be)^2}{x(x+2y+2\al+2\be)}, & \text{$x$ even}.\\
\end{dcases} 
\label{eev}
\end{equation}

\end{thm}

\begin{proof} The same approach as in the proof of Theorem \ref{tea} works. Let $T'$ be the region obtained from $H'$ by completely squeezing out the up-pointing lobes of the fern into the down-pointing lobes. Let $T_1$ and $T_2$ be the regions obtained in an analogous fashion from the symmetrized regions $H_1$ and $H_2$.
 
  Write down equation \eqref{eel} for the pairs of regions $(H',T')$ and $(H_1,T_1)$, and divide the resulting equations side by side to get an expression for $\frac{\M(H')/\M(T')}{\M(H_1)/\M(T_1)}$. Do the same for the pairs of regions $(H',T')$ and $(H_2,T_2)$ to get an expression for $\frac{\M(H')/\M(T')}{\M(H_2)/\M(T_2)}$. Finally, multiply these two equations to get an expression for $\frac{\M(H')^2/(\M(H_1)\M(H_2))}{\M(T')^2/(\M(T_1)\M(T_2))}$; this will involve only point couples and line couples, as the weights $\w$ cancel out (since they are the same for $H'$, $H_1$ and $H_2$, and also for $T'$, $T_1$ and $T_2$). If the distances $x$ and $y$ are as indicated in the picture on the right in Figure \ref{feb}, the node couples pertaining to the regions $H'$, $H_1$ and $H_2$ contribute
\begin{equation}  
\prod_A\frac{\ka_A^{(H')}\ka_{A'}^{(H')}}{\ka_A^{(H_1)}\ka_{A'}^{(H_1)}}  
 \frac{\ka_A^{(H')}\ka_{A'}^{(H')}}{\ka_A^{(H_2)}\ka_{A'}^{(H_2)}}  
=
\prod_A
\frac{\h(y+1)\h(x)\, \h(x+1)\h(y) }{\h(y)\h(x)\, \h(y)\h(x) }
\frac{\h(y+1)\h(x)\, \h(x+1)\h(y) }{\h(y+1)\h(x+1)\, \h(y+1)\h(x+1) }
=1,
\label{eem}
\end{equation}
where the products are over the nodes of the fern to the left of its symmetry axis, and $A'$ is the mirror image of $A$ across this axis.

A similar calculation shows that the line couples pertaining to the regions $H'$, $H_1$ and $H_2$ also contribute 1. The same calculations show that the factors pertaining to the regions $T'$, $T_1$ and $T_2$ have an overall contribution of 1 as well. This shows that we have $\M(H')^2/(\M(H_1)\M(H_2))=\M(T')^2/(\M(T_1)\M(T_2))$. The statement follows now from Theorem \ref{tdd}, upon observing that for the outer hexagon of $T'$, the $a$-, $b$- and $k$-parameters in that theorem are $x$, $y$ and $\al+\be$, respectively.
\end{proof}









\section{Nearly symmetric hexagons with a lozenge hole}

Using Theorem \ref{tbb}, we can also obtain a counterpart of a result of Fulmek and Krattenthaler which was presented in \cite{FulmekKrattenthaler1}. In \cite{FulmekKrattenthaler1} and \cite{FulmekKrattenthaler2}, these authors considered the problem of enumerating lozenge tilings of a hexagon that has both vertical and horizontal symmetry --- in this section, we will call this a {\it symmetric hexagon}\footnote{ Doubly-symmetric would be more descriptive, but then naming the region obtained by translating the lozenge-shaped hole off the symmetry axis would become cumbersome.} --- which contain a fixed lozenge (or equivalently, tilings of a hexagon with a lozenge-shaped hole). In \cite{FulmekKrattenthaler1}, they considered a symmetric hexagon, fixed a unit lozenge on the vertical\footnote{Unlike in the current paper, in \cite{FulmekKrattenthaler1}, the authors use a triangular lattice such that one family of lattice line is vertical. Thus, a horizontal symmetry axis in \cite{FulmekKrattenthaler1} corresponds to a vertical symmetry axis in the current paper.} symmetry axis in an arbitrary position, and proved the following two results
(Figure \ref{ffa} illustrates the two types of regions considered).

\begin{figure}
    \centering
    \includegraphics[width=.6\textwidth]{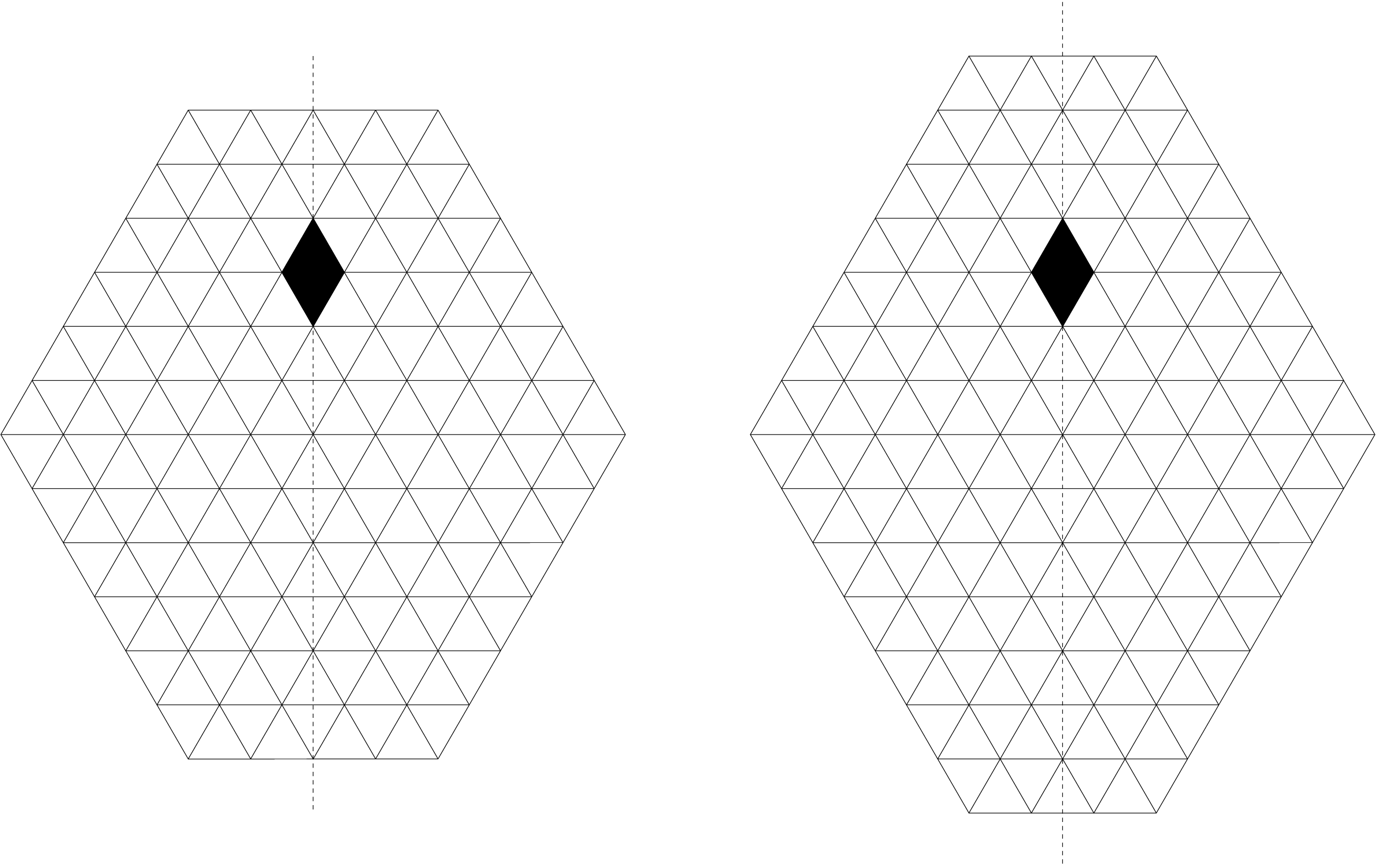}
    \caption{The regions considered in Theorem \ref{tfa} (left) and Theorem \ref{tfb} (right). These show the case when $m=2$, $N=6$, and $l=2$.}
    \label{ffa}
\end{figure}

\begin{thm}[Theorem 1 in \cite{FulmekKrattenthaler1}]
    Let $m$ be a non-negative integer and $N$ a positive integer. The number of lozenge tilings of a hexagon with sides $N,2m,N,N,2m,N$, which contain the $l$-th lozenge on the symmetry axis which cuts through the sides of length $2m$, equals\footnote{ The reason the formula is written as a prefactor times the shown triple product is that, by a classical result of MacMahon \cite{MacM}, the number of lozenge tilings of a hexagon of side-lengths $a$, $b$, $c$, $a$, $b$, $c$ equals $\prod_{i=1}^{a}\prod_{j=1}^{b}\prod_{k=1}^{c}\frac{i+j+k-1}{i+j+k-2}$.} 
    \begin{equation}
        \frac{m\binom{m+N}{m}\binom{m+N-1}{m}}{\binom{2m+2N-1}{2m}}\Bigg(\sum_{e=0}^{l-1}(-1)^{e}\binom{N}{e}\frac{(N-2e)(\frac{1}{2})_{e}}{(m+e)(m+N-e)(\frac{1}{2}-N)_{e}}\Bigg)\cdot\prod_{i=1}^{N}\prod_{j=1}^{N}\prod_{k=1}^{2m}\frac{i+j+k-1}{i+j+k-2}.
    \end{equation}
\label{tfa}
\end{thm}

\begin{thm}[Theorem 2 in \cite{FulmekKrattenthaler1}]
    Let $m$ and $N$ be positive integers. The number of lozenge tilings of a hexagon with sides $N+1,2m-1,N+1,N+1,2m-1,N+1$, which contain the $l$-th lozenge on the symmetry axis which cuts through the sides of length $2m-1$, equals
    \begin{equation}
        \frac{m\binom{m+N}{m}\binom{m+N-1}{m}}{\binom{2m+2N-1}{2m}}\Bigg(\sum_{e=0}^{l-1}(-1)^{e}\binom{N}{e}\frac{(N-2e)(\frac{1}{2})_{e}}{(m+e)(m+N-e)(\frac{1}{2}-N)_{e}}\Bigg)\cdot\prod_{i=1}^{N+1}\prod_{j=1}^{N+1}\prod_{k=1}^{2m-1}\frac{i+j+k-1}{i+j+k-2}.
    \end{equation}    
\label{tfb}
\end{thm}

In \cite{FulmekKrattenthaler2} Fulmek and Krattenthaler considered the problem of enumerating lozenge tilings of a symmetric hexagon with a lozenge hole on the {\it horizontal} symmetry axis. Formulas were provided for three cases:
when the lozenge hole is (1) a half unit, (2) one unit, and (3) one and a half units away from the vertical symmetry axis. We recall the two theorems in \cite{FulmekKrattenthaler2} that correspond to the first case (the two pictures in Figure \ref{ffb} illustrate these regions).

\begin{figure}
    \centering
    \includegraphics[width=.6\textwidth]{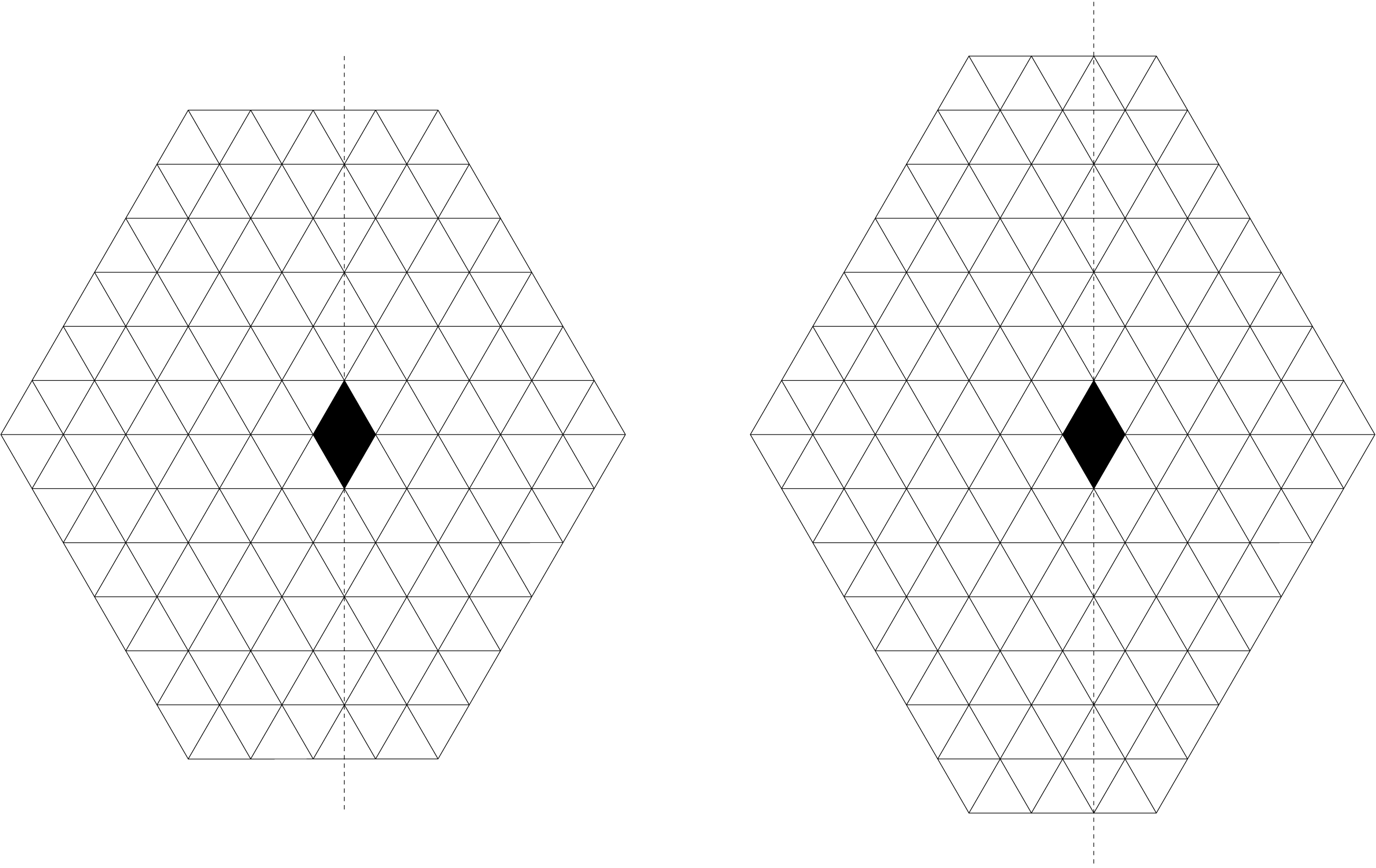}
    \caption{The regions considered in Theorem \ref{tfc} (left) and Theorem \ref{tfd} (right). These pictures illustrate the case when $m=2$ and $n=3$.}
    \label{ffb}
\end{figure}

\begin{thm} [Theorem 1 in \cite{FulmekKrattenthaler2}]
    Let $n$ and $m$ be positive integers. The number of lozenge tilings of a hexagon with side lengths $2n,2m,2n,2n,2m,2n$ $($clockwise from top left$)$, which contain the lozenge just to the right of the center of the hexagon, equals
    \begin{equation}
    \begin{aligned}
        \frac{nm\binom{2n}{n}\binom{2n-1}{n}\binom{2m}{m}}{\binom{4n+2m-1}{2n+m}}\Bigg(-\frac{1}{(n+m)^2}&+\frac{4n+2}{(n+1)(2n-1)(n+m-1)(n+m+1)}\\
        &\cdot\sum_{h=0}^{n-1}\frac{(2)_{h}(1-n)_{h}(\frac{3}{2}+n)_{h}(1-n-m)_{h}(1+n+m)_{h}}{(1)_{h}(2+n)_{h}(\frac{3}{2}-n)_{h}(2+n+m)_{h}(2-n-m)_{h}}\Bigg)\\
        &\cdot\prod_{i=1}^{2n}\prod_{j=1}^{2m}\prod_{k=1}^{2n}\frac{i+j+k-1}{i+j+k-2}.            
    \end{aligned}
    \end{equation}
\label{tfc}
\end{thm}

\begin{thm} [Theorem 2 in \cite{FulmekKrattenthaler2}]
    Let $n$ be a non-negative integer and $m$ a positive integer. The number of lozenge tilings of a hexagon with side lengths $2n+1,2m-1,2n+1,2n+1,2m-1,2n+1$ $($clockwise from top left$)$, which contain the lozenge just to the right of the center of the hexagon, equals
    \begin{equation}
    \begin{aligned}
        \frac{(n+1)m\binom{2n}{n}\binom{2n+1}{n}\binom{2m-1}{m}}{\binom{4n+2m}{2n+m}}\Bigg(\frac{1}{(n+m)^2}&+\frac{4n}{(n+1)(2n-1)(n+m-1)(n+m+1)}\\
        &\cdot\sum_{h=0}^{n-1}\frac{(2)_{h}(1-n)_{h}(\frac{3}{2}+n)_{h}(1-n-m)_{h}(1+n+m)_{h}}{(1)_{h}(2+n)_{h}(\frac{3}{2}-n)_{h}(2+n+m)_{h}(2-n-m)_{h}}\Bigg)\\
        &\cdot\prod_{i=1}^{2n+1}\prod_{j=1}^{2m-1}\prod_{k=1}^{2n+1}\frac{i+j+k-1}{i+j+k-2}.            
    \end{aligned}
    \end{equation}
\label{tfd}
\end{thm}

Note that given a region $R$ and a lozenge $L$ inside the region, the lozenge tilings of $R$ containing $L$ can be identified with the lozenge tilings of the region obtained from $R$ by creating a lozenge-shaped hole at location $L$.  
Thus, the four theorems above can be regarded as lozenge tiling enumerations of symmetric hexagons with unit lozenges removed.

The result we present in this section is an extension of Theorems \ref{tfc} and \ref{tfd} in the style of Theorems \ref{tfa} and \ref{tfb}. For a non-negative integer $m$ and a positive integer $N$, let $V(N,2m,l)$ be the region obtained from the symmetric hexagon of side-lengths $2m,N,N,2m,N,N$ (clockwise from top) by deleting the $l$-th lozenge on the vertical symmetry axis. Similarly, let $V'(N,2m,l)$ be the region obtained from the same hexagon by deleting the $l$-th lozenge on the vertical axis that is half a unit to the right of the vertical symmetry axis.

Similarly, for positive integers $m$ and $N$, let $V(N+1,2m-1,l)$ be the region obtained from the symmetric hexagon of side-lengths $2m-1,N+1,N+1,2m-1,N+1,N+1$ (clockwise from top) by deleting the $l$-th lozenge on the vertical symmetry axis. Similarly, let $V'(N+1,2m-1,l)$ be the region obtained from the same hexagon by deleting the $l$-th lozenge on the vertical axis that is half a unit to the right of the vertical symmetry axis.

The first two results of Fulmek and Krattenthaler stated above (Theorem \ref{tfa} and \ref{tfb}) are about the number of lozenge tilings of the regions $V(N,2m,l)$ and $V(N+1,2m-1,l)$.
The  main results of this section (see Theorems \ref{tfe} and \ref{tff} below) give formulas for the number of lozenge tilings of the regions $V'(N,2m,l)$ and $V'(N+1,2m-1,l)$, and thus extends the results of Fulmek and Krattenthaler quoted as Theorems \ref{tfc} and \ref{tfd} by lending them the generality of the context of their results quoted as Theorems \ref{tfa} and \ref{tfb}.




\begin{figure}
    \centering
    \includegraphics[width=.6\textwidth]{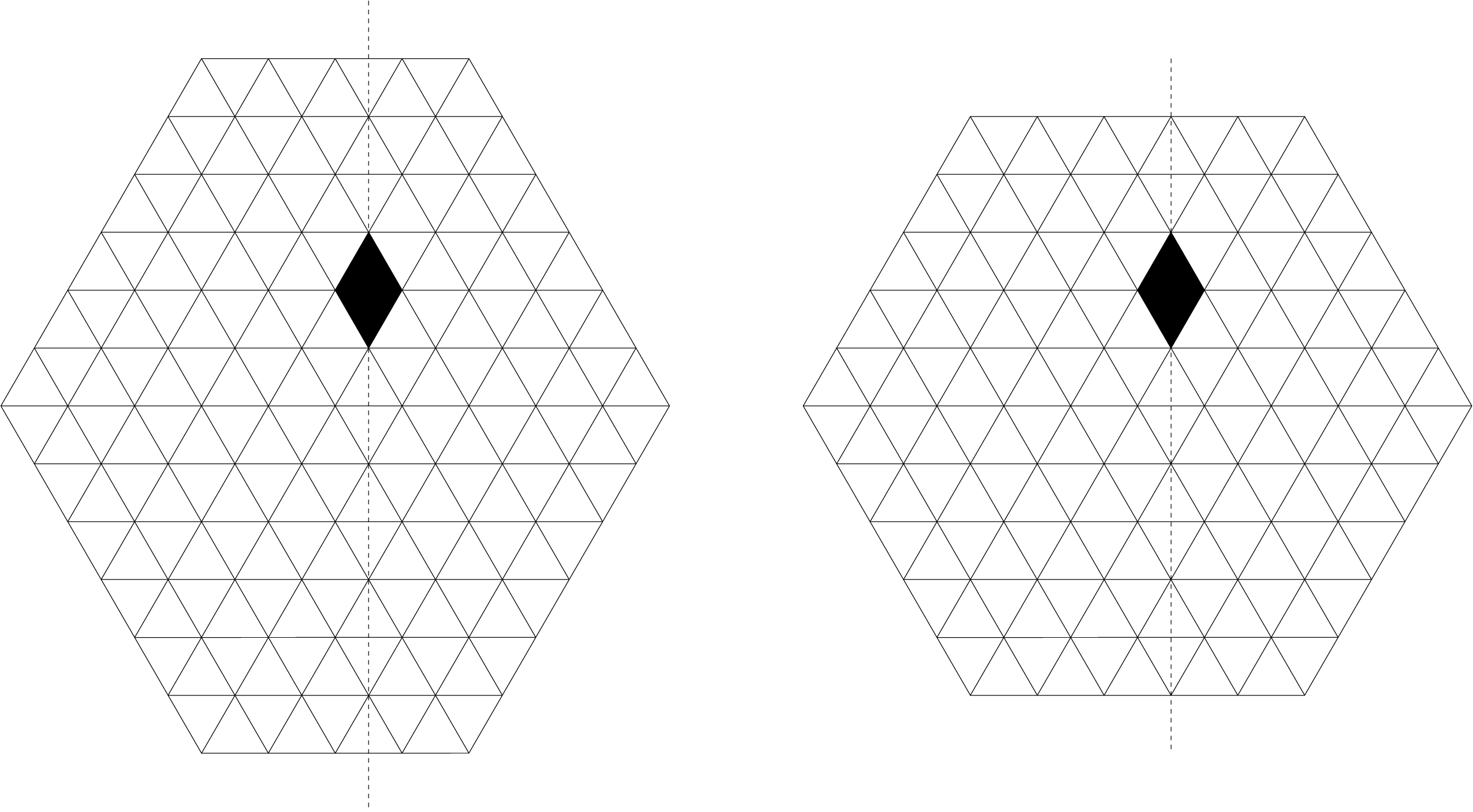}
    \caption{The regions considered in the current section. These two pictures illustrate $V'(N,2m,l)$ and $V'(N-1,2m+1,l)$ when $m=2$, $N=6$, and $l=2$.}
    \label{ffc}
\end{figure}

\begin{thm}
    Let $m$ be a positive integer and $N$ be a positive integer greater than $1$. The number of lozenge tilings of the region  $V'(N,2m,l)$ obtained from the hexagon with sides $2m,N,N,2m,N,N$ $($clockwise from top$)$ by deleting the $l$-th lozenge from the top on the vertical axis half a unit away from the vertical symmetry axis, is given by
    \begin{equation}
        \M(V'(N,2m,l))=\frac{(m+1)_{N}(m)_{N}(N)_{N-1}}{(2m+1)_{N}(2m+N)_{N}(N-1)!}\cdot\sigma_{1}(N,m,l)\cdot\prod_{i=1}^{N}\prod_{j=1}^{N}\prod_{k=1}^{2m}\frac{i+j+k-1}{i+j+k-2}
    \end{equation}
where
\begin{equation}
\begin{aligned}
    \sigma_{1}(N,m,l)=&(2m+2N-1)\cdot\sum_{e=0}^{l-1}(-1)^{e}\binom{N-1}{e}\frac{(N-2e-1)(\frac{1}{2})_{e}}{(m+e)(m+N-e-1)(\frac{3}{2}-N)_{e}}\\
    +&(2m+1)\cdot\sum_{e=0}^{l-1}(-1)^{e}\binom{N-1}{e}\frac{(N-2e-1)(\frac{1}{2})_{e}}{(m+e+1)(m+N-e)(\frac{3}{2}-N)_{e}}.
\end{aligned}
\end{equation}
\label{tfe}
\end{thm}

\begin{thm}
    Let $m$ be a non-negative integer and $N$ be a positive integer greater than $1$. The number of lozenge tilings of the region $V'(N-1,2m+1,l)$ obtained from the hexagon of sides $2m+1,N-1,N-1,2m+1,N-1,N-1$ $($clockwise from top$)$ by deleting the $l$-th lozenge from the top on the vertical axis half a unit away from the vertical symmetry axis, is given by
    \begin{equation}
        \M(V'(N-1,2m+1,l))=\frac{(m+1)_{N-1}(m+1)_{N-1}(N-1)_{N-1}}{(2m+2)_{N-1}(2m+N)_{N-1}(N-1)!}\cdot\sigma_{2}(N,m,l)\cdot\prod_{i=1}^{N-1}\prod_{j=1}^{N-1}\prod_{k=1}^{2m+1}\frac{i+j+k-1}{i+j+k-2}
    \end{equation}
where
\begin{equation}
\begin{aligned}
    \sigma_{2}(N,m,l)=&m\cdot\sum_{e=0}^{l-1}(-1)^{e}\binom{N-1}{e}\frac{(N-2e-1)(\frac{1}{2})_{e}}{(m+e)(m+N-e-1)(\frac{3}{2}-N)_{e}}\\
    +&(m+N)\cdot\sum_{e=0}^{l-1}(-1)^{e}\binom{N-1}{e}\frac{(N-2e-1)(\frac{1}{2})_{e}}{(m+e+1)(m+N-e)(\frac{3}{2}-N)_{e}}.
\end{aligned}
\end{equation}
\label{tff}
\end{thm}

We deduce the above two results from Theorem \ref{tbb}, using also  enumeration results on certain families of regions given in \cite{FulmekKrattenthaler1}. Although one of them is a special case of a region we considered in an earlier section,
for consistency with \cite{FulmekKrattenthaler1} we use the notation from there.

For non-negative integers $m$ and $N$, let $S(N,m)$ be the pentagonal region depicted on the left in Figure \ref{ffd}; note that this region was previously denoted by $\overline{R}_{(1,\ldots,N),\emptyset}(m)$ in subsection \ref{subsection4.2}. Also, for a non-negative integer $m$ and positive integers $l$ and $N$ such that $l\leq N$, we define the region $C(N,m,l)$ as follows: on $S(N,m)$, we label the $N$ rightmost vertical lozenges by $1,\ldots,N$, from top to bottom. Then we remove the vertical lozenge labeled by $l$ and assign the weight $\frac{1}{2}$ to the remaining $N-1$ vertical lozenges (see the picture on the right in Figure \ref{ffd}). In \cite{FulmekKrattenthaler1}, Fulmek and Krattenthaler determined the number of lozenge tilings for $S(N,m)$ and the weighted count of lozenge tilings\footnote{Our formula looks somewhat different compared to the one provided in \cite{FulmekKrattenthaler1} since we made some simplifications in the latter.} of $C(N,m,l)$.

\begin{figure}
    \centering
    \includegraphics[width=.5\textwidth]{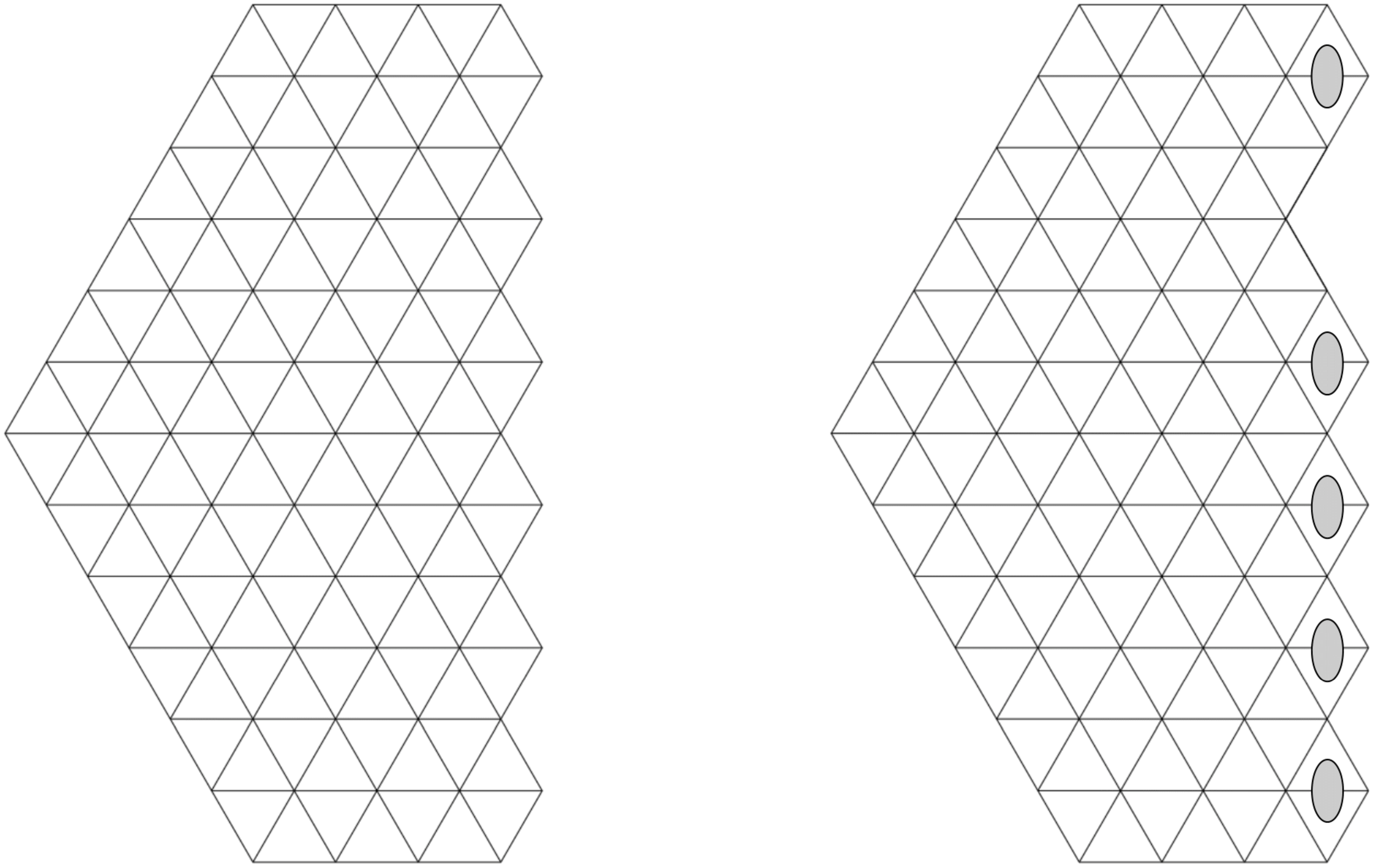}
    \caption{{\it Left.} The region $S(N,m)$ for $N=6$ and $m=3$. {\it Right.}  The region $C(N,m,l)$ for $N=6$, $m=3$ and $l=2$. Lozenges marked by ellipses are weighted by $\frac{1}{2}$, as in the previous sections.}
    \label{ffd}
\end{figure}

\begin{lemma}$((2.4)$ and $(2.6)$ in \cite{FulmekKrattenthaler1}$)$
    For any non-negative integer $m$ and positive integer $N$,
    \begin{equation}
        \M(S(N,m))=\prod_{i=1}^{N}\frac{(N+m-i+1)!(i-1)!(2m+i+1)_{i-1}}{(m+i-1)!(2N-2i+1)!}.
    \label{efi}
    \end{equation}
\label{lfg}
\end{lemma}

\begin{lemma}$((2.5)$ and $(2.7)$ in \cite{FulmekKrattenthaler1}$)$
    For any non-negative integer $m$ and positive integers $l$ and $N$ such that $l\leq N$,
    \begin{equation}
    \begin{aligned}
        \M(C(N,m,l))=\frac{(m)_{N+1}}{2^{N-1}N!}\prod_{i=1}^{\lfloor\frac{N}{2}\rfloor}\frac{(2m+2i)_{2N-4i+1}}{(2i)_{2N-4i+1}}\cdot\sum_{e=0}^{l-1}(-1)^{e}\binom{N}{e}\frac{(N-2e)(\frac{1}{2})_{e}}{(m+e)(m+N-e)(\frac{1}{2}-N)_{e}}.
    \end{aligned}
    \label{efj}
    \end{equation}
\label{lfh}
\end{lemma}

\vspace{3mm}
\begin{proof}[Proof of Theorem \ref{tfe}]
We apply Theorem \ref{tbb} to the region $V'(N,2m,l)$. To do that, we split the region using the two different zigzag lines along the vertical axis through the hole. While one of the zigzag lines splits the region into $S(N,m)$ and $C(N-1,m,l)$, the other splits it into $S(N,m-1)$ and $C(N-1,m+1,l)$ (see the two pictures in Figure \ref{ffe}). Thus, by Theorem \ref{tbb}, we have

\begin{figure}
    \centering
    \includegraphics[width=.6\textwidth]{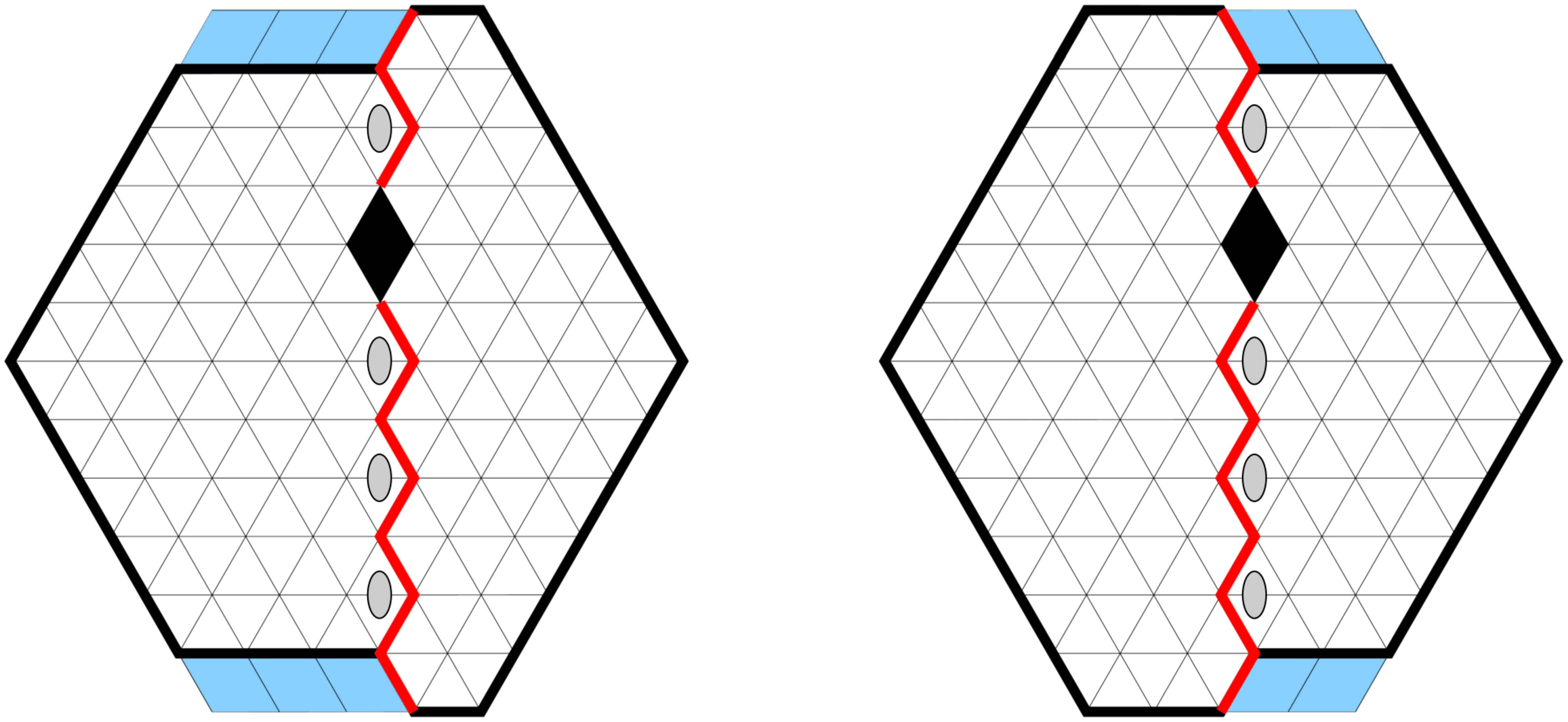}
    \caption{One of the two zigzag lines splits $V'(4,6,2)$ into the subregions $C(5,3,2)$ and $S(6,1)$ (left). The other zigzag line splits it into $C(5,2,2)$ and $S(6,2)$ (right).}
    \label{ffe}
\end{figure}

\begin{equation}
    \M(V'(N,2m,l))=2^{N-2}\Big[\M(S(N,m))\M(C(N-1,m,l))+\M(S(N,m-1))\M(C(N-1,m+1,l))\Big].
\label{efk}
\end{equation}
By combining \eqref{efi}, \eqref{efj}, \eqref{efk} and factoring out some common factors, we get the desired formula for $\M(V'(N,2m,l))$.
\end{proof}

\begin{figure}
    \centering
    \includegraphics[width=.6\textwidth]{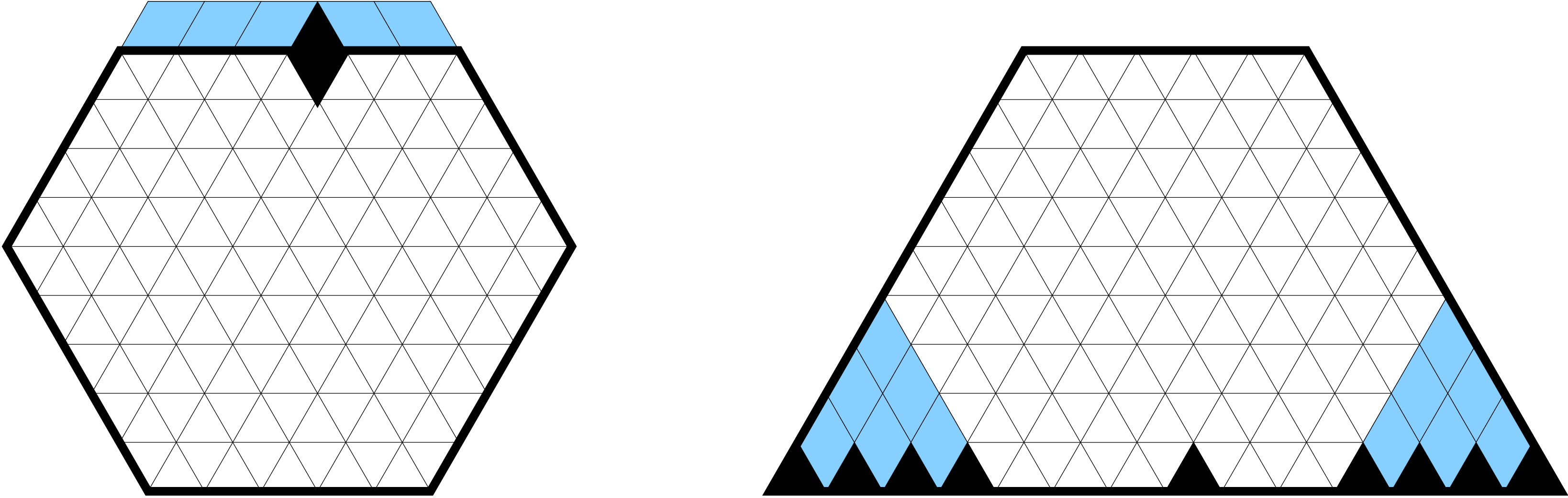}
    \caption{If we remove the forced lozenges from the region $V'(N,m,1)$ on the left (here $N=5$ and $m=5$), and also from the region on the right (which is a special case of the regions whose tilings are enumerated by Proposition 2.1 of \cite{cohn1998shape}), the two resulting regions are congruent.
}
    \label{fff}
\end{figure}

\begin{proof}[Proof of Theorem \ref{tff}]
Theorem \ref{tbb} is not applicable when $l=1$, because the region $R_0$ is not well-defined. Because of this, we consider the case $l=1$ separately. In this case, the region $V'(N-1,2m+1,1)$ can be regarded as a trapezoid region with dents on its base (see Figure \ref{fff} and the explanation in its caption). The number of lozenge tilings of such regions was given in \cite[Proposition 2.1]{cohn1998shape} and one can check that the $l=1$ specialization of our claimed formula agrees with the expression obtained using \cite[Proposition 2.1]{cohn1998shape}. Assume therefore that $l\geq2$, and apply Theorem \ref{tbb}. To do that, we again consider the two zigzag lines along the vertical symmetry axis of the lozenge hole. One of them divides the region into $S(N-2,m)$ and $C(N-1,m+1,l)$, while the other divides it into $S(N-2,m+1)$ and $C(N-1,m,l)$ (see Figure \ref{ffg}. Thus, by Theorem \ref{tbb}, we obtain

\begin{figure}
    \centering
    \includegraphics[width=.6\textwidth]{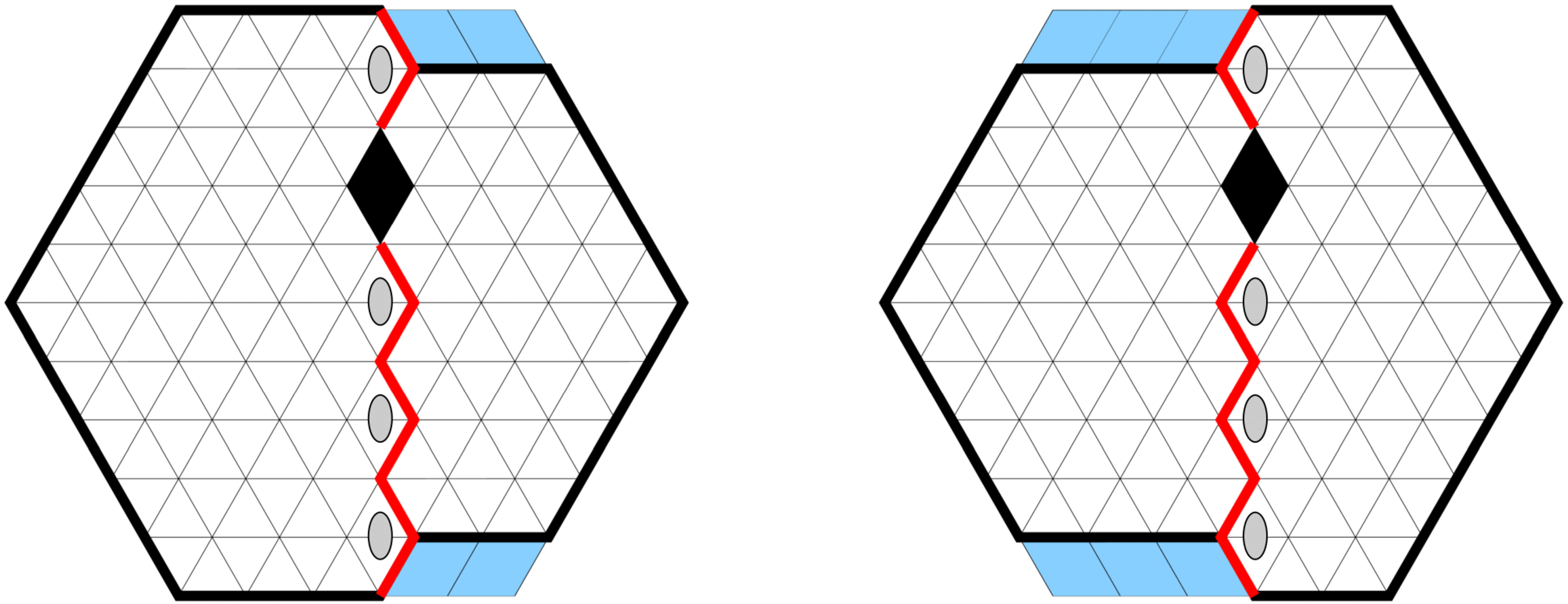}
    \caption{One of the two zigzag lines splits $V'(5,5,2)$ into the subregions $C(5,3,2)$ and $S(4,2)$ (left). The other splits it into $C(5,2,2)$ and $S(4,3)$ (right).}
    \label{ffg}
\end{figure}

\begin{equation}
\begin{aligned}
    &\M(V'(N-1,2m+1,l))\\
    =&2^{N-3}\Big[\M(S(N-2,m))\M(C(N-1,m+1,l))+\M(S(N-2,m+1))\M(C(N-1,m,l))\Big].
\end{aligned}
\label{efl}
\end{equation}
Using \eqref{efi}, \eqref{efj}, \eqref{efl} and factoring out some common factors, we get the claimed formula for $\M(V'(N+1,2m-1,l))$.
\end{proof}

\section{Aztec rectangles with collinear unit holes}

Let $m$ and $n$ be positive integers with $m \leq n$. The \textit{$m\times n$ Aztec rectangle}, denoted by $AR(m,n)$, is the region consisting of all unit squares on $\Z^{2}$ whose centers $(x,y)$ satisfy the inequalities $|x|+|y|\leq n$ and $x+y\geq n-2m$. In this section, we rotate the region $AR(m,n)$ by $45^{\circ}$ in counterclockwise direction (see the pictures in Figure \ref{fga}; at this point, ignore the black unit squares in the pictures).

Unless $m=n$ (when it becomes the Aztec diamond $AD_n$), the Aztec rectangle $AR(m,n)$ does not have any domino tilings: it is not hard to check that if we color the square grid in chessboard fashion, $AR(m,n)$ has $n-m$ more unit squares of one color than the other. This can be fixed as follows.
Label the bottom-most unit squares in $AR(m,n)$ by $1,\ldots,n$ from left to right. For any subset $S\subseteq[n]\coloneqq\{1,\ldots,n\}$ of size $m$, we delete the unit squares labeled with elements of $[n]\setminus S$ from $AR(m,n)$, and denote the resulting region by $AR(m,n;S)$ (see the picture on the left in Figure~\ref{fga} for an illustration). Then the region $AR(m,n;S)$ always admits domino tilings; the number of them is given by (see \cite{CiucuMatchingFactorization}\cite{elkies1992alternating}\cite{mills1983alternating})
\begin{equation}\label{ega}
  \M(AR(m,n;S)) = 2^{m(m+1)/2} \prod_{1 \leq i<j \leq m} \frac{s_j-s_i}{j-i},
\end{equation}
where $S = \{s_1,\ldots,s_m\}$, with $s_1<\cdots<s_m$.
\begin{figure}
    \centering
    \includegraphics[width=.8\textwidth]{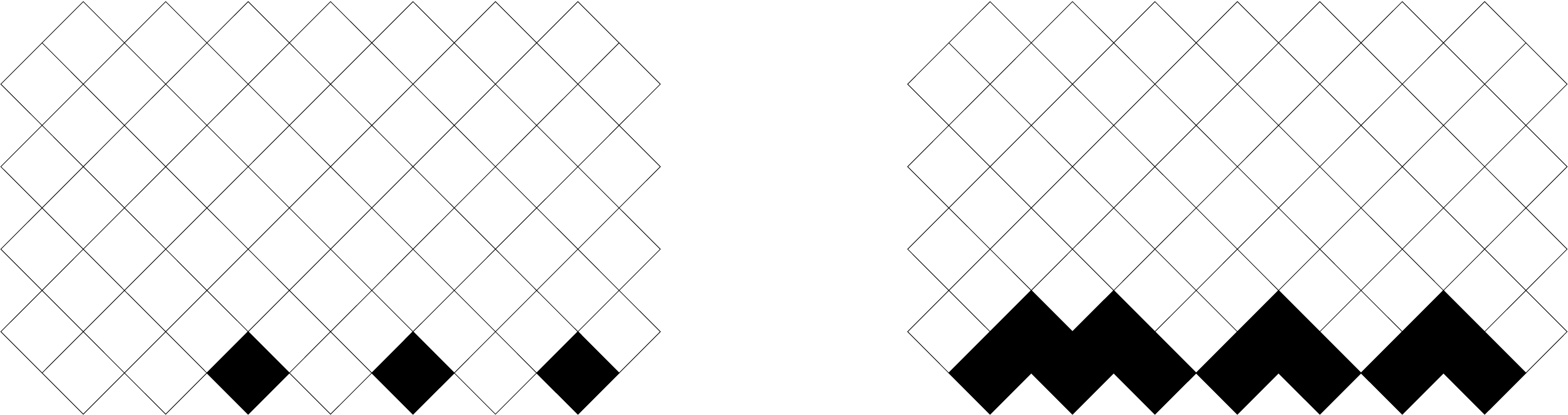}
    \caption{The regions $AR(4,7;\{1,2,4,6\})$ (left) and $\overline{AR}(4,7;\{2,3,5,7\})$ (right). Both regions are obtained from $AR(4,7)$ by deleting some unit squares.}
    \label{fga}
\end{figure}

Next, for any subset $T=\{t_{1},\ldots,t_{m}\}\subseteq\{n+1\}$ of size $m$ (again, with elements of $T$  written in increasing order), consider the region $\overline{AR}(m,n;T)$ defined as follows. From $AR(m,n)$, remove the $n$ unit squares at the bottom. Note that the resulting region has $n+1$ unit squares on the (new) bottom. Label these unit squares by $1,\ldots,n+1$, from left to right. Finally, delete the unit squares labeled by elements of $T$, and denote the resulting region by $\overline{AR}(m,n;T)$. The number of domino tilings for this region was studied in \cite{elkies1992alternating} and \cite{gessel1999enumeration}, and is given by 

\begin{equation}\label{egb}
  \M(\overline{AR}(m,n;T)) = 2^{m(m-1)/2} \prod_{1 \leq i<j \leq m} \frac{t_j-t_i}{j-i},
\end{equation}

In \cite{CiucuMatchingFactorization}, the second author proved that if an arbitrary collection of unit squares is deleted from the horizontal symmetry axis of an Aztec rectangle, the number of domino tilings of the resulting region is given by a simple product formula. More precisely, let $m$ and $N$ be positive integers with $2m\leq N$. Then the number of domino tilings of a $2m\times N$ Aztec rectangle, where all the unit squares on the central horizontal row, except the $t_{1}$-st, the $t_{2}$-nd,$\ldots$, and the $t_{2m}$-th unit square have been removed
equals
\begin{equation}\label{egc}
    \displaystyle\frac{2^{m^{2}+2m}}{\displaystyle\prod_{i=1}^{m}(i-1)!^{2}}\prod_{1\leq i<j\leq m}(t_{2j}-t_{2i})\prod_{1\leq i<j\leq m}(t_{2j-1}-t_{2i-1}).
\end{equation}

This was extended by Krattenthaler \cite{KrattenthalerAztecRectangle}, who considered the following more general question: if we delete collinear unit squares along an {\it arbitrary} horizontal line from an Aztec rectangle, what is the number of domino tilings of the resulting regions\footnote{In fact, Krattenthaler expressed this in terms of perfect matchings of the planar dual graphs. In this paper, in order to be consistent with the previous sections, we state everything in terms of domino tilings.}? Using certain Schur function identities, he answered this question by finding some explicit formulas.

As equation \eqref{egc} shows,
the special case when the horizontal axis containing the deleted unit squares is the symmetry axis, this formula becomes a simple product formula. Furthermore, the special case of Krattenthaler's formula when the line containing the deleted squares is just below the symmetry axis also turns out to be a simple product formula.

\begin{figure}
    \centering
    \includegraphics[width=.6\textwidth]{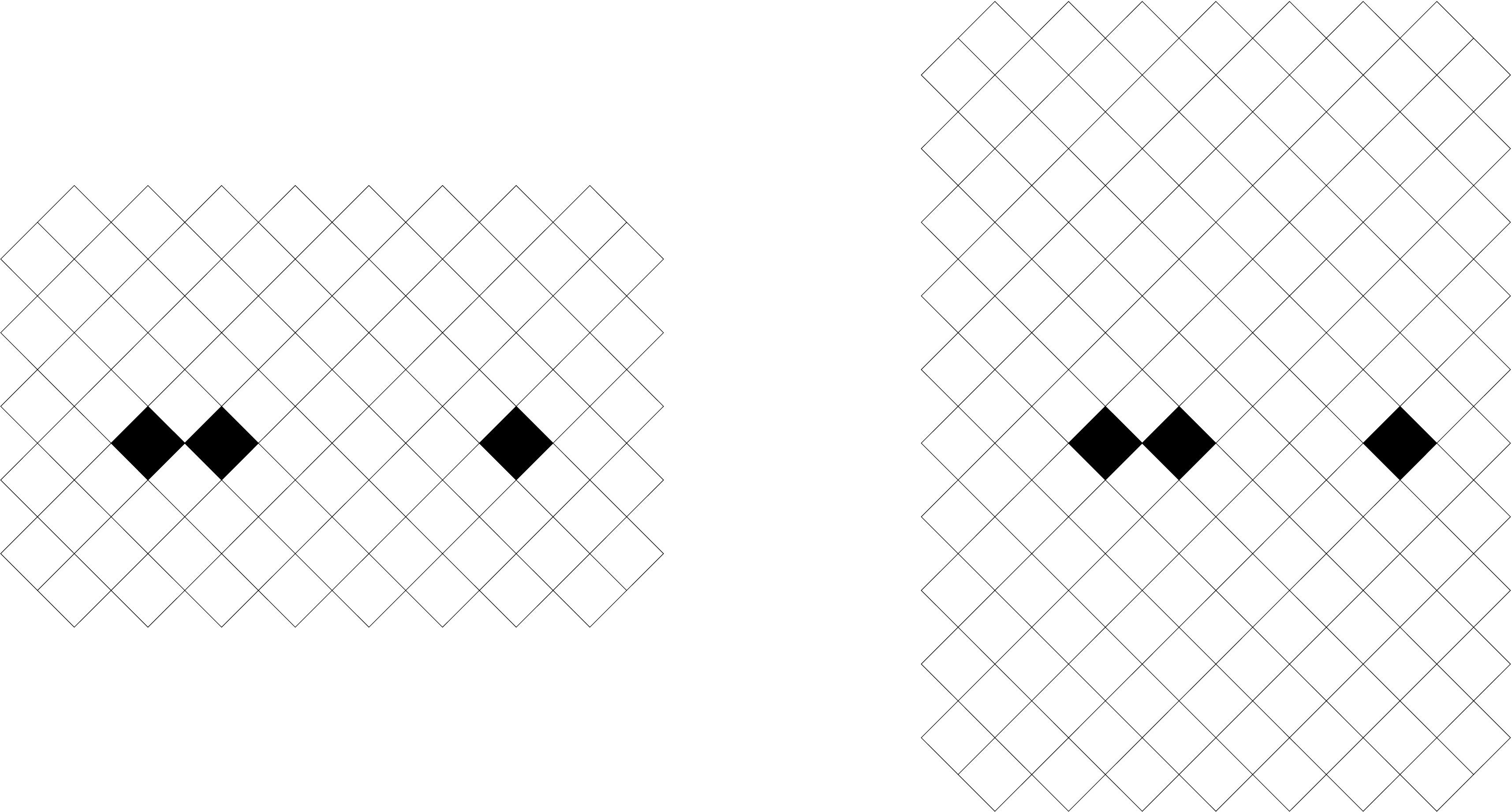}
    \caption{A picture that illustrates Theorem \ref{tgc} with $m=2$, $N=8$, and $(t_{1},t_{2},t_{3},t_{4},t_{5})=(1,4,5,6,8)$ (left) and another picture that illustrates Theorem \ref{tgd} with $m=5$, $N=7$, and $(t_{1},t_{2},t_{3},t_{4},t_{5})=(1,2,5,6,8)$ (left). Later when we prove Theorem \ref{tge}, the picture on the left is denoted by $R_{1}(2m+1,N;(t_{1},\ldots,t_{2m+1}))=R_{1}(5,8;(1,4,5,6,8))$.}
    \label{fgc}
\end{figure}

\begin{thm}[Theorem 9 in \cite{KrattenthalerAztecRectangle}]
\label{tgc}
Let $m$ and $N$ be positive integers with $2m+1\leq N$. Then the number of domino tilings of a $(2m+1)\times N$ Aztec rectangle, where all the unit squares on the horizontal row that is just below the central row, except for the $t_{1}$-st, the $t_{2}$-nd,$\ldots$, and the $t_{2m+1}$-th unit square, have been removed $($see the left picture in Figure $\ref{fgc})$, equals
\begin{equation}\label{ege}
    \displaystyle\frac{2^{m^{2}+3m+1}}{\displaystyle\prod_{i=1}^{m}(i-1)!\prod_{i=1}^{m+1}(i-1)!}\prod_{1\leq i<j\leq m}(t_{2j}-t_{2i})\prod_{1\leq i<j\leq m+1}(t_{2j-1}-t_{2i-1}).
\end{equation}
\end{thm}

\begin{thm}[Theorem 10 in \cite{KrattenthalerAztecRectangle}]
\label{tgd}
Let $m$ and $N$ be positive integers with $2m\geq N$. Then the number of domino tilings of a $2m\times N$ Aztec rectangle, where all the unit squares on the horizontal row that is just below the central row, except for the $t_{1}$-st, the $t_{2}$-nd,$\ldots$, and the $t_{2N-2m+1}$-th unit square, have been removed $($see the right picture in Figure $\ref{fgc})$, equals
\begin{equation}\label{egf}
    \displaystyle2^{m^{2}-m+N}\frac{\displaystyle\prod_{i=m+1}^{N+1}(i-1)!\prod_{i=m+2}^{N+1}(i-1)!}{\displaystyle\prod_{i=1}^{2N-2m+1}(t_{i}-1)!(N+1-t_{i})!}\times \prod_{1\leq i<j\leq N-m}(t_{2j}-t_{2i})\prod_{1\leq i<j\leq N-m+1}(t_{2j-1}-t_{2i-1}).
\end{equation}
\end{thm}

Krattenthaler's proofs of Theorems \ref{tgc}--\ref{tgd} in \cite{KrattenthalerAztecRectangle} were based on certain Schur function identities. In this section, we give a very simple proof of these theorems using Theorem \ref{tba}. 

\begin{proof}[Proof of Theorems \ref{tgc}--\ref{tgd}]

%
%


%
\begin{figure}
    \centering
    \includegraphics[width=.6\textwidth]{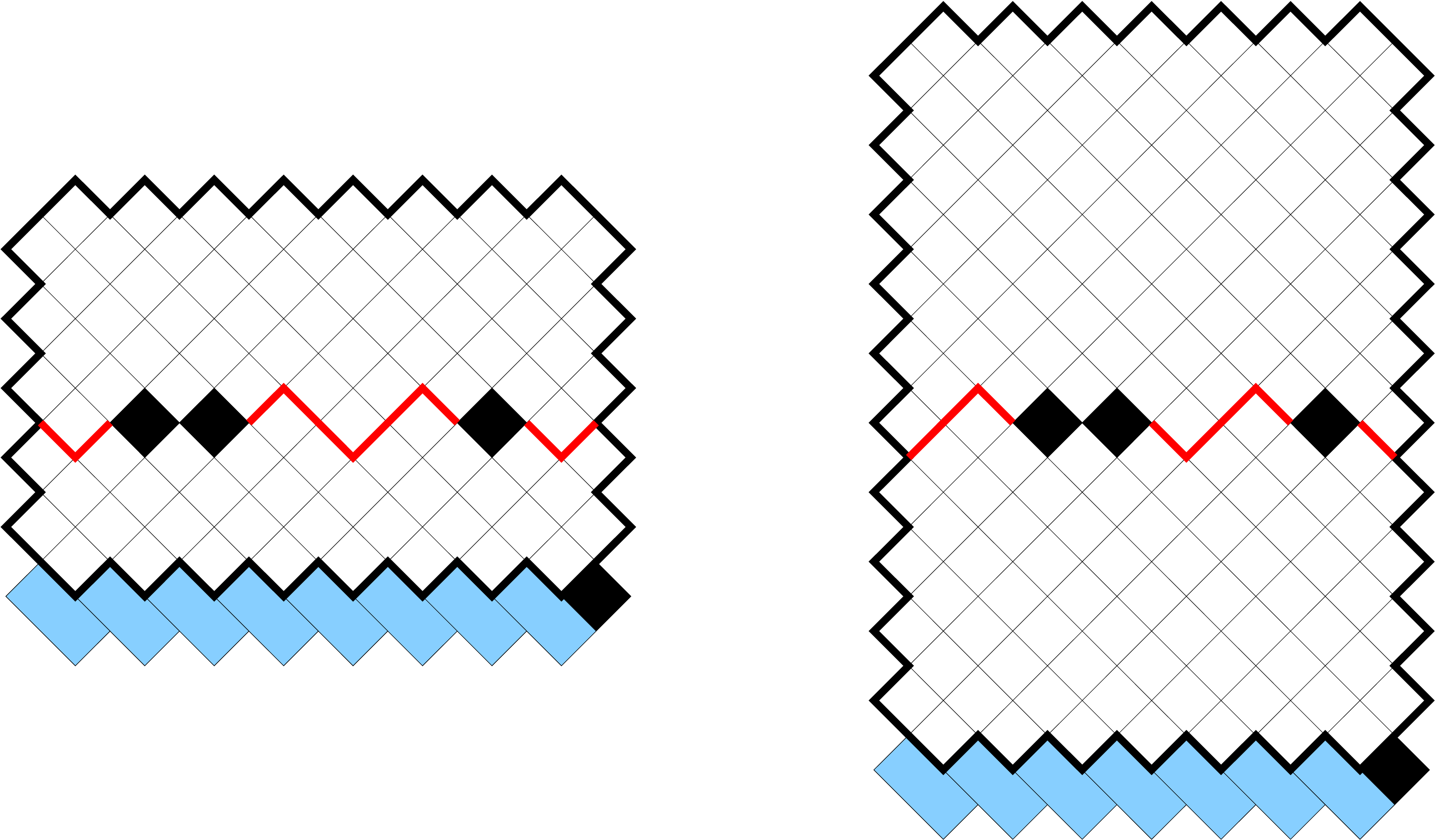}
    \caption{Application of Theorem \ref{tba}(b) to the two regions in Figure \ref{fgc}.}
    \label{fge}
\end{figure}

We first prove Theorem \ref{tgc}. To do that, we add one layer at the bottom of the region by adding $2N+1$ unit squares that form a zigzag path, and then we delete the rightmost newly added unit square (see the left picture in Figure \ref{fge}). Then, the dual graph of the resulting region is a symmetric graph with a vertex removed from its boundary, to which Theorem \ref{tba}(b) can be applied. If we apply Theorem \ref{tba}(b) to this dual graph, take the dual again, and delete the forced dominos, we get that the number of domino tilings of the region in Theorem~\ref{tgc} can be expressed as follows:
\begin{equation}\label{egh}
    2^{m}\M(AR(m+1,N;\{t_{2i-1}\}_{i\in[m+1]}))\M(AR(m,N;\{t_{2i}\}_{i\in[m]})).
\end{equation}
Combining \eqref{ega} and \eqref{egh} one obtains 
\eqref{ege}, which completes the proof of Theorem\ref{tgc}.

The same approach works for proving Theorem \ref{tgd}.
We add one layer at the bottom of the Aztec rectangle by including $2N+1$ unit squares in the same way as in the proof of Theorem \ref{tgc}, and then remove the rightmost of these (see the picture on the right in Figure \ref{fge}). We can again apply Theorem \ref{tba}(b), and after removing the forced forced dominos of the resulting regions, we obtain that the number sought in Theorem \ref{tgd} equals
\begin{equation}\label{egi}
    2^{N-m}\M(\overline{AR}(m,N;[N+1]\setminus{
    \{t_{2i-1}\}}_{i\in[N-m+1]}))\M(\overline{AR}(m+1,N;[N+1]\setminus\{t_{2i}\}_{i\in[N-m]})).
\end{equation}
Combining \eqref{egb} and \eqref{egi} one gets \eqref{egf}, which proves Theorem \ref{tgd}.
\end{proof}

We finish this section by presenting one more new proof, namely of the following theorem, which is another special case of Krattenthaler's general result (Theorem 11 in \cite{KrattenthalerAztecRectangle}). 

\begin{thm}[The case $d=2$ in Theorem 11 in \cite{KrattenthalerAztecRectangle}]
\label{tge}
Let $m$ and $N$ be positive integers with $2m+2\leq N$. Then the number of domino tilings of a $(2m+2)\times N$ Aztec rectangle, where all the unit squares on the horizontal row that is by $2$ below the central row, except for the $t_{1}$-st, the $t_{2}$-nd,$\ldots$, and the $t_{2m+2}$-th unit square, have been removed, equals
\begin{equation}\label{egj}
    \displaystyle\frac{2^{m^{2}+4m+3}}{\displaystyle\prod_{i=1}^{m}(i-1)!\prod_{i=1}^{m+2}(i-1)!}\cdot\Bigg[\prod_{1\leq i<j\leq m+1}(t_{2j}-t_{2i})(t_{2j-1}-t_{2i-1})\Bigg]\cdot\Bigg(\sum_{i=1}^{2m+2}(-1)^{i}t_{i}\Bigg).
\end{equation}
\end{thm}
Note that our formula \eqref{egj} has a simpler form than Krattenthaler's original formula (see (4.5) in \cite{KrattenthalerAztecRectangle}; it is not so straightforward to show that the second line in the $d=2$ specialization of that formula equals $\sum_{i=1}^{2m+2}(-1)^{i}t_{i}$). Krattenthaler's proof of the above theorem used some Schur function identities. By contrast, we prove this theorem using Kuo's graphical condensation method. Although we stated Kuo's condensation theorem in Section 2 (see Theorem \ref{tbc}), we state here a different version of it, because this is the one we need in our proof of Theorem \ref{tge}.

\begin{thm}[Theorem 2.1 in \cite{kuo2004applications}]
\label{tgf}
Let $G=(V_1,V_2,E)$ be a plane bipartite graph in which $|V_1|=|V_2|$. Let vertices $a,b,c,$ and $d$ appear in a cyclic order on the same face of $G$. If $a,c\in V_1$ and $b,d\in V_2$, then
\begin{equation*}
    \M(G)\M(G\setminus\{a,b,c,d\})=\M(G\setminus\{a,c\})\M(G\setminus\{b,d\})+\M(G\setminus\{a,d\})\M(G\setminus\{b,c\}).
\end{equation*}
\end{thm}
Let $R_{0}(2m,N;\{t_{1},\ldots,t_{2m}\})$, $R_{1}(2m+1,N;\{t_{1},\ldots,t_{2m+1}\})$, and $R_{2}(2m+2,N;\{t_{1},\ldots,t_{2m+2}\})$ be the regions described 
just before equation \eqref{egc}, in Theorem
\ref{tgc}, and in Theorem \ref{tge}, respectively. Under this notation, we want to show that the number of domino tilings of $R_{2}(2m+2,N;\{t_{1},\ldots,t_{2m+2}\})$ is given by \eqref{egj}. Note that 
the region $R_{2}(2m+2,N;\{t_{1},\ldots,t_{2m+2}\})$
has $N-(2m+2)$ unit square holes in it. This is important because we will use an induction on this quantity.

\begin{proof}[Proof of Theorem \ref{tge}]

\begin{figure}
    \centering
    \includegraphics[width=.8\textwidth]{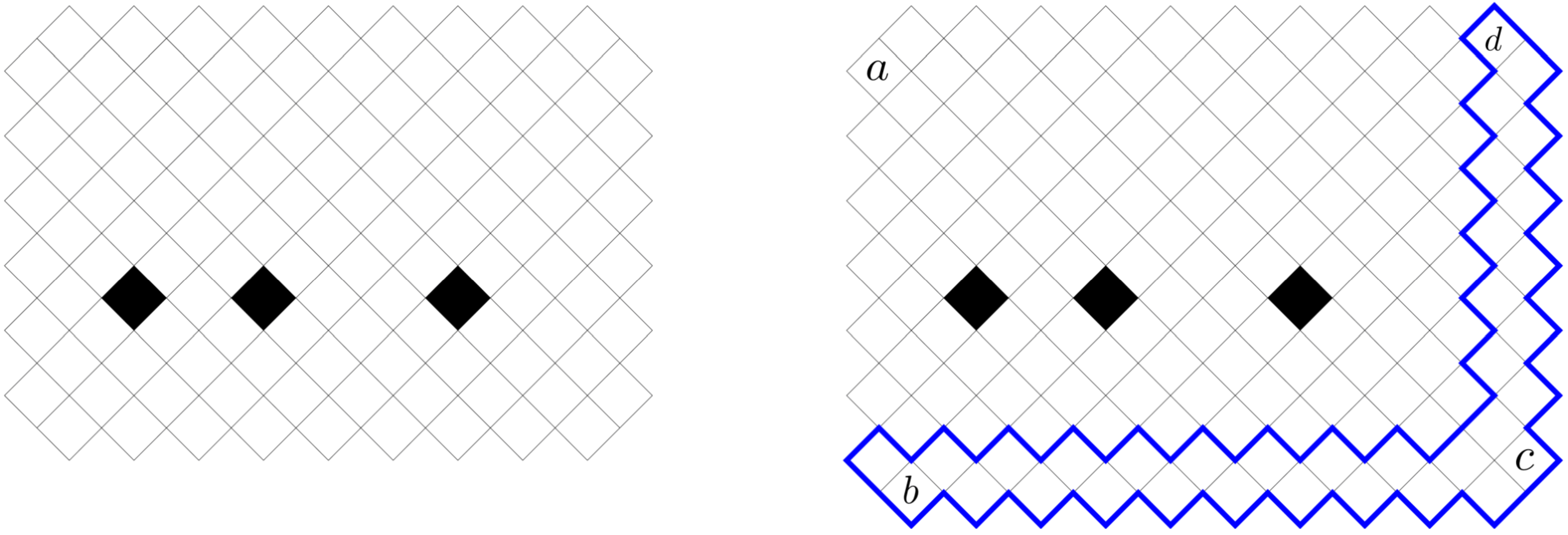}
    \caption{$R_{2}(2m+2,N,\{t_{1},\ldots,t_{2m+2}\})=R_{2}(6,9;\{1,3,5,6,8,9\})$ (left) and $R_{1}(2m+3,N+1;\{t_{1},\ldots,t_{2m+2},N+1\})=R_{1}(7,10,\{1,3,5,6,8,9,10\})$ (right). The region on the right is obtained from the one on the left by adding a layer along the bottom and right sides. The extended part is marked by thick blue lines; four unit squares $a,b,c,$ and $d$ are also marked.}
    \label{fgf}
\end{figure}

Note that it suffices to prove the statement of Theorem \ref{tge} for the regions $R_{2}(2m+2,N;\{t_{1},\ldots,t_{2m+2}\})$ with $t_{1}=1$ and $t_{2m+2}=N$. Indeed, if $t_{1}\neq1$ or $t_{2m+2}\neq N$, then by deleting some forced dominos, one obtains a smaller region with fewer holes in it.

We first construct a recurrence relation using Kuo's graphical condensation (Theorem \ref{tgf}). Consider a region $R_{2}(2m+2,N;\{t_{1},\ldots,t_{2m+2}\})$ with $t_{1}=1$ and $t_{2m+2}=N$. Add a layer on the bottom and right sides of this region, as described in Figure \ref{fgf}. On this extended region, we choose four unit squares $a,b,c,$ and $d$ as shown in the picture on the right in Figure \ref{fgf}. The dual graph of the extended region, together with its four vertices corresponding to the unit squares labeled by $a,b,c,$ and $d$, satisfies the conditions required in Theorem \ref{tgf}. If we apply Kuo's graphical condensation on this dual graph, take the dual again, and remove the forced dominos in the resulting regions, we obtain the following recurrence relation (see the six regions in Figure \ref{fgg}):
\begin{equation}\label{egk}
\begin{aligned}
    &\M(R_{1}(2m+3,N+1;\{t_{1},\ldots,t_{2m+2},N+1\}))\M(R_{1}(2m+1,N-1;\{t_{2}-1,\ldots,t_{2m+2}-1\}))\\
    =&\M(R_{2}(2m+2,N;\{t_{2}-1,\ldots,t_{2m+2}-1,N\}))\M(R_{0}(2m+2,N;\{t_{1},\ldots,t_{2m+2}\}))\\
    +&\M(R_{2}(2m+2,N;\{t_{1},\ldots,t_{2m+2}\}))\M(R_{0}(2m+2,N;\{t_{2}-1,\ldots,t_{2m+2}-1,N\})).
\end{aligned}
\end{equation}

\begin{figure}
    \centering
    \includegraphics[width=1\textwidth]{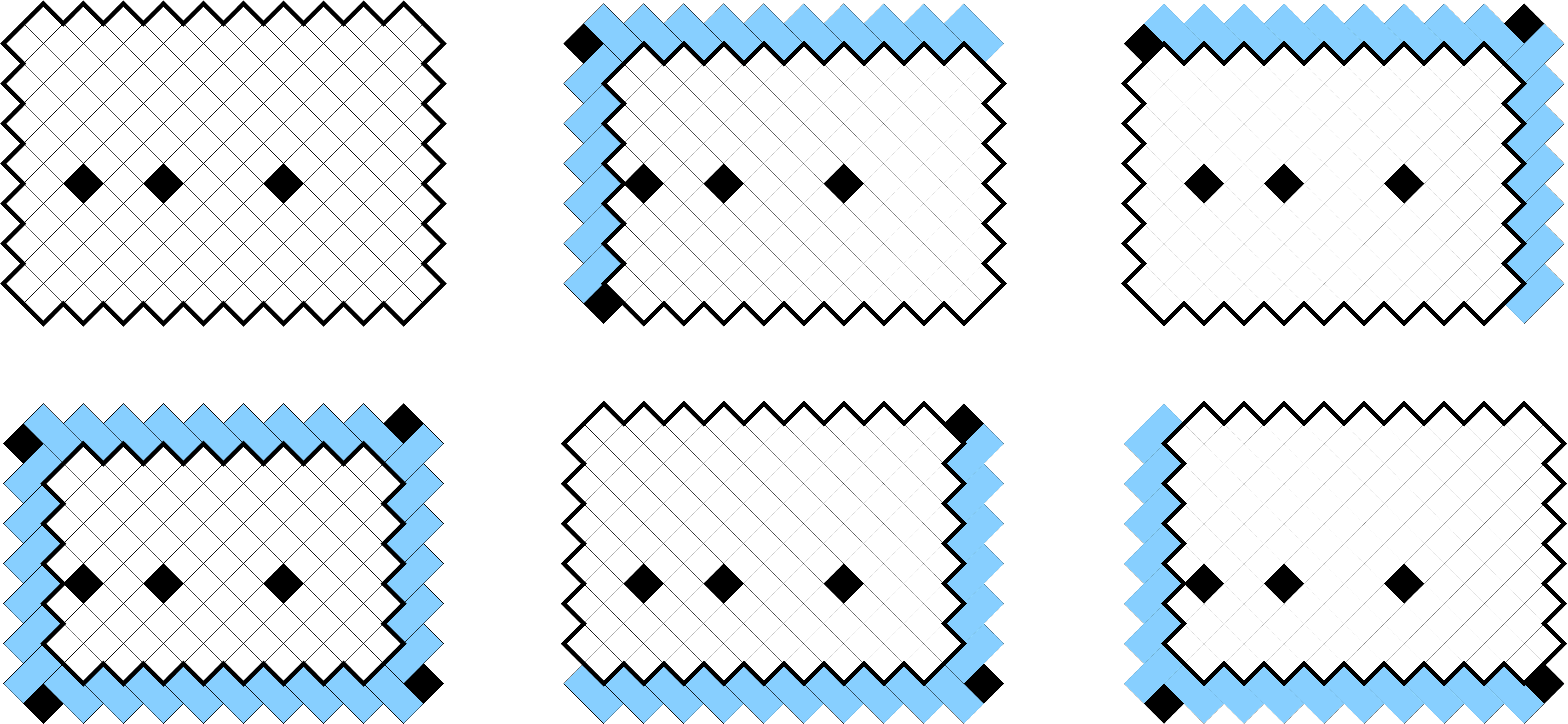}
    \caption{The six regions appearing in the recurrence relation we obtain by applying Kuo's graphical condensation to the region in the picture on the right in Figure \ref{fgf}.}
    \label{fgg}
\end{figure}

One can check that the expressions on the right hand sides of equations \eqref{egc}, \eqref{ege} and~\eqref{egj} satisfy the above recurrence relation \eqref{egk}. Using this, we will give an inductive proof of Theorem \ref{tge}. We start the proof with an induction on the number of holes in the region $R_{2}(2m+2,N;\{t_{1},\ldots,t_{2m+2}\})$, which is $N-(2m+2)$. Let us call this the \textit{outer induction}.

When $N-(2m+2)=0$, $\{t_{1},\ldots,t_{2m+2}\}=[2m+2]$ (where $[n]:=\{1,2,\dotsc,n\}$) and the region $R_{2}(2m+2,N;\{t_{1},\ldots,t_{2m+2}\})=R_{2}(2m+2,2m+2;[2m+2])$ becomes the Aztec diamond of order $2m+2$. One can check that if we set $N=2m+2$ and $\{t_{1},\ldots,t_{2m+2}\}=[2m+2]$ in the expression on the right hand side of \eqref{egj}, we get $2^{\frac{(2m+2)(2m+3)}{2}}$, which is the number of domino tilings of the Aztec diamond of order $2m+2$. Thus, the theorem is verified for the case when $N-(2m+2)=0$.

Suppose Theorem \ref{tge} holds whenever $N-(2m+2)<k$ for some positive integer $k$. Under this assumption, we need to show that the theorem still holds when $N-(2m+2)=k$. To show this, we need another induction. This time we induct on the smallest element of $[N]\setminus\{t_{1},\ldots,t_{2m+2}\}$. We call this the \textit{inner induction}. If the smallest element is $1$, then the unit square labeled $1$ was removed. But this causes several forced dominos on the left side of the region. Due to this, in this case, the region $R_{2}(2m+2,N;\{t_{1},\ldots,t_{2m+2}\})$ can be identified with $R_{2}(2m+2,N-1;\{t_{1}-1,\ldots,t_{2m+2}-1\})$, which is the region whose number of domino tilings is given by \eqref{egj} by the outer induction hypothesis. Thus we have checked the case when the smallest element of $[N]\setminus\{t_{1},\ldots,t_{2m+2}\}$ is $1$.

Now suppose that Theorem \ref{tge} holds when $N-(2m+2)=k$ and the smallest element of $[N]\setminus\{t_{1},\ldots,t_{2m+2}\}$ is less than or equal to a positive integer $l$. Under this assumption, we need to show that the theorem still holds when $N-(2m+2)=k$ and the smallest element of $[N]\setminus\{t_{1},\ldots,t_{2m+2}\}$ is $l+1$. Consider the recurrence relation \eqref{egk}, and focus on the two terms containing an $R_2$-region. It is not difficult to check that when $t_{1}=1$ (which is the case, as we can now assume that the smallest element of $[N]\setminus\{t_{1},\ldots,t_{2m+2}\}$ is $>1$), the smallest element of $N\setminus\{t_{1},\ldots,t_{2m+2}\}$ is greater than that of $N\setminus\{t_{2}-1,\ldots,t_{2m+2}-1,N\}$, and it exceeds the latter by exactly one. Since we already checked that \eqref{egk} is satisfied by \eqref{egc}, \eqref{ege}, and \eqref{egj}, we can conclude that Theorem \ref{tge} also holds when $N-(2m+2)=k$ and the smallest element of $N\setminus\{t_{1},\ldots,t_{2m+2}\}$ is $l+1$. Thus, the inner induction is checked, and it follows that Theorem \ref{tge} always holds whenever $N-(2m+2)=k$. Therefore the induction step is verified for the outer induction, and it follows that Theorem \ref{tge} holds whenever $N-(2m+2)$ is a non-negative integer. This completes the proof.
\end{proof}

\bigskip
{\bf Acknowledgments.} Part of this research was performed while the first author was visiting the Institute for Pure and Applied Mathematics (IPAM), which is supported by the National Science Foundation (Grant No. DMS-1925919). He would like to thank the program organizers and IPAM staff members for their hospitality during his stay in Los Angeles. The second author thanks the Simons Foundation for supporting in part his work by Simons Foundation Collaboration Grant 710477.

\bibliography{bibliography}{}
\bibliographystyle{abbrv}

\end{document}